\documentclass[a4paper,12pt,reqno,oneside]{amsart}

\usepackage{setspace} 
\usepackage{typearea} 

\usepackage{amscd,amsmath,amssymb,verbatim}

\usepackage
{graphicx} \setkeys{Gin}{width=\linewidth}

\usepackage{ifthen}

\usepackage{xr-hyper}
\usepackage{cite}
\usepackage[colorlinks,backref,final,bookmarksnumbered,bookmarks]{hyperref}
\usepackage[backrefs]{amsrefs}
\newcommand{\arxiv}[2][]{\ifthenelse{\equal{#1}{}}
{\href{http://arxiv.org/abs/#2}{\tt arXiv:#2}}
{\href{http://arxiv.org/abs/math/#2}{\tt arXiv:math.#1/#2}}}

\makeatletter
\renewcommand\subsection{\@startsection
{subsection}{2}{0cm} 
{-\baselineskip}     
{0.5\baselineskip}   
{\sffamily}} 
\makeatother

\theoremstyle{plain}
\newtheorem{theorem}{Theorem}[section]
\newtheorem{lemma}[theorem]{Lemma}
\newtheorem{corollary}[theorem]{Corollary}
\newtheorem{proposition}[theorem]{Proposition}
\newtheorem{problem}[theorem]{Problem}
\newtheorem{addendum}[theorem]{Addendum}
\newtheorem{conjecture}[theorem]{Conjecture}

\theoremstyle{definition}
\newtheorem{example}[theorem]{Example}

\newtheoremstyle{remark}
{}{}{}{}{\itshape}{}{ }{\thmname{#1}\thmnumber{ \itshape #2.}}
\theoremstyle{remark}
\newtheorem{remark}[theorem]{Remark}

\newtheoremstyle{concise}
{}{}{}{}{\bfseries}{}{ }{\thmnumber{#2.}\thmnote{ #3.}}
\theoremstyle{concise}
\newtheorem{definition}[theorem]{}

\newenvironment{roster}[1][0]
{

\begin{enumerate}\setcounter{enumi}{#1}}
{\end{enumerate}}

 \def\R{\mathbb{R}} \def\Z{\mathbb{Z}} \def\C{\mathbb{C}}
\def\Ham{\mathbb{H}}
\def\x{\times} \def\but{\setminus} \def\emb{\hookrightarrow} 
\def\eps{\varepsilon} \def\phi{\varphi} \def\emptyset{\varnothing}
\def\xr#1{\xrightarrow{#1}}  \renewcommand{\:}{\colon}
\def\imp{$\Rightarrow$}  
\DeclareMathOperator{\st}{st} \DeclareMathOperator{\lk}{lk}
 \DeclareMathOperator{\id}{id}
\DeclareMathOperator{\Ext}{Ext}
 \def\theta{\vartheta}
\def\mono{\rightarrowtail} \def\epi{\twoheadrightarrow}

\def\cel{{}^\lfloor{}} \def\cer{{}^\rfloor{}}
\def\fll{{}_\lceil{}} \def\flr{{}_\rceil{}}

\def\bydef{\mathrel{\mathop:}=}

\begin{document}

\title{Combinatorics of embeddings}
\author{Sergey A. Melikhov}
\dedicatory{To the memory of my father Alexandr Pavlovich Melikhov}
\address{Steklov Mathematical Institute of the Russian Academy of Sciences,
ul.\ Gubkina 8, Moscow, 119991 Russia}
\email{melikhov@mi.ras.ru}
\thanks{Supported by Russian Foundation for Basic Research Grant No.\ 11-01-00822}

\begin{abstract}
We offer the following explanation of the statement of the Kuratowski
graph planarity criterion and of $6/7$ of the statement of
the Robertson--Seymour--Thomas intrinsic linking criterion.
Let us call a cell complex {\it dichotomial} if to every nonempty cell there
corresponds a unique nonempty cell with the complementary set of vertices.
Then every dichotomial cell complex is PL homeomorphic to a sphere; there exist
precisely two $3$-dimensional dichotomial cell complexes, and their $1$-skeleta
are $K_5$ and $K_{3,3}$; and precisely six $4$-dimensional ones, and their
$1$-skeleta all but one graphs of the Petersen family.

In higher dimensions $n\ge 3$, we observe that in order to characterize those
compact $n$-polyhedra that embed in $\R^{2n}$ in terms of finitely many
``prohibited minors'', it suffices to establish finiteness of the list of all
$(n-1)$-connected $n$-dimensional finite cell complexes that do not embed in
$\R^{2n}$ yet all their proper subcomplexes and proper cell-like combinatorial
quotients embed there.
Our main result is that this list contains the $n$-skeleta of $(2n+1)$-dimensional
dichotomial cell complexes.
The $2$-skeleta of $5$-dimensional dichotomial cell complexes include (apart from
the three joins of the $i$-skeleta of $(2i+2)$-simplices) at least ten
non-simplicial complexes.
\end{abstract}

\maketitle

\section{Introduction}\label{intro}

\begin{definition}[Reading guide]
This introduction attempts to motivate the notions we eventually arrive at.
The fast reader may want to first look at the Main Theorem and its corollaries in
\S\ref{dichotomial spheres}; these involve only the (very short) definitions of
linkless and knotless embeddings (in \S\ref{graphs}), of a cell complex (in
\S\ref{notation}) and of an $h$-minor (in \S\ref{h-minors}).
However most of the examples and remarks in \S\ref{section:main} do depend on much
of the preceding development.

The constructively minded reader might be best guided by \S\ref{algorithmic},
which explains how non-algorithmic topological notions such as PL spheres and
contractible polyhedra can be eliminated from our results and conjectures.
\end{definition}

\begin{definition}[Conventions] All posets, and in particular simplicial complexes,
shall be implicitly assumed to be finite.
By a {\it graph} we mean a $1$-dimensional simplicial complex (so no loops or
multiple edges).
By a {\it circuit} in a graph we mean a connected subgraph with all vertices of
degree $2$.
Following the terminology of PL topology (as opposed to that of convex geometry),
by a {\it polyhedron} we mean a space triangulable by a simplicial complex, and
moreover endowed with a PL structure, i.e.\ a family of compatible triangulations
(see e.g.\ \cite{Hu}); by our convention above, all polyhedra are compact.
The polyhedron triangulated by a simplicial complex $K$ is denoted $|K|$.
All {\it maps} between polyhedra shall be assumed piecewise linear, unless stated
otherwise.
Two embeddings $f$, $g$ of a polyhedron in a sphere are considered {\it equivalent}
if they are related by an isotopy $h_t$ of the sphere (i.e.\ $h_0=\id$ and $h_1f=g$).
\end{definition}

\subsection{Graphs}\label{graphs}

The complete graph $K_5$ and the complete bipartite graph $K_{3,3}$
(also known as the utilities graph) are shown in Fig.\ 1.
They can be viewed as the $1$-skeleton $(\Delta^4)^{(1)}$ of the $4$-simplex and
the join $(\Delta^2)^{(0)}*(\Delta^2)^{(0)}$ of two copies of the three-point set.

\begin{theorem}[Kuratowski, 1930] A graph $G$ contains no subgraph that is
a subdivision of $K_5$ or $K_{3,3}$ iff $|G|$ embeds in $S^2$.
\end{theorem}

Known proofs of the `only if' part involve exhaustion of cases.
A relatively short proof was given by Yu.\ Makarychev \cite{Mak} and further
developed in \cite{Sk-A1}.
An interesting configuration space approach was suggested by Sarkaria \cite{Sa3}
(but beware of an error, pointed out in \cite{Sk-M}).
Given the considerable difficulty of the result, it is astonishing that besides
Kuratowski's own proof \cite{Ku} (announced in 1929), there were independent
contemporary proofs: by L. S. Pontryagin (unpublished%
\footnote{Pontryagin's autobiography dates it to the 1926/27 academic year, and
mentions that it corrected a previous result by Kuratowski.
Did Pontryagin hesitate to publish because he wanted to understand where
$K_5$ and $K_{3,3}$ come from?}, but acknowledged
in Kuratowski's original paper \cite{Ku}), and by O. Frink and P. A. Smith,
announced in 1930 \cite{FS}, \cite{Wh}.

A useful reformulation of Kuratowski's theorem was suggested by K. Wagner
\cite{Wa}.
The following non-standard definition is equivalent to the standard one (see
Proposition \ref{nevo} below).
We call a graph $H$ a {\it minor} of a graph $G$, if $H$ is obtained from
a subgraph $F$ of $G$ by a sequence of edge contractions, where an edge
contraction as a simplicial map $f$ that shrinks one edge $\{v_1,v_2\}$ onto
a vertex, provided that $\lk\{v_1\}\cap\lk\{v_2\}=\emptyset$.
The latter condition is equivalent to saying that all point-inverses of
$|f|$ are points, except for one point inverse, which is an edge.

It is easy to see that $S^2$ modulo an arc is homeomorphic to $S^2$ (cf.\ Lemma
\ref{trivial}); hence if $|G|$ embeds in $S^2$ and $H$ is a minor of $G$,
then $|H|$ embeds in $S^2$.
This observation along with Kuratowski's theorem immediately imply

\begin{theorem}[Wagner, 1937] A graph $G$ has no minor isomorphic to $K_5$ or
$K_{3,3}$ iff $|G|$ embeds in $S^2$.
\end{theorem}

We are now ready for a more substantial application of minors.
We call an embedding $g$ of a $1$-polyhedron $\Gamma$ in $S^m$ {\it knotless} if
for every circuit $C\subset\Gamma$, the restriction of $g$ to $C$ is an unknot.
We call an embedding $g$ of a polyhedron $X$ into $S^m$ {\it linkless} if for
every two disjoint closed subpolyhedra of $g(X)$, one is contained in an $m$-ball
disjoint from the other one.
An $n$-polyhedron admitting no linkless embedding in $S^{2n+1}$ is also known in
the literature (at least for $n=1$) as an ``intrinsically linked'' polyhedron.

\includegraphics{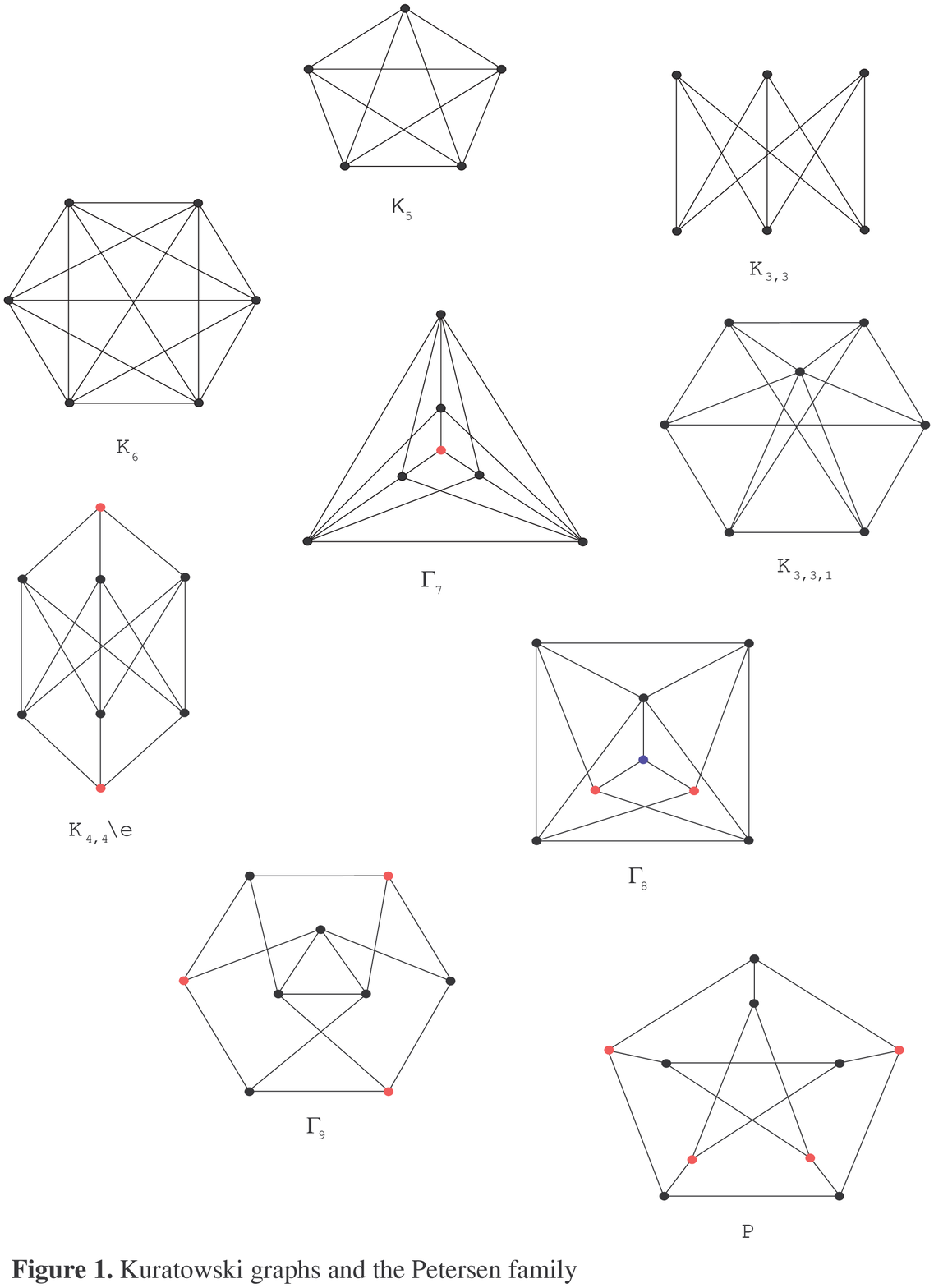}
\newpage

For $n>1$, the notion of a linkless embedding in $S^{2n+1}$ can be reformulated
in terms of linking numbers of pairs of {\it co}cycles (see Lemma \ref{circuits}
and Lemma \ref{is-linkless}).
When $n=1$, variations of the notion of ``linking'' (such as existence of
a nontrivial two-component sublink, or of a two-component sublink with nonzero
linking number) lead to inequivalent versions of the notion of a linkless
embedding in $\R^3$.
It is amazing, however, that they all become equivalent upon adding
a quantifier:

\begin{theorem}[Robertson--Seymour--Thomas, 1995, + $\eps$]\label{RST0}
Let $G$ be a graph. If $|G|$ admits an embedding in $S^3$ that links every
pair of disjoint circuits with an even linking number, then $|G|$ admits a
linkless, knotless embedding in $S^3$.
\end{theorem}

This is a very powerful result; it implies, {\it inter alia}, that the Whitney
trick can be made to work in dimension four in a certain limited class of problems
(Theorem \ref{Whitney}).

Theorem \ref{RST0} above as well as the following Theorem \ref{RST} are essentially
due to Robertson, Seymour and Thomas \cite{RST}, \cite{RST'}; they had a different
formulation but its equivalence with the present one is relatively easy \cite{M2}
(see also Lemma \ref{panels} and Remark \ref{erratum} below).

\begin{theorem}[Robertson--Seymour--Thomas, 1995, + $\eps$]\label{RST}
Let $G$ be a graph.

(a) Two linkless, knotless embeddings of $|G|$ in $S^3$ are inequivalent iff
they differ already on $|H|$ for some subgraph $H$ of $G$, isomorphic to
a subdivision of $K_5$ or $K_{3,3}$.

(b) $|G|$ linklessly embeds in $S^3$ iff $G$ has no minor in the Petersen family.
\end{theorem}

The {\it Petersen family} of graphs is shown in Fig.\ 1 (disregard the colors for
now) and includes the Petersen graph $P$, the complete graph $K_6$, the complete
tripartite graph $K_{3,3,1}$, the graph $K_{4,4}\but$(edge), and three further
graphs which we denote $\Gamma_7$, $\Gamma_8$ and $\Gamma_9$.
These seven graphs can be defined as all the graphs obtainable from $K_6$ by
a sequence of $\nabla{\mathrm Y}$- and ${\mathrm Y}\nabla$-{\it exchanges}, which
as their name suggests interchange a copy of the $1$-skeleton of a $2$-simplex
$\Delta^2\subset\Delta^3$ with a copy of the triod, identified with the remaining
edges of the $1$-skeleton of $\Delta^3$.

Of course, (a) can be reformulated in terms of minors, for given a minor $H$ of
$G$ and a knotless and linkless embedding of $|G|$ in $S^3$, the homeomorphism
$S^3/\text{tree}\cong S^3$ (see Lemma \ref{trivial}) yields a knotless and linkless
embedding of $|H|$ in $S^3$, unique up to equivalence (by an orientation-preserving
homeomorphism of $S^3$).

\begin{remark} Robertson and Seymour have proved, in a series of twenty papers
spanning over 500 pages, that every minor-closed family of graphs is
characterized by a finite set of forbidden minors (see \cite{Di} for an outline).
Let us abbreviate ``a subdivision of a subgraph'' to a {\it $\tau$-subgraph}.
There exists a $\tau$-subgraph-closed family of trees that is not characterized
by a finite set of forbidden $\tau$-subgraphs (cf.\ \cite[\S12, Exercise 5]{Di}).

However, it is well-known and easy to see that every minor-closed family of graphs
that is characterized by a finite set $S$ of forbidden minors is also characterized
by a possibly larger but still finite set $S^+$ of forbidden $\tau$-subgraphs.
(For each $G\in S$ and each $v\in G$, replace the star $\st(v,G)$ by a tree with
$\deg v$ leaves and no vertices of degree $2$.
Since for every $v$ there are only finitely many such trees, we obtain a finite set
of graphs containing $S^+$.)
It is easy to see that $\{K_5,\,K_{3,3}\}^+=\{K_5,\,K_{3,3}\}$, but $\Pi^+\ne\Pi$
where $\Pi$ is the Petersen family.
\end{remark}

\begin{remark}
Y. Colin de Verdiere introduced a combinatorially defined invariant $\mu(G)$ of
the graph $G$, for which it is known that $\mu(G)\le 3$ iff $|G|$ is planar, and
$\mu(G)\le 4$ iff $|G|$ admits a linkless embedding in $S^3$ (see \cite{LS},
\cite{Iz}).

Van der Holst conjectured that $\mu(G)\le 5$ iff $|\hat G|$ has zero $\bmod 2$ van
Kampen obstruction, where $\hat G$ is the cell complex obtained by glueing up
all circuits in $G$ by $2$-cells \cite{Ho}.
His supporting evidence for this conjecture is that if $G$ is any of the 78 graphs
(called the {\it Heawood family} in \cite{Ho}), related by a sequence of $\nabla$Y-
and Y$\nabla$-exchanges to $K_7$ or $K_{3,3,1,1}$, then $\mu(G)=6$, and $|\hat G|$
has nonzero $\bmod 2$ van Kampen obstruction (so in particular does not embed in
$S^4$); but if $H$ is a proper minor of $G$, then $\mu(G)\le 5$ and $|\hat H|$
has zero $\bmod 2$ van Kampen obstruction \cite{Ho}.

On the other hand, since $K_7$ contains circuits of length $\ge 4$, $\hat K_7$
itself contains a {\it proper} subcomplex isomorphic to the $2$-skeleton of
the $6$-simplex, and it is well-known (see \cite{M2}) that $|(\Delta^6)^{(2)}|$
still has a nonzero $\bmod 2$ van Kampen obstruction.
The similar proper subcomplex of $\hat K_{3,3,1,1}$ is discussed in
\S\ref{dichotomial spheres} below.
\end{remark}

\begin{remark} Let $E_n$ stand for ``the problem of embeddability of a certain
$n$-polyhedron in $S^{2n}$'', and let $L_n$ stand for ``the problem of
linkless embeddability of a certain $n$-polyhedron in $S^{2n+1}$''; it is
understood that the polyhedron is specified, but nevertheless omitted in
the notation.
In \S\ref{commensurability} we describe a reduction (if and only if) of every $E_n$
to an $L_n$; and of every $L_n$ to an $E_{n+1}$ (Theorem \ref{commensuration}).
The construction is geometric, i.e.\ does not involve configuration spaces.
The case $n=1$ was also done by van der Holst \cite{Ho}.

Let us mention other known relations between embeddability and linkless
embeddability.
(i) M. Skopenkov has derived the non-embeddability of a certain $n$-polyhedron in
$S^{2n}$ from non-existence of linkless embeddings of its links of vertices in
$S^{2n-1}$ (thereby reducing a certain $E_n$ to a few $L_{n-1}$'s) by
a geometric argument \cite{Sk-M2}.
(ii) Conversely, the author derived the linkless non-embeddability of
the $n$-skeleton of the $(2n+3)$-simplex in $S^{2n+1}$ from the fact that
the $(n+1)$-skeleton of the $(2n+4)$-simplex has nonzero van Kampen obstruction
even modulo $2$ (thereby reducing a certain $L_n$ to a certain strengthened
$E_{n+1}$ by an algebraic argument with configuration spaces) \cite[Example 4.7]{M2}.
(iii) An odd-dimensional version of the van Kampen obstruction $\theta_{2n}$
to an $E_n$ can be identified as a complete obstruction $\theta_{2n+1}$ to
an $L_n$ \cite{M2}.
\end{remark}

Arguments which might relate to a possible common higher dimensional
generalization of the Kuratowski and Robertson--Seymour--Thomas theorems
could include, apart from those in \S\ref{commensurability}, those in
\cite{Sa0}, \cite{Sa1} (see also \cite{Um1}), \cite{Sk-A2}.
Unfortunately, the present paper gives little clue to understanding {\it proofs} of
the Kuratowski and Robertson--Seymour--Thomas theorems, but it does attempt to
provide a better understanding of their {\it statements}.

To this end let us first look at a more statistically representative selection
of cases.

\subsection{Self-dual complexes}
MacLane and Adkisson proved that two embeddings of a $1$-polyhedron $\Gamma$ in
the plane are inequivalent iff they differ already on a copy of $S^1$ or on
a copy of the triod contained in $\Gamma$ \cite{AdM}.
We restate this as follows.

\begin{theorem}[MacLane--Adkisson, 1941]
If $K$ is a simplicial complex, then two embeddings of $|K|$ in $S^2$ are
inequivalent iff they differ already on $|L|$ for some subcomplex $L$ of $K$,
isomorphic to $\Delta^2$ or $(\Delta^2)^{(0)}*\Delta^0$ or a subdivision of
$\Delta^0\sqcup\partial\Delta^2$.
\end{theorem}

It is quite obvious also that $|K|$ embeds in $S^1$ iff $K$ contains no
subcomplex isomorphic to $\Delta^2$ or $(\Delta^2)^{(0)}*\Delta^0$ or
a subdivision of $\Delta^0\sqcup\partial\Delta^2$.

Halin and Jung were able to go one dimension higher and gave a list of prohibited
subcomplexes for the problem of embedding of a $2$-polyhedron in the plane \cite{HJ}.
We again restate it as applied to the problem of embedding a polyhedron of
an arbitrary dimension in $S^2$:

\begin{theorem}[Halin--Jung, 1964] If $K$ is a simplicial complex, $|K|$ embeds
in $S^2$ iff it does not contain a subcomplex, isomorphic to a subdivision of
$K_5$, $K_{3,3}$, or one of the following complexes:
\begin{gather*}
HJ_0\bydef\Delta^3,\\
HJ_1\bydef\Delta^0\sqcup\partial\Delta^3,\\
HJ_2\bydef\Delta^0*(\Delta^0\sqcup\partial\Delta^2),\\
HJ_3\bydef\Delta^1*\partial\Delta^1*\emptyset\cup
\partial\Delta^1*\emptyset*\Delta^0\cup\emptyset*\Delta^1*\emptyset,\\
HJ_4\bydef(\Delta^2)^{(0)}*\Delta^1,\\
HJ_5\bydef\Delta^2*\emptyset\cup(\Delta^2)^{(0)}*\partial\Delta^1.\\
\end{gather*}
\end{theorem}

See \cite[Appendix A]{MTW} for pictures of $HJ_1$ through $HJ_5$.
We note that $|HJ_0|$, $|HJ_3|$, $|HJ_4|$ and $|HJ_5|$ each contain a subpolyhedron
homeomorphic to one of $|K_5|$, $|K_{3,3}|$, $|HJ_1|$ and $|HJ_2|$.
This leads to a shorter list of prohibited {\it subpolyhedra} for the problem of
embedding a polyhedron in $S^2$.
A proof of this weak version of the Halin--Jung theorem is rather easy modulo
the Kuratowski theorem (see \cite{Sk-A1}).

A statement like that of the preceding theorem might look bewildering, and just
like with the previously mentioned results, its proof does not seem to explain what
is special about these particular complexes.
But the following definition, going back to Schild \cite{Sch}, does it, by $7/8$:

A subcomplex of $\partial\Delta^n$ is called {\it self-dual} in $\partial\Delta^n$
if it contains precisely one face out of each pair $\Delta^k$, $\Delta^{n-k-1}$
of complementary faces of $\Delta^n$.

Now observe that $K_5$, as well as each $HJ_i$, is self-dual as a subcomplex of
$\partial\Delta^4$.
Indeed, this is obvious for $HJ_0$.
But each $HJ_{i+1}$, can be obtained from $HJ_i$ by exchanging a pair of
complementary faces of $\partial\Delta^4$, except that $HJ_5$ is obtained in
this way from $HJ_3$ not $HJ_4$.
Also, $K_5$ is obtained in this way from $HJ_5$.

It is not hard to check that the seven complexes $K_5$ and the $HJ_i$,
are in fact all the self-dual subcomplexes of $\partial\Delta^4$ (up to isomorphism).
Note that $K_{3,3}$ is missing in this picture (but see Example
\ref{Halin-Jung revisited} below).

Similarly one can check that the three complexes $\Delta^2$,
$(\Delta^2)^{(0)}*\Delta^0$ and $\Delta^0\sqcup\partial\Delta^2$ in
the MacLane--Adkisson theorem are precisely all the self-dual subcomplexes
of $\partial\Delta^3$.

This is no coincidence; in fact, the following general result already implies
that the prohibited subcomplexes that occur in the Kuratowski, MacLane--Adkisson
and Halin--Jung theorems as well as in part (a) of the Robertson--Seymour--Thomas
theorem {\it must} occur there (possibly along with some additional ones).

\begin{theorem}\label{self-dual} Let $K=K_1*\dots*K_r$, where each $K_i$ is
a self-dual subcomplex of $\partial\Delta^{m_i}$, where $m_1+\dots+m_r=m$.
Then

(a) {\rm (Schild, 1993)}
$|K|$ does not embed in $S^{m-2}$; but $|L|$ embeds in $S^{m-2}$ for every proper
subcomplex $L$ of $K$;

(b) every embedding $g$ of $|K|$ in $S^{m-1}$ is inequivalent to $hg$, where $h$
is an orientation-reversing homeomorphism of $S^{m-1}$; but for every proper
subcomplex $L$ of $K$, the restrictions of $j$ and $hj$ to $|L|$ are equivalent,
where $j$ denotes the inclusion of $|K|$ in
$|\partial\Delta^{m_1}*\dots*\partial\Delta^{m_r}|=S^{m-1}$.
\end{theorem}

A simple proof of the assertions on proper subcomplexes in (a) and (b) is
given below.
A simple proof of the non-embeddability in (a) and the inequivalence in (b) is
given in \S\ref{embeddings}, where we also elaborate on historic/logical
antecedents of Theorem \ref{self-dual}.

\begin{proof}[Proof. Subcomplexes in (a)]
Let $\sigma$ be a maximal simplex of $K$ that is not in $L$.
Then $\sigma=\sigma_1*\dots*\sigma_r$, where each $\sigma_i$ is a maximal
(and in particular nonempty) simplex of $K_i$.
Since $K_i$ is self-dual, the complementary simplex $\tau_i$ to $\sigma_i$
is not in $K_i$.
It follows that $L\subset(\partial\tau_1*\dots*\partial\tau_r)*\partial\sigma$.
The latter is a combinatorial sphere, which is of codimension $r+1$ in the
$(m+r-1)$-simplex $(\tau_1*\dots*\tau_r)*\sigma=\Delta^{m_1}*\dots*\Delta^{m_r}$.
\end{proof}

\begin{proof}[Subcomplexes in (b)] The preceding construction exhibits, within
the sphere $S^{m-1}$ that contains $|K|$, an embedded copy of $S^{m-2}$
that contains $|L|$.
Let $r\:S^{m-1}\to S^{m-1}$ be the reflection in $S^{m-2}$.
Then $hr^{-1}=hr$ is orientation preserving, and so is isotopic to the identity
by the Alexander trick.
Hence $hj$ is equivalent to $rj=j$.
\end{proof}

\begin{remark}\label{van Kampen} The above construction can be easily extended
(cf.\ \cite[Example 3.5]{M2}) to yield, for any maximal simplex $\sigma$ of $K$,
a map $f_\sigma\:|K|\to S^{m-1}$ with precisely one double point, one of whose
two preimages lies in the interior of $|\sigma|$.
This implies that every proper subpolyhedron $P$ of $|K|$ embeds in $S^{2m-2}$.
Indeed, since $P$ is compact, it is disjoint from some point in the interior of
$|\sigma|$ for some maximal face $\sigma$ of $K$, and therefore from a disk $D$
in the interior of $|\sigma|$.
Then $P$ is embedded in $S^{m-1}$ by $f_\sigma$ precomposed with an appropriate
self-homeomorphism of $|K|$, fixed outside $|\sigma|$.
\end{remark}

We shall next see that Theorem \ref{self-dual} is not as exciting as it might
appear to be.

\subsection{Collapsible and cell-like maps}

In this subsection we assume familiarity with collapsing and regular neighborhoods
(see e.g.\ \cite{Hu}).
The following fact is well-known.

\begin{lemma}\label{trivial} If $Q$ is a collapsible subpolyhedron of a manifold
$M$, then the quotient $M/Q$ is homeomorphic to $M$.
\end{lemma}

\begin{proof} Let $N$ be a regular neighborhood of $Q$ in $M$.
Since $Q$ is collapsible, $N$ is a ball.
On the other hand, $N/Q$ is homeomorphic rel $\partial N$ to
$pt*\partial N$, which is also a ball.
This yields a homeomorphism $N\to pt*\partial N$ keeping $\partial N$ fixed, which
extends by the identity to a homeomorphism $M\to M/Q$.
\end{proof}

Let us call a map between polyhedra {\it finite-collapsible} if it is
the composition of a sequence of quotient maps, each shrinking a collapsible
subpolyhedron to a point.

\begin{corollary}\label{finite-collapsible} Let $f\:P\to Q$ be
a finite-collapsible map between polyhedra.
If $P$ embeds in $S^m$, then so does $Q$.
\end{corollary}

We use this nearly trivial observation to give a simple proof of a result
by Zaks \cite{Za} (see also \cite{Um2}, \cite[3.7.1]{Sa1}); as a byproduct,
we also get a slightly stronger statement:

\begin{theorem}\label{Zaks} For each $n\ge 2$ there exists an infinite list of
pairwise non-homeomorphic compact $(n-1)$-connected $n$-polyhedra $P_i$ such that
each $P_i$ does not embed in $S^{2n}$, but every its proper subpolyhedron does.
\end{theorem}

Zaks' series of examples satisfied the conclusion of Theorem \ref{Zaks} except
for being $(n-1)$-connected.

\begin{proof}
Let $K$ be the $n$-skeleton of the $(2n+2)$-simplex and let $\sigma$ be
an $(n-1)$-simplex of $K$.
By inspection, it is a face of at least three $n$-simplices; let $\tau$
be one of them.
Let $b$ be the barycenter of $\sigma$, and let $B_i$ (resp.\ $D_i$) be the star
of $b$ in the $i$th barycentric subdivision of $\tau$ (resp.\ of $\sigma$).
Let $C_i$ be the closure of $|\partial B_i|\but |D_i|$; it is a codimension one
ball properly embedded in $|\tau|$, with boundary sphere embedded in $|\sigma|$.

Given a positive integer $r$, let $P_r$ be the polyhedron obtained from
$P_0\bydef|K|$ by shrinking each $C_i$ to a point $p_i$ for $i=1,\dots,r$.
Then the quotient map $f_r\:P_0\to P_r$ is finite-collapsible, and sends $|B_1|$
onto a collapsible polyhedron $X_r$.
The quotient $P_r/X_r$ is homeomorphic to $P_0/|B_1|$ and therefore to $P_0$.
Thus we obtain a finite-collapsible map $g_r\:P_r\to P_0$.

The links of the points $p_i$ in $P_r$ are homeomorphic to each other and not
homeomorphic to the link of any other point in $P_r$.
Consequently, $P_0,P_1,P_2,\dots$ are pairwise non-homeomorphic.
Since $P_0$ does not embed in $S^{2n}$ (see Theorem \ref{self-dual}), and $g_r$
is finite-collapsible, $P_r$ does not embed in $S^{2n}$.
If $Q$ is a proper subpolyhedron of $P_r$, then $R\bydef f_r^{-1}(Q)$ is a proper
subpolyhedron of $P_0$, and $f_r|_R\:R\to Q$ is clearly finite-collapsible.
Since $R$ embeds in $S^{2n}$ (see Remark \ref{van Kampen}), so does $Q$.
\end{proof}

A map between polyhedra is called {\it collapsible}, resp.\ {\it cell-like} if
every its point-inverse is collapsible, resp.\ contractible (and so, in particular,
nonempty).
The following generalization of Corollary \ref{finite-collapsible} is a relatively
easy consequence of well-known classical results.

\begin{theorem}\label{minors embed} (a) If $f\:P\to Q$ is a collapsible map between
polyhedra, and $P$ embeds in a manifold $M$, then $Q$ embeds in $M$.

(b) If $f\:P\to Q$ is a cell-like map between $n$-polyhedra, and $P$ embeds in
an $m$-manifold $M$, where $m\ge n+3$, then $Q$ embeds in $M$.
\end{theorem}

The proof is not hard and quite instructive, but as this introduction is getting
a bit too involved we defer it until \S\ref{collapsing}.
There we also elaborate on the following

\begin{corollary}\label{quasi-embedding} If $f\:X\to Y$ is a map between
$n$-polyhedra whose nondegenerate point-inverses lie in a subpolyhedron of
dimension $\le m-n-2$, and $X$ embeds in $S^m$, then $Y$ embeds in $S^m$.
\end{corollary}

\begin{proof} The given embedding of $X$ extends by general position to
an embedding of $X\cup_Z MC(f|_Z)$, where $Z$ is the given subpolyhedron and
$MC$ denotes the mapping cylinder.
On the other hand, the point-inverses of the projection $X\cup_Z MC(f|_Z)\to Y$
are cones, so Theorem \ref{minors embed}(a) yields an embedding of $Y$.
\end{proof}

Let us call a polyhedron $Z$ an {\it $h$-minor} of a polyhedron $X$, if
there exists a subpolyhedron $Y$ of $X$ and a cell-like map $Y\to Z$.
(We note that composition of cell-like maps is obviously cell-like, cf.\ \cite{Sm};
see also \cite[comment to Corollary 2.3]{Hat} for a combinatorial proof.)
By Theorem \ref{minors embed}(a), all $h$-minors of an $n$-polyhedron embeddable
in $S^m$, $m-n\ge 3$, also embed in $S^m$.

One could hope that using $h$-minors instead of subpolyhedra enables one to
prevent Theorem \ref{Zaks} from ``driving us from the paradise'' which Theorem
\ref{self-dual} might seem to promise.
This is not so: using Corollary \ref{quasi-embedding}, it is easy to construct
an infinite list $P_0,P_1,\dots$ of pairwise non-homeomorphic $n$-polyhedra,
$n\ge 3$, such that each $P_i$ does not embed in $S^{2n}$, but every its proper
(in any reasonable sense) $h$-minor embeds in $S^{2n}$.
In fact, the original examples of Zaks \cite{Za} work (compare \cite[6.5]{Ne}).

However,  all such examples (Zaks' examples and their modifications constructed
using Corollary \ref{quasi-embedding}) are not going to be $(n-1)$-connected.
Relevance of $(n-1)$-connected $n$-polyhedra is ensured by the following
observation.

\begin{theorem}\label{connected} Let $P$ be an $n$-polyhedron, $n\ne 2$, that
embeds in $S^{2n}$.
Then $P$ embeds in an $(n-1)$-connected $n$-polyhedron $Q$ such that $Q$
embeds in $S^{2n}$.
\end{theorem}

\begin{proof} We shall show that more generally if $P$ is
an $n$-polyhedron, $n\ge 2$, with vanishing van Kampen obstruction,
then $P$ embeds in a polyhedron $Q$ with vanishing van Kampen obstruction.
The van Kampen obstruction is well-known to be complete for $n\ne 2$ (see \cite{M2}
and Remark \ref{erratum2} below).

Let $K$ be some triangulation of $P$.
By general position $P\cup |CK^{(n-2)}|$ embeds in $S^{2n}$.
So without loss of generality $P$ is $(n-2)$-connected.
Then $P$ is homotopy equivalent to a wedge of $(n-1)$-spheres.
If $n=2$, $\pi_1(P)$ is finitely generated (and even finitely presented).
If $n\ge 3$, $P$ is simply-connected, hence $\pi_{n-1}(P)$ is finitely
generated over $\Z$ (see \cite[20.6.2(3)]{tD} for a new simple proof of this
classical result).
Let $f_i\: S^{n-1}\to P$ represent the free homotopy classes of some finite basis
of $\pi_{n-1}(P)$.
We may amend the $f_i$ so that the image of each $f_i$ meets each
$n$-simplex of $K$.
Let $Q$ be obtained by adjoining $n$-disks to $P$ along the $f_i$.
By construction, $Q$ is $(n-1)$-connected.
Moreover, due to our choice of the $f_i$, the van Kampen obstruction of $Q$
is zero.
\end{proof}

\begin{remark}\label{dual-collapsible}
The proof of Theorem \ref{minors embed}(a) works to establish it for the more
general class of dual-collapsible maps.
A map is {\it dual-collapsible} if it can be triangulated by a simplicial map
$f\:K\to L$ such that for each simplex $\tau$ of $L$, its dual cone $\tau^*$
has collapsible preimage $|f^{-1}(\tau^*)|$.
It can be shown that composition of dual-collapsible maps is dual-collapsible
(which seems not to be the case for collapsible maps, even though it happens
to be the case for collapsible retractions \cite[8.6]{Co2}).
Many assertions in the present paper involving cell-like maps (directly or through
$h$-minors), including the statement of the main theorem, hold for dual-collapsible
maps.
\end{remark}

\subsection{Edge-minors}\label{edge-minors}

An {\it edge contraction} is a simplicial map $f\:K\to L$ that sends every edge
onto an edge, apart from one edge $\{v_1,v_2\}$ which it shrinks onto a vertex.
We call $f$ admissible if $\lk\{v_1\}\cap\lk\{v_2\}=\lk\{v_1,v_2\}$.
An equivalent condition is that $\{v_1,v_2\}$ is not contained in any ``missing
face'' of $K$, i.e.\ in an isomorphic copy of $\partial\Delta^n$ in $K$ that does
not extend to an isomorphic copy of $\Delta^n$ in $K$.
We define a simplicial complex $L$ to be an {\it edge-minor} of a simplicial
complex $K$ if $L$ is obtained from a subcomplex of $K$ by a sequence of
admissible edge contractions.

Yet another equivalent formulation of the admissibility condition is that
every point-inverse of $|f|$ is collapsible.
This has the following consequences:

\begin{enumerate}
\item If $f\:K\to L$ is an edge contraction, and $\Lambda$ is a subcomplex of $L$,
then $f|_{f^{-1}(\Lambda)}$ is either an edge contraction or a homeomorphism.

\item If $L$ is an edge-minor of $K$ and $|K|$ embeds in $S^m$, then $|L|$
embeds in $S^m$.
\end{enumerate}

E. Nevo considered a slightly different definition of a ``minor'' which we
shall term a {\it Nevo minor} \cite{Ne}.
It is similar to that of edge-minor, with the following amendment.
We call an edge contraction $f\:K\to L$, where $\dim K=n$, Nevo-admissible, if
the $(n-2)$-skeleton $(\lk\{v_1\}\cap\lk\{v_2\})^{(n-2)}=\lk\{v_1,v_2\}$;
an equivalent condition is that $\{v_1,v_2\}$ is not contained in any missing
face of $K$ of (missing) dimension $\le n$.
Note that in the case of graphs this is a vacuous condition.

\begin{proposition}\label{nevo} The notions of Nevo minor and edge-minor are
equivalent.
\end{proposition}

The author learned from E. Nevo that the published version of his paper in fact
contains this remark, which was also pointed out by his referee.
We note that Proposition \ref{nevo} along with assertion (2) above yields a proof
of Conjecture 1.3 in \cite{Ne}: if $L$ is a Nevo minor of $K$, and $|K|$ embeds
in $S^m$, then $|L|$ embeds in $S^m$.

\begin{proof} By assertion (1) above, it suffices to show that if $f\:K\to L$ is
a Nevo-admissible edge contraction, then $L$ is a minor of $K$.
Suppose that $f$ shrinks an edge $e=\{v_1,v_2\}$.
If $f$ is not admissible, $\lk(v_1)\cap\lk(v_2)$ is the union of $\lk(e)$
with a nonempty collection of $(n-1)$-simplices $\sigma_1,\dots,\sigma_k$.
Let $K^+$ be the subcomplex of $K$ obtained by removing the $n$-simplices
$\{v_0\}*\sigma_i$ from $K$.
Then $f|_{K^+}\:K^+\to L$ is an admissible edge contraction.
Hence $L$ is a minor of $K$.
\end{proof}

\begin{lemma}\label{Steinitz} A simplicial complex of dimension $\le 2$ is
an edge-minor of every its subdivision.
\end{lemma}

This is proved in \S\ref{edge-minors2} using innermost circle arguments.

Since every triangulation $T$ of $S^2$ is easily seen to be a subdivision of
$\partial\Delta^3$, Lemma \ref{Steinitz} implies the result of Steinitz (1934)
that $\partial\Delta^3$ is a minor of $T$; see \cite{Zo} and \cite{Su}.

In contrast, there exists a subdivision $S$ of $\partial\Delta^4$ such that
$\partial\Delta^4$ is not an edge-minor of $S$; see \cite[Example 6.1]{Ne}.

The preceding results now imply the following version of the Halin--Jung theorem:

\begin{theorem} A simplicial complex $K$ has no edge-minors among the seven self-dual
subcomplexes of $\partial\Delta^4$ along with $K_{3,3}$ iff $|K|$ embeds in $S^2$.
\end{theorem}

Since a given (finite) simplicial complex only has finitely many minors,
this yields an algorithm deciding embeddability of $|K|$ in $S^2$.
A presumably faster, but more elaborate algorithm is discussed in \cite{MTW}.

\begin{addendum}[to Theorem \ref{self-dual}]\label{self-dual-addendum}
Let $K=K_1*\dots*K_r$, where each $K_i$ is a self-dual subcomplex of
$\partial\Delta^{m_i}$, where $m_1+\dots+m_r=m$, and let $L$ be a proper
edge-minor of $K$.
Then

(a) $|L|$ embeds in $S^{m-2}$;

(b) the embeddings of $|L|$ in $S^{m-1}$ induced by $j$ and $hj$ are equivalent,
where $j$ is the inclusion of $|K|$ in
$S^{m-1}\bydef|\partial\Delta^{m_1}*\dots*\partial\Delta^{m_r}|$, and $h$ is
an orientation-reversing homeomorphism of $S^{m-1}$.
\end{addendum}

The proof, given in \S\ref{edge-minors2}, is reminiscent of the above argument
towards Theorem \ref{self-dual}.

\subsection{Hemi-icosahedron and hemi-dodecahedron}\label{semi}

The central symmetry of $\R^3$ in the origin yields a simplicial free involution
on the boundary of a regular icosahedron centered at the origin.
Its quotient by this involution is a simplicial complex $\R P^2_\triangle$ with
$6$ vertices, triangulating the real projective plane.
It is easy to see that its $1$-skeleton is the complete graph $K_6$, and for each
pair of disjoint circuits in $K_6$, precisely one bounds a $2$-simplex in
$\R P^2_\triangle$.
Hence $\R P^2_\triangle$ is self-dual as a subcomplex of $\partial\Delta^5$
(compare \cite[5.8.5]{Mat}).
By Theorem \ref{self-dual} this implies that the projective plane
$\R P^2=|\R P^2_\triangle|$ does not embed in $S^3$, embeds in $S^4$, and
every embedding of $\R P^2$ in $S^4$ is inequivalent to its reflection.
In fact, every embedding of $\R P^2$ in $S^4$ with fundamental group of
the complement isomorphic to $\Z/2$ is known to be topologically equivalent
to either the standard embedding or its reflection \cite{La3}.

A well-known $9$-vertex triangulation of $\C P^2$ \cite{KB} (see \cite{MY},
\cite{BD1}, \cite{MS}, \cite{BD3} for other constructions) is easily seen to be
self-dual as a subcomplex of $\partial\Delta^8$.
Thus Theorem \ref{self-dual} applies again.
In this connection we note that every embedding of $\C P^2$ in $S^7$ is known
to be equivalent to either the standard embedding or its reflection \cite{Sk-A3}.

In fact, it is known that a combinatorial $n$-manifold with $v$ vertices is
self-dual as a subcomplex of $\partial\Delta^{v-1}$ if and only if $2v=3n+6$
\cite{ArM} (if), \cite{Da} (only if).
Combinatorial $n$-manifolds with $\frac{3n}2+3$ vertices can only occur in
dimensions $n=0,2,4,8,16$ and with a $\bmod 2$ cohomology ring isomorphic
to that of the respective projective plane \cite{BK1} (see also \cite{La2}).
In dimensions $2$ and $4$ these are unique (up to a relabelling of vertices);
an algebraic topology proof can be found in \cite{ArM} and a combinatorial
one in \cite{BD2}.
In dimension $8$, there exist three self-dual subcomplexes of $\partial\Delta^{14}$,
all triangulating a certain $8$-manifold $\Ham P^2_{(?)}$ \cite{BK2}.
The existence in dimension $16$ seems to be still open.

Theorem \ref{self-dual} now implies

\begin{corollary}\label{1.10} The join of $r$ copies of $\R P^2$, $c$ copies of
$\C P^2$ and $h$ copies of $\Ham P^2_{(?)}$ embeds in
$S^{5r+8c+14h-1}$ and does not embed in $S^{5r+8c+14h-2}$.
Every embedding $g$ of this join in $S^{5r+8c+14h-1}$ is inequivalent to $hg$,
where $h$ is an orientation reversing homeomorphism of the sphere.
\end{corollary}

\begin{remark}
We note that $K_6$ admits an embedding in $S^3$ that links any given disjoint
circuits $|C|$, $|C'|$ in $|K_6|$ with any given odd linking number and does not
link any other disjoint pair of circuits.
Similar arguments work to prove the same assertion for any other graph of
the Petersen family (cf.\ Remark \ref{hemi-exchanges}).

Indeed, one of $C$, $C'$ bounds a simplex in the semi-icosahedron and the other
one can be identified with $\R P^1$.
Thus $K_6$ is the $1$-skeleton of a triangulation of the M\"obius band $\mu$
where $\partial\mu$ and the middle curve of $\mu$ are triangulated by $C$ and $C'$.
Using various embeddings of $\mu$ in $S^3$, e.g.\ those corresponding to all
half-odd-integer framings of the trivial knot, we obtain embeddings of $K_6$ in
$S^3$ with $\lk(|C|,|C'|)$ equal to any given odd number and all other disjoint
pairs of circuits geometrically unlinked.
\end{remark}

Similarly to the hemi-icosahedron $\R P^2_\triangle$ we have the hemi-dodecahedron
$\R P^2_\bigstar$ (compare \cite[\S6C]{McS}), whose $1$-skeleton is the Petersen
graph $P$, and for each pair of disjoint circuits in $P$, precisely one bounds
a cell in $\R P^2_\bigstar$.

This suggests that we should not be too fixed on simplicial complexes; this will
be our next concern.

\section{Main results}\label{section:main}

By a cell complex we mean what can be described as a finite CW-complex where
each attaching map is a homeomorphism of the sphere onto a subcomplex.
(We recall that we assume all maps between polyhedra to be PL by default.)
Note that the empty set is not a cell in our notation.

Cell complexes, and in particular simplicial complexes, are uniquely determined
by their face posets, so we may alternatively view them as posets of
a special kind.
This view is inherent in the following combinatorial notation \cite{M3}.

\subsection{Combinatorial notation}\label{notation}

Given a poset $P$ and a $\sigma\in P$, the {\it cone} $\fll\sigma\flr$
(resp.\ the {\it dual cone} $\cel\sigma\cer$) is the subposet of all $\tau\in P$
such that $\tau\le\sigma$ (resp.\ $\tau\ge\sigma$).

The {\it prejoin} $P+Q$ of posets $P=(\mathcal P,\le)$ and $Q=(\mathcal Q,\le)$
is the poset $(\mathcal P\sqcup\mathcal Q,\le)$, where $\mathcal P$ and $\mathcal Q$
retain their original orders, and every $p\in\mathcal P$ is set to be less than
every $q\in\mathcal Q$.
The {\it cone} $CP=P+\{\hat 1\}$ and the {\it dual cone} $C^*P=\{\hat 0\}+P$,
where $\{\hat 1\}$ and $\{\hat 0\}$ are just fancy notations for the one-element
poset.
The {\it boundary} $\partial Q$ of a cone $Q=CP$ is the original poset $P$,
and the {\it coboundary} $\partial^*Q$ of a dual cone $Q=C^*P$ is again $P$.
Given a finite set $S$, we denote the poset of all subsets of $S$ by $2^S$;
and the poset $\partial^*2^S$ of all nonempty subsets by $\Delta^S$.
When $S$ has cardinality $n+1$ and the nature of its elements is irrelevant,
the {\it (combinatorial) $n$-simplex} $\Delta^S$ is also denoted $\Delta^n$.

We call a poset $P$ a {\it simplicial complex} if it is a complete quasi-lattice
(i.e.\ every subset of $P$ that has an upper bound in $P$ has a least upper
bound in $P$; or equivalently every subset of $P$ that has a lower bound in $P$
has a greatest lower bound in $P$), and every cone of $P$ is order-isomorphic to
a simplex.

For a poset $P$ we distinguish its {\it barycentric subdivision} $P^\flat$ that is
the poset of all nonempty chains in $P$, ordered by inclusion, and the {\it order
complex} $|P|$ that is the polyhedron triangulated by the simplicial complex
$P^\flat$.
It should be noted that many fundamental homeomorphisms in combinatorial topology
can be promoted to combinatorial isomorphisms by upgrading from the barycentric
subdivision to the {\it canonical subdivision} $P^\#$ that is the poset of all order
intervals in $P$, ordered by inclusion \cite{BBC}, \cite{M3}.

We call a poset $P$ a {\it cell complex} if for every $p\in P$
the order complex $|\partial\fll p\flr|$ is homeomorphic to a sphere.

It is not hard to see that the so defined cell/simplicial complexes are precisely
the posets of nonempty faces of the customary cell/simplicial complexes \cite{M3}
(the case of cell complexes is trivial and well-known \cite{Mc}, \cite{Bj}).
General posets may be thought of as ``cone complexes'' \cite{vK2}, \cite{Mc}, \cite{M3}, and
their order-preserving maps may be thought of as ``conical'' maps.

From now on, we switch to the new formalism.

A few more auxiliary definitions follow.
The {\it dual} of a poset $P=(\mathcal P,\le)$ is the poset
$P^*\bydef(\mathcal P,\ge)$.
We note that $2^S$ is isomorphic to its own dual (by taking the complement);
and therefore so is $\partial\Delta^S=\partial(\partial^*2^S)$.
A poset $Q=(\mathcal Q,\preceq)$ is a {\it subposet} of $P$ if $\mathcal Q$ is
a subset of $\mathcal P$, and $p\preceq q$ iff $p\le q$ for all $p,q\in\mathcal Q$.
We will often identify a poset with its underlying set by an abuse of
notation.
A subposet $Q$ of $P$ is a {\it subcomplex} of $P$ if the cone (in $P$) of
every element of $Q$ lies in $Q$.

Let $P=(\mathcal P,\le)$ and $Q=(\mathcal Q,\le)$ be posets.
The {\it product} $P\x Q$ is the poset $(\mathcal P\x\mathcal Q,\preceq)$,
where $(p,q)\preceq (p',q')$ iff $p\le p'$ and $q\le q'$.
It is easy to see that $2^S\x 2^T\simeq 2^{S\sqcup T}$ naturally in $S$ and $T$.

The {\it join} $P*Q\bydef\partial^*(C^*P\x C^*Q)$ is obtained from $(C^*P)\x (C^*Q)$
by removing the bottom element $(\hat0,\hat0)$.
Thus $C^*(P*Q)\simeq C^*P\x C^*Q$, whereas $P*Q$ itself is the union
$C^*P\x Q\cup P\x C^*Q$ along their common part $P\x Q$.

From the above, $\Delta^S*\Delta^T\simeq\Delta^{S\sqcup T}$
naturally in $S$ and $T$.
It follows that the join of simplicial complexes
$K\subset\Delta^S$ and $L\subset\Delta^T$ is isomorphic to the simplicial complex
$\{\sigma\cup\tau\subset S\sqcup T\mid\sigma\in K\cup\{\emptyset\},
\tau\in L\cup\{\emptyset\},\,\sigma\cup\tau\ne\emptyset\}\subset\Delta^{S\sqcup T}$.

The join and the prejoin are related via barycentric subdivision:
$(P+Q)^\flat\simeq P^\flat*Q^\flat$.
Indeed, a nonempty finite chain in $P+Q$ consists of a finite chain in $P$
and a finite chain in $Q$, at least one of which is nonempty.
Note that in contrast to prejoin, join is commutative: $P*Q\simeq Q*P$.
Prejoin is associative; in particular, $C(C^*P)\simeq C^*(CP)$.

\subsection{{\it h}-Minors of cell complexes}\label{h-minors}

We call a cell complex $L$ an {\it $h$-minor} of a cell complex $K$, if there
exists a subcomplex $M$ of $K$ and an order preserving map $f\:M\to K$ such that
on the level of order complexes, $|f|\:|M|\to|K|$ is a cell-like map.
(Note that cell-like maps include dual-collapsible maps, see Remark
\ref{dual-collapsible}.)

\begin{example}[subdivision]\label{subdivision}
If a simplicial complex $L$ is an edge-minor of a simplicial complex $K$, then
of course $L$ is an $h$-minor of $K$; but not vice versa.
Indeed, let $K$ be a simplicial complex and $K'$ its simplicial subdivision
such that $K'$ is not an edge-minor of $K$ (see \cite[Example 6.1]{Ne}); we may
fix an identification of their order complexes.
Let $f\:K'\to K$ send every $\sigma\in K'$ to the minimal $\tau\in K$ such that
$|\sigma|\subset|\tau|$.
Then $f$ is an order preserving map such that $|f^{-1}(\fll\tau\flr)|$ is an
$n$-ball for each $n$-simplex $\tau\in K$.
(This understanding of a subdivision also arises in the study of PL transversality
\cite{M3} and in that of combinatorial grassmanians \cite{Mn}.)
It is easy to see that $|f^{-1}(\fll\tau\flr)|$ collapses onto $|f^{-1}(\tau)|$
for each $\tau\in K$, and it follows that $|f|$ is a cell-like map.

More generally, we say that an order preserving map $f\:P'\to P$ between posets
is a {\it subdivision} if $|f^{-1}(\fll\tau\flr)|$ is homeomorphic to
$|Cf^{-1}(\partial\fll\tau\flr)|$ by a homeomorphism fixed on
$|f^{-1}(\partial\fll\tau\flr)|$.
It is easy to see that $|f|$ is a cell-like map and $|P'|$ is homeomorphic
to $|P|$ \cite{M3} (see also \cite{Ak}, \cite[1.4]{DM} for a special case).

In particular, a cell complex is an $h$-minor of every its simplicial subdivision.
\end{example}

\begin{example}[zipping]
Another special case of taking an $h$-minor is zipping.
Given a poset $P$ and a $\sigma\in P$ such that
$\fll\sigma\flr$ is isomorphic to $Q+\Delta^1$ for some $Q$, by an isomorphism
$h$, we say that $P$ {\it elementarily zips} to $P/h^{-1}(\Delta^1)$ (the quotient
in the concrete category of posets over the category of sets, cf.\ \cite{AHS}).
A {\it zipping} is a sequence of elementary zippings.

It is not hard to see that if $K$ edge-contracts to $L$, then $K$ zips to $L$
(for instance, it takes two elementary zippings to zip a $2$-simplex onto
a $1$-simplex).
The author learned from E. Nevo that he has independently observed this fact,
and that a definition of zipping appears in Reading's paper \cite{Re}.
In fact, we borrow the term ``zipping'' from that paper.
E. Nevo also observed that if $K$ elementarily zips to $L$, then the
barycentric subdivision $K^\flat$ edge-contracts to $L^\flat$ in two steps.

We shall encounter a modification of zipping with $\Delta^1$ replaced by
$C((\Delta^2)^{(0)})$, as well as zipping itself, in the proof of
Proposition \ref{homology mfld}.
\end{example}

\begin{conjecture}\label{conjecture:main} For each $n$ there exist only finitely
many $n$-dimensional cell complexes $K$ such that $|K|$ is $(n-1)$-connected and
does not embed in $S^{2n}$, but $|L|$ embeds in $S^{2n}$ for each proper
$h$-minor $L$ of $K$.
\end{conjecture}

The author is not absolutely committed to this particular formulation, but
he strongly feels that at least some reasonable modification of this conjecture
should hold.
Some variations are discussed in \S\ref{algorithmic}.

\subsection{Dichotomial spheres}\label{dichotomial spheres}
Some members of the list in the preceding conjecture are provided by
the following result.

\begin{theorem}[Main Theorem] Let $B$ be an $m$-dimensional {\rm dichotomial} cell
complex, that is a cell complex that together with each cell $A$ contains a unique
cell $\bar A$ whose vertices are precisely all the vertices of $B$ that are not
in $A$.

Let $K$ be the $n$-skeleton of $B$, where $m=2n+1$ or $2n+2$,
and let $L$ be any proper $h$-minor of $K$.
Then:

(i) $B$ is uniquely determined by $K$.

(ii) $|B|\cong S^m$.

(iii) If $m=2n+1$, $n\ne 2$, then $|K|$ does not embed in $S^{2n}$, but $|L|$ does.

(iii$'$) If $m=2n+1$, $n=2$, then $|K|$ has non-zero van Kampen obstruction, even
modulo $2$ (so in particular does not embed in $S^4$) but $|L|$ has zero van Kampen
obstruction.

(iv) If $m=2n+2$, then $|K|$ does not linklessly embed in $S^{2n+1}$, but $|L|$
does.

(v) If $m=2n+1$, then every embedding $g$ of $|K|$ in $S^{2n+1}$ is
inequivalent with $hg$, where $h$ is an orientation-reversing
homeomorphism of $S^{2n+1}$; but every two embeddings of $|K|$ in $S^{2n+1}$
(knotless when $n=1$) have equivalent ``restrictions'' to $|L|$.

(vi) Moreover, if $m=2n+1$, then every embedding of $|K|$ in $S^{2n+1}$ is
linkless; and if $m=2n+2$, then every embedding of $|K|$ in $S^{2n+1}$ contains
a link of a pair of disjoint $n$-spheres with an odd linking number.

(vii) If $M$ is a {\rm self-dual} subcomplex of $B$ (that is, $M$ contains
precisely one cell out of every pair $A$, $\bar A$ of complementary cells),
then $|M|$ does not embed in $S^m$, and every embedding $g$ of $|M|$ in
$S^{m+1}$ is inequivalent with $hg$, where $h$ is an orientation-reversing
homeomorphism of $S^{m+1}$.
\end{theorem}

The simplest example of a dichotomial complex is the boundary of a simplex.
In particular, for $n=1$, the graphs in (iii) and (iv) include respectively
the complete graphs $K_5$ and $K_6$.
It is easy to see that there are no other dichotomial simplicial complexes,
apart from the boundary of a simplex.
The term ``dichotomial'' was suggested to the author by E. V. Shchepin.

\begin{example}
It is easy to construct the dichotomial $3$-sphere
whose $1$-skeleton is $K_{3,3}$.
In accordance with (i), we may start with $K_{3,3}$ itself.
For every edge $\tau$ of $K_{3,3}$, the four edges disjoint from $\tau$ form
a circuit; we glue up this circuit by a quadrilateral $2$-cell.
(Note that circuits of length $6$ are {\it not} glued up by hexagonal $2$-cells.)
For every vertex $\sigma$ of $K_{3,3}$, the edges disjoint from $\sigma$ form
a $K_{3,2}$, to which we have attached three $2$-cells.
Their union is a $2$-sphere; we glue it up by a $3$-cell.

Let us verify that the resulting dichotomial cell complex $B_{3,3}$ is indeed
a $3$-sphere.
We shall identify each element of $B_{3,3}$ with a subcomplex of the barycentric
subdivision of $\partial\Delta^2*\partial\Delta^2$.
The atoms of $B_{3,3}$ are identified with the vertices of
$\partial\Delta^2*\partial\Delta^2$, and the maximal elements of edges of $B_{3,3}$
are identified with the edges of $(\Delta^2)^{(0)}*(\Delta^2)^{(0)}$.
The maximal element of the $2$-cell of $B_{3,3}$ disjoint with an edge
$\sigma_1*\sigma_2$ of $K_{3,3}$ is identified with the disk
$D_{\sigma_1\sigma_2}\bydef c*\partial\tau_1*\partial\tau_2$, where $\tau_i$ is
the $1$-simplex in $\Delta^2$ disjoint from the vertex $\sigma_i$, and $c$ is
the barycenter of the simplex $\tau_1*\tau_2$.
The maximal element of the $3$-cell of $B_{3,3}$ disjoint with a vertex
$\sigma*\emptyset$ (resp.\ $\emptyset*\sigma$) of $K_{3,3}$ is identified with
the ball $E_{\sigma\emptyset}\bydef
c*(D_{\sigma\sigma_1}\cup D_{\sigma\sigma_2}\cup D_{\sigma\sigma_3})$
(resp.\ $E_{\emptyset\sigma}\bydef
c*(D_{\sigma_1\sigma}\cup D_{\sigma_2\sigma}\cup D_{\sigma_3\sigma})$),
where $c$ is the barycenter of the $1$-simplex $\tau$ of $\Delta^2$ disjoint
form $\sigma$, and $\sigma_1,\sigma_2,\sigma_2$ are the three vertices of
$\Delta^2$.
It is easy to see that the six balls $E_{\sigma\emptyset}$, $E_{\emptyset\sigma}$
cover the entire $(\partial\Delta^2*\partial\Delta^2)^\flat$.
\end{example}

\begin{example}
Similarly, it is not hard to construct the dichotomial $4$-sphere whose $1$-skeleton
is the Petersen graph $P$.
Firstly we glue up every circuit consisting of $5$ edges by a pentagonal $2$-cell.
(Note that circuits of length $6$ are {\it not} glued up by hexagonal $2$-cells.)
Then each $2$-cell already has its opposite $2$-cell.
For each edge $\tau$ of $P$, the edges disjoint from $\tau$ form a graph that
is a subdivision of the $1$-skeleton of the tetrahedron, and the four $2$-cells
disjoint from $\tau$ can be identified with the $2$-simplices of this tetrahedron,
with boundaries subdivided into pentagons.
Thus these edges and $2$-cells cellulate a $2$-sphere; we glue it up by a $3$-cell.
For every vertex $\sigma$ of $P$, the edges and the $2$-cells disjoint from
$\sigma$ form a cell complex that subdivides $(\partial\Delta^2)+(\Delta^2)^{(0)}$.
Then the edges, the $2$-cells and the $3$-cells disjoint from $\sigma$ form a cell
complex that subdivides the $3$-sphere $(\partial\Delta^2)+(\partial\Delta^2)$.
We glue it up by a $4$-cell.

According to part (ii) of the Main Theorem, the resulting dichotomial cell
complex $B_P$ is a $4$-sphere.
Note that the hemi-dodecahedron is isomorphic to its self-dual subcomplex.
\end{example}

It can be similarly verified by hand that the graphs of the Petersen family except
$K_{4,4}\but$(edge) are $1$-skeleta of $4$-dimensional dichotomial cell complexes;
whereas $K_{4,4}\but$(edge) is not.

We describe a more industrial way of seeing this in \S\ref{transforms}, where
we also observe that $K_{4,4}\but$(edge) is the $1$-skeleton of a $4$-dimensional
dichotomial poset, whose order complex admits a collapsible map onto $S^4$
(so in particular is homotopy equivalent to $S^4$).
It should be possible to include this graph into the general framework of
the Main Theorem by extending it to cone complexes whose cones are singular
cells, with not `too many' cells being `too singular'.
There can be other approaches (see Example \ref{4.9}).

\begin{remark}
We note that the assertion on minors in part (iii) of the Main Theorem
does not extend to cover arbitrary self-dual subcomplexes in part (vii).
Indeed, by zipping an edge of the hemi-icosahedron we obtain a non-simplicial
proper minor of the hemi-icosahedron which still cellulates $\R P^2$.
The author believes that it should be possible to overcome this trouble,
at least in the metastable range, by considering only minors that belong to
a restricted class of cell complexes, such as ones whose cells have
collapsible (or empty) pairwise intersections.
(Note that two cells do not always intersect along a cell, and three cells
do not always have a collapsible or empty intersection in the dichotomial
$4$-sphere whose $1$-skeleton is the Petersen graph.)
\end{remark}

Taking into account the Kuratowski--Wagner and Robertson--Seymour--Thomas
theorems and the obvious fact that a $0$-polyhedron admits a linkless embedding
in $S^1$ iff it contains less than $4$ points, we obtain

\begin{corollary}\label{main corollary} The only dichotomial complex in dimension
two is $\partial\Delta^3$; there exist precisely two in dimension $3$, with
$1$-skeleta $K_5$ and $K_{3,3}$, and precisely six in dimension $4$, with
$1$-skeleta all graphs of the Petersen family excluding $K_{4,4}\but e$.
\end{corollary}

Let us now review three constructions of new examples of dichotomial spheres.

(i) An $n$-dimensional join of the $i$-skeleta of dichotomial $(2i+1)$-spheres
is the $n$-skeleton of some dichotomial $(2n+1)$-sphere dimensional cell
complex (see Lemma \ref{5.3} and Theorem \ref{construction}).
For instance, this yields, in addition to $\partial\Delta^6$, two dichotomial
$5$-complexes, with $1$-skeleta $K_{1,1,1,1,1,3}$ and $K_{3,3,3}$ and with
simplicial $2$-skeleta.

(ii) If $B$ is a dichotomial cell complex, it is easy to see, using part (ii)
of Main Theorem, that $CB\cup_B B*pt$ is again a dichotomial cell complex.
Applied to $B_{3,3}$, this construction produces the dichotomial $4$-sphere
$B_{3,3,1}$ with $1$-skeleton $K_{3,3,1}$ (which belongs to the Petersen family).
Applying this construction to the six dichotomial $4$-spheres, we get, apart
from $\partial\Delta^6$, five dichotomial $5$-complexes whose $1$-skeleta are
obtained from the graphs $\Gamma_7$, $\Gamma_8$, $\Gamma_9$, $K_{3,3,1}$, $P$
by adjoining an additional vertex and connecting it by edges to all the existing
vertices.
For instance, in the case of $B_{3,3,1}$ this yields a dichotomial sphere
$B_{3,3,1,1}$ with $1$-skeleton $K_{3,3,1,1}$.

A more general construction is: given a dichotomial $m$-sphere $B$ and
a dichotomial $n$-sphere $B'$, the boundary sphere $\partial(CB*CB')$ of
the join of the cones is a dichotomial $(m+n+2)$-sphere.
(The previous paragraph treated the case where $B'$ is the dichotomial
$(-1)$-sphere $\emptyset$.)
Taking $B$ to be the dichotomial $3$-sphere with $1$-skeleton $K_{3,3}$ and
$B'$ to be the dichotomial $0$-sphere, we arrive again at $B_{3,3,1,1}$.
We note that the $2$-skeleton of $B_{3,3,1,1}$ is the union of the $2$-skeleton
of $B_{3,3}$ with the $2$-skeleton of the join $K_{3,3}*\Delta^1$.
Since the circuits of length $6$ in $K_{3,3}$ are not glued up by $2$-cells in
$B_{3,3}$, they are similarly not glued up by $2$-cells in $B_{3,3,1,1}$
(although they of course bound {\it subdivided} $2$-cells in the $2$-skeleton of
$B_{3,3,1,1}$).
Thus the $2$-skeleton of $B_{3,3,1,1}$ is a proper subcomplex of van der Holst's
$2$-complex $\hat K_{3,3,1,1}$ (see \S\ref{graphs}).

(iii) A more interesting method, called ($\nabla$,Y)-transform, is
introduced in \S\ref{transforms}.
It is natural to expect that an appropriate generalization of this transform
(or perhaps even the transform itself) would suffice to relate any two
dichotomial spheres of the same dimension (compare the proof of Steinitz'
theorem in \cite[Chapter 4]{Zi}).
It does relate with each other the two dichotomial $3$-spheres and all six
dichotomial $4$-spheres.
When applied to $\partial\Delta^6$, it produces, inter alia, dichotomial
$5$-complexes with $1$-skeleta obtained by adjoining an additional vertex to
any of the graphs $\Gamma_7$, $\Gamma_8$, $\Gamma_9$, $K_{4,4}\but$(edge), $P$
and connecting it to all existing vertices that are not marked red in Fig.\ 1.
(The details are similar to Examples \ref{4.8}, \ref{4.9}.)

To summarize, we have just used the easy (existence) part of Corollary
\ref{main corollary} to show the following

\begin{corollary}\label{2-skeleta} There exist at least $13$ dichotomial
$5$-spheres, distinct already on the level of their $1$-skeleta, of which
$10$ have non-simplicial $2$-skeleta.
\end{corollary}

The three simplicial $2$-skeleta, $(\Delta^6)^{(2)}$, $K_5*(\Delta^2)^{(0)}$ and
$K_{3,3}*(\Delta^2)^{(0)}$, are well-known \cite{Gr}, \cite{Sa1}.
The non-simplicial ones are probably new, although some (all?) of their $1$-skeleta
are in the Heawood family (see \S\ref{graphs}; it is clear that there in fact must
be many more dichotomial $5$-spheres with $1$-skeleta in the Heawood family).

By part (iii$'$) of the Main Theorem, none of the $13$ cell complexes in
Corollary \ref{2-skeleta} is a minor of any other one.
But beware that some of these $2$-complexes have underlying polyhedra that are
$h$-minors of each other (namely, some can be obtained from the others by
shrinking $2$-simplices to triods).

\begin{example}\label{Halin-Jung revisited}
The self-dual subcomplexes of the dichotomial $3$-sphere with
$1$-skeleton $K_{3,3}$ are, apart from $K_{3,3}$ itself:
\begin{gather*}
HJ'_4=I\x I\cup \{(0,0),(1,1)\}*a\cup\{(0,1),(1,0)\}*b,\\
HJ'_3=I\x I\cup_{0\x I\cup I\x 0}I\x I\cup \{(0,0)\}*a,\\
HJ'_2=S^2_\square\sqcup a,\\
HJ'_1=CS^2_\square,
\end{gather*}
where $S^2_\square$ is the cellulation of $S^2$ obtained by glueing up
all three circuits in $K_{3,2}$ by $2$-cells.
Each $|HJ'_i|$ is homeomorphic with some $|HJ_j|$, but none of $HJ_i$, $HJ'_j$,
$K_5$, $K_{3,3}$ is a subdivision of any other one.
\end{example}

\begin{problem}\label{problem:main} Given an $n\ge 5$, are there only finitely
many of dichotomial $n$-spheres?
\end{problem}

\subsection{Algorithmic issues}\label{algorithmic}

Due to the algorithmic nature of the van Kampen obstruction, the problem of
embeddability of a compact $n$-polyhedron (given by a specific triangulation) in
$\R^{2n}$ is algorithmically decidable for $n\ge 3$ \cite{MTW}.
This suggests seeking a higher-dimensional Kuratowski embeddability criterion
that would also provide an algorithm deciding the embeddability of the polyhedron.
Let us thus discuss amendments needed to fit Problem \ref{problem:main},
Conjecture \ref{conjecture:main}, and the Main Theorem in the algorithmic
framework.
We do not address issues of complexity of algorithms here.

1. The definition of a cell complex involves PL homeomorphism with
$|\partial\Delta^n|$ which is not a fully algorithmic notion by S. P. Novikov's
theorem (see \cite{CL}).
A standard workaround is to consider only cell complexes whose cells are shellable
(see \cite{Bj}).
This might potentially exclude some interesting examples, but the Main Theorem is
still valid and Conjecture \ref{conjecture:main} and Problem \ref{problem:main}
are still sensible.

1$'$. An alternative possibility is to generalize cell complexes to {\it circuit
complexes}, where the boundary of every cone is a circuit; we call an $n$-dimensional
poset $M$ an {\it $n$-circuit}, if $H^n(|M\but\cel p\cer|)=0$ for every $p\in M$.
Then part (ii) of the Main Theorem has to be amended; its proof can be reworked to
show that every $m$-dimensional dichotomial circuit complex has the integral homology
of $S^m$.
(Whether it must still be PL homeomorphic to $S^m$ is unknown to the author.)
Other parts of the Main Theorem hold without changes, and so do their proofs;
Conjecture \ref{conjecture:main} and Problem \ref{problem:main} stand.

2. The condition of being $(n-1)$-connected is not algorithmically decidable already
for $n=2$ by Adian's theorem (see \cite{CL}).
However, the proof of Theorem \ref{connected} produces, for each $n$-dimensional
cell complex $K$ such that $|K|$ embeds in $\R^{2n}$, $n\ge 3$, a cell
complex $K^+$ containing $K$, whose
\begin{roster}
\item $1$-skeleton lies in a collapsible subcomplex of the $2$-skeleton.
\end{roster}
On the other hand, if a cell complex $K$ satisfies (i) along with
\begin{roster}[1]
\item $H_i(|K|)=0$ for $i\le n-1$,
\end{roster}
then $|K|$ is $(n-1)$-connected by the Hurewicz theorem.
The modification of Conjecture \ref{conjecture:main} with the hypothesis that $|K|$
is $(n-1)$-connected replaced by the algorithmically decidable conditions
(i) and (ii) still makes sense.

3. The definition of an $h$-minor involves cell-like maps, which in turn
involve the non-algorithmic notion of contractibility.
An {\it ad hoc} solution is the following condition on the order-preserving map $f$:
instead of requiring the geometric realization $|f|$ to be cell-like, we require
the barycentric subdivision $f^\flat$ to be the composition of a sequence of
admissible edge contractions (see \S\ref{edge-minors}).
Drawbacks of this condition have been discussed in Example \ref{subdivision}, but
it is working so that the Main Theorem remains valid and Conjecture
\ref{conjecture:main} remains meaningful.

4. The fragmentary character of this subsection suggests that combinatorial
topology is badly missing a coherent algorithmic development of foundations,
which would include, {\it inter alia}, mutually compatible notions of an algorithmic
cell complex, of an algorithmic subdivision, and of an algorithmic cell-like map.
The author is working on such a project whose success is not yet obvious.

\section{Embeddings}

\subsection{Flores--Bier construction (simplified and generalized)}\label{embeddings}
In this subsection we prove the van Kampen--Flores--Gr\"unbaum--Schild non-embedding
theorem (see Theorem \ref{self-dual}) and its generalization to dichotomial posets.

\begin{definition}[Deleted product, join, prejoin]
Let $R$ be a poset and $P$, $Q$ be embedded in $R$.
The {\it deleted product} $P\otimes Q$ is the embedded poset in $P\x Q$,
consisting of all $(p,q)$ such that $\fll p\flr\cap\fll q\flr=\emptyset$.
The {\it deleted join} $P\circledast Q=C^*P\otimes C^*Q$ (where
$(\hat 0,\hat 0)$ is not subtracted like in the definition of
$P*Q$ since it is already missing here).
The {\it deleted prejoin} is not $P\oplus Q$ as one might guess, but
$P\oplus Q^*=(\mathcal P\sqcup\mathcal Q,\preceq)$, where $P=(\mathcal P,\le)$,
$Q=(\mathcal Q,\le)$, and $p\preceq q$ iff either

\begin{itemize}
\item $p,q\in P$ and $p\le q$, or

\item $p,q\in Q$ and $p\ge q$, or

\item $p\in P$, $q\in Q$ and $\fll p\flr\cap\fll q^*\flr=\emptyset$.
\end{itemize}

\noindent
Similarly to the non-deleted case,
$(P\oplus Q^*)^\flat\simeq (P\circledast Q)^\flat_{P\circledast\emptyset
\cup\emptyset\circledast Q}\subset P^\flat*Q^\flat$ (using that
$(Q^*)^\flat\simeq Q^\flat$).
In particular, $|P\oplus Q^*|\cong |P\circledast Q|$.
More specifically, it is not hard to see that $(P\circledast Q)^\flat$ is
a subdivision of
$(P\circledast Q)^\flat_{P\circledast\emptyset\cup\emptyset\circledast Q}\simeq
(P\oplus Q^*)^\flat$.

An advantage of the deleted prejoin $P\oplus Q^*$ over the would-be
$P\oplus Q$ is revealed already by the slightly more delicate isomorphism
$(P\oplus Q^*)^\#\simeq (P\circledast Q)^\#_{P\circledast\emptyset
\cup\emptyset\circledast Q}\subset P^\#*Q^\#$, where the star can
no longer be dropped.
\end{definition}

\begin{example}
Associativity of join implies
$(K\circledast K)*(L\circledast L)\simeq (K*L)\circledast (K*L)$.
Thinking of the $n$-simplex $\Delta^n$ as the join of $n+1$ copies of a point,
we get that $\Delta\circledast\Delta$ is isomorphic to the join of $n+1$ copies of
$pt\circledast pt=S^0$, which is the boundary of the $(n+1)$-dimensional
cross-polytope.
\end{example}

\begin{definition}[$m$-Obstructor]
A poset $Q$ will be called an {\it $m$-obstructor} if $|Q\circledast Q|$ with
the factor exchanging involution is $\Z/2$-homeomorphic to $S^{m+1}$ with
the antipodal involution.
\end{definition}

\begin{remark}
Much of what follows can be done without assuming the homeomorphism to be
$\Z/2$-equivariant or to be with the genuine sphere, because it follows from
the Smith sequences that every free involution on a polyhedral ($\Z/2$-)homology
$m$-sphere has cohomological ($\bmod 2$) sectional category equal to $m$
(see \cite{M2}).
\end{remark}

\begin{example}\label{5.1}
Let $[3]=\{0,1,2\}$ denote the three-point set.
It is easy to see that $[3]\oplus[3]\simeq\partial\Delta^2$, whence
$|[3]\circledast[3]|\cong|[3]\oplus[3]|\cong S^1$.
From an explicit form of this homeomorphism it is easily seen to be
equivariant.
Thus $[3]$ is a $0$-obstructor.
(In fact, this is a special case of a general fact, which will be given
a more conceptual explanation in Example \ref{5.5} and Theorem \ref{5.6}.)
\end{example}

\begin{lemma}\label{5.2} If $Q$ is an $m$-obstructor, then

(a) $|Q|$ does not embed in $S^m$;

(b) every embedding $g$ of $|Q|$ in $S^{m+1}$ is inequivalent with $hg$, where $h$
is an orientation-reversing self-homeomorphism of $S^{m+1}$.
\end{lemma}

Part (a) is due essentially to Flores \cite{F1} (see also \cite{Ro}).
The following proof of (a) occurs essentially in \cite{VS}.
The method of (b) yields another proof of (a), see \cite[Example 3.3]{M2}.

\begin{proof}[Proof. (a)]
If $|Q|$ embeds in $S^m$, then the cone $|C^*Q|$ over $|Q|$ embeds in $B^{m+1}$.
Since the homeomorphism $S^{m+1}\to |Q\circledast Q|=|C^*Q\otimes C^*Q|$
is equivariant, its composition with the projection
$|C^*Q\otimes C^*Q|\subset |C^*Q\x C^*Q|\to|C^*Q|$
does not identify any pair of antipodes in $S^{m+1}$.
The embedding $|C^*Q|\emb B^{m+1}$ then yields a map $S^{m+1}\to B^{m+1}$
identifying no pair of antipodes, which contradicts the Borsuk--Ulam theorem.
\end{proof}

For the proof of (b) we need the following definition.

\begin{definition}[Deleted product map] Given a poset $K$ and an embedding
$g\:|K|\emb\R^m$, define a map $|K\x K|\but\Delta_{|K|}\to S^{m-1}$ by
$(x,y)\mapsto \frac{G(x)-G(y)}{||G(x)-G(y)||}$.
This map is equivariant with respect to the factor exchanging involution on
$|K\x K|\but\Delta_{|K|}$ and the antipodal involution on $S^{m-1}$.
In particular, we have a $\Z/2$-map $\tilde g\:|K\otimes K|\subset
|K\x K|\but\Delta_{|K|}\to S^{m-1}$.
\end{definition}

\begin{proof}[Proof of \ref{5.2}(b)]
An embedding $g\:|Q|\emb S^{m+1}$ extends to an embedding $G\:|C^*Q|\emb B^{m+2}$.
Since $C^*Q\x C^*Q\simeq Q\circledast Q$, this yields a cohomology class
$\tilde G^*(\Xi)\in H^{m+1}(|Q\circledast Q|)$, where $\Xi\in H^{m+1}(S^{m+1})$ is
a generator (cf.\ \cite[\S3, subsection ``1-Paramater van Kampen obstruction'']{M2}).
The $\bmod 2$ reduction of $\tilde G^*(\Xi)$ is nonzero, since the Yang index of
the factor exchanging involution on $|Q\circledast Q|\cong S^{m+1}$ is $m+1$
(see \cite[\S3, subsection ``Unoriented van Kampen obstruction'']{M2}).
The mirror symmetry $r$ in the equator $S^m\subset S^{m+1}$ extends to the
mirror symmetry $R$ in $B^{m+1}\subset B^{m+2}$, which in turn corresponds
to $r$ (i.e.\ $\widetilde{RG}=r\tilde G$).
Since $r^*(\Xi)=-\Xi$ and $\tilde G^*(\Xi)$ is a nonzero integer, $RG$ is
inequivalent to $G$.
Hence $rg$ is inequivalent to $g$.
\end{proof}

Lemma \ref{5.2} and Example \ref{5.1} imply that the three-point set $[3]$ does
not embed in $S^0$ and knots in $S^1$.
However, this is not the end of the story.

\begin{lemma}\label{5.3} If $K$ is a $k$-obstructor and $L$ an $l$-obstructor, then
$K*L$ is an $(k+l+2)$-obstructor.
\end{lemma}

\begin{proof}
We are given $\Z/2$-homeomorphisms $S^{k+1}\to|K\circledast K|$
and $S^{l+1}\to |L\circledast L|$.
Their join is a $\Z/2$-homeomorphism
$S^{k+l+3}\to |(K\circledast K)*(L\circledast L)|$.
From the associativity of join
$(K\circledast K)*(L\circledast L)\simeq (K*L)\circledast (K*L)$,
which implies the assertion.
\end{proof}

Now from Example \ref{5.1} and Lemma \ref{5.3}, the join $[3]*[3]$ is
a $2$-obstructor.
Thus the graph $K_{3,3}\bydef|[3]*[3]|$ does not embed in $S^2$.
Moreover (by the proof of Lemma \ref{5.2}), the cone over $K_{3,3}$ does not
embed in $B^3$.
The same argument establishes

\begin{theorem}[{van Kampen \cite{vK}}]\label{5.4} The join of $n+1$ copy of
the three-point set $[3]$ does not embed in $S^{2n}$.
\end{theorem}

The more precise assertion that the join of $n+1$ copy of $[3]$ is an
$n$-obstructor is due to Flores \cite{F1} (see also \cite{Ro}).

\begin{definition}[Atoms]
An element $\sigma$ of a poset $P$ is called an {\it atom} of $P$, if
$\fll\sigma\flr=\{\sigma\}$.
The set of all atoms of $P$ will be denoted $A(P)$.
A poset $P$ is called {\it atomistic}, if every its element is the least upper bound
of some set of atoms of $P$.
It is easy to see that $A(\fll\sigma\flr)=A(P)\cap\fll\sigma\flr$.
Hence every element $\sigma$ of an atomistic poset is the least upper bound of
$A(\fll\sigma\flr)$.
(Beware that ``atomic'' has a different meaning in the literature on posets.)
\end{definition}

\begin{definition}[Dichotomial poset]
Let us call a poset $B$ {\it dichotomial}, if it is atomistic, and
for each $\sigma\in B$ the set of atoms $A(B)\but A(\fll\sigma\flr)$ has
the least upper bound, denoted $h(\sigma)$, in $B$.
In other words the latter condition says that there exists an $h(\sigma)\in B$
such that $\fll\sigma\flr\cap\fll h(\sigma)\flr=\emptyset$ and at the same time
$\fll\sigma\flr\cup\fll h(\sigma)\flr$ contains all the atoms of $B$.
Clearly, $h(h(\sigma))=\sigma$, so there is defined an involution $h\:B\to B$.
Clearly, the composition $H\:B\xr{h}B\xr{\id}B^*$ is order-preserving.
In particular, every dichotomial poset $B$ is isomorphic to its dual $B^*$.
\end{definition}

\begin{definition}[Combinatorial Alexander duality]
If $K$ is a subcomplex of a dichotomial poset $B$, then $B\but K$ is a dual
subcomplex of $B$.
Hence $H(B\but K)$ is a dual subcomplex of $B^*$.
Then $D(K)\bydef H(B\but K)^*$ is subcomplex of $B$.
In particular, $D(K)=K$ iff $K$ is a fundamental domain of the involution $h$;
in this case we say that $K$ is {\it self-dual} as a subcomplex of $B$.
\end{definition}

\begin{example}\label{5.5}
The boundary of every simplex $\partial\Delta^S$ is dichotomial:
$h(\sigma)=S\but\sigma$.
If $n<m$ and $K=(\Delta^m)^{(n)}$ is the $n$-skeleton of (the boundary of)
the $m$-simplex, then it is easy to see that $D(K)=(\Delta^m)^{(m-n-2)}$.
In particular, the $n$-skeleton of the $(2n+2)$-simplex is self-dual
in $\partial\Delta^{2n+2}$.
\end{example}

\begin{theorem}\label{5.6} If $K$ is a subcomplex of a dichotomial poset $B$,
then the deleted prejoin $K\oplus D(K)^*$ is isomorphic to $B$.
Moreover, if $K=D(K)$, then the isomorphism $f$ is anti-equivariant with
resect to the anti-involution $H$ and the factor-exchanging anti-involution $t$,
i.e.\ the following diagram commutes:
$$\begin{CD}
B@>f>>K\oplus K^*\\
@VHVV@VtVV\\
B^*@>f^*>>(K^*\oplus K)^*.
\end{CD}$$
\end{theorem}

The first assertion of Theorem \ref{5.6} implies the following ``Semi-combinatorial
Alexander duality'' theorems:

(i) $(K\circledast D(K))^\flat$ is a subdivision of $B^\flat$;

(ii) $|K\circledast D(K)|$ is homeomorphic to $|B|$.

In the case where $B$ is the boundary of a simplex (see Example \ref{5.5}),
(ii) was originally proved by T. Bier in 1991 (see \cite{Mat}) and reproved in
a more direct way by de Longueville \cite{dL}.
The assertion of (i) is a special case of a result of
Bj\"orner--Paffenholz--S\"ostrand--Ziegler \cite{BPSZ} (see also \cite{CD}).

\begin{proof} Let us map $K\oplus D(K)^*$ to $B$ by sending the first factor
via the inclusion $K\subset B$ and the second factor via
$H^{-1}\:H(B\but K)\to B\but K$.
The resulting bijection $f\:K\oplus D(K)^*\to B$ is clearly an order-preserving
embedding separately on the first factor and on the second factor.
If $p\in K$ and $q\in D(K)^*$, then $p\le q$ in $K\oplus D(K)^*$ iff
$\fll p\flr\cap\fll q^*\flr=\emptyset$.
The latter is equivalent to $A(\fll p\flr)\cap A(\fll q^*\flr)=\emptyset$ (since
$B$ is atomistic).
Now $f(p)=p$, and $A(\fll f(q)\flr)$ is the complement of $A(\fll q^*\flr)$ in $A(B)$.
Hence $A(\fll p\flr)\cap A(\fll q^*\flr)=\emptyset$ is equivalent to
$A(\fll f(p)\flr)\subset A(\fll f(q)\flr)$.
The latter is in turn equivalent to $\fll f(p)\flr\subset\fll f(q)\flr$ (since $B$
is atomistic), which is the same as $f(p)\le f(q)$.
Thus $f$ is an isomorphism.

In the case $K=D(K)$, the composition $K\oplus K^*\xr{f}B\xr{H}B^*$ is the
the identity on the second factor and is $H$ on the first factor.
The same is true of the composition
$K\oplus K^*\xr{t}(K^*\oplus K)^*\xr{f^*}B^*$.
\end{proof}

\begin{definition}[Dichotomial sphere] We say that a dichotomial poset $B$ with
its $*$-involution $H\:B\to B^*$ is a {\it dichotomial $m$-sphere} if
$|B|$ is $\Z/2$-homeomorphic to $S^m$ with the antipodal involution.
\end{definition}

Since $|K\circledast K|\cong|K\oplus K^*|$ equivariantly, we obtain

\begin{corollary}\label{5.7} Let $B$ be a dichotomial $(m+1)$-sphere.
Then every self-dual subcomplex of $B$ is an $m$-obstructor.
\end{corollary}

In particular, the self-dual subcomplex $(\Delta^4)^{(1)}$ of the
dichotomial $3$-sphere $\partial\Delta^4$ is a $2$-obstructor.
Thus the graph $K_5\bydef(\Delta^4)^{(1)}$ does not embed in the plane.
The same argument establishes

\begin{theorem}[{van Kampen \cite{vK}}]\label{5.8} The $n$-skeleton of
the $(2n+2)$-simplex does not embed in $S^{2n}$.
\end{theorem}

The more precise assertion that the $n$-skeleton of the $(2n+2)$-simplex is an
$n$-obstructor is due to Flores \cite{F2}.

Bringing in Lemma \ref{5.3}, we immediately obtain the following generalization of
Theorem \ref{5.8}, which includes Theorem \ref{5.4} as well:

\begin{theorem}[{Gr\"unbaum \cite{Gr}}]\label{5.9} Every $n$-dimensional join of
skeleta $(\Delta^{2n_i+2})^{(n_i)}$ does not embed in $S^{2n}$.
\end{theorem}

The more precise assertion that $n$-dimensional $2n$-obstructors include
$n$-dimensional joins of the form $F_{i_1}*\dots*F_{i_r}$, where each $F_i$ is
the $i$-skeleton of the $(2i+2)$-simplex has a converse.
Using matroid theory, Sarkaria has shown that these are the only simplicial
complexes among $n$-dimensional $2n$-obstructors \cite{Sa1}.
It would be interesting (in the light of Problem \ref{problem:main}) to extend
his methods to cell complexes.

\subsection{Construction of dichotomial spheres}

Our only example so far of a dichotomial poset is the boundary of simplex.
To get more examples, we can move in the opposite direction and
utilize Lemma \ref{5.3}.

\begin{definition}[Atomistic category] We say that an atomistic poset $K$ has
{\it atomistic category $\ge n$} if the set of atoms $A(K)$ is not contained in
a union of $n$ cones of $K$.
\end{definition}

\begin{theorem}\label{5.10} If $K$ is atomistic and has atomistic category
$\ne 1$, then $K\oplus K^*$ is dichotomial.
\end{theorem}

The proof shows that if $K$ is atomistic and has atomistic category $1$, then
$K\oplus K^*$ is not atomistic.

\begin{proof} The hard part is to show that $K\oplus K^*$ is atomistic.
Since $K$ is a subcomplex of $K\oplus K^*$, we have
$A(K)\subset A(K\oplus K^*)$.
If $\sigma\in K$, then $\sigma$ is the least upper bound in $K$ of some set
$S$ of atoms of $K$.
If some $\tau\in K^*$ is an upper bound of $S$, then
$\fll\tau^*\flr\cap S=\emptyset$.
Hence $\fll\tau^*\flr\cap\fll\sigma\flr=\emptyset$,
and so $\tau>\sigma$.
Thus $\sigma$ is the least upper bound of $S$ in $K\oplus K^*$.

If $\sigma\in K^*$, then it is an upper bound of
$S\bydef A(K)\but A(\fll\sigma^*\flr_K)$.
If $S$ is nonempty and $K$ is not a cone, then by the hypothesis $S$ has no
upper bound in $K$.
If $\tau\in K^*$ is another upper bound of $S$, then
$\fll\tau^*\flr\cap S=\emptyset$.
Hence $A(\fll\tau^*\flr)\subset A(\fll\sigma^*\flr)$.
Since $K$ is atomistic, $\fll\tau^*\flr\subset\fll\sigma^*\flr$.
Then $\tau^*\le\sigma^*$, so $\tau\ge\sigma$.
Thus $\sigma$ is the least upper bound of $S$ in $K\oplus K^*$.

In the case $S=\emptyset$ we have that $\sigma^*$ is an upper bound
of $A(K)$.
Since $\sigma^*$ is also the least upper bound of some subset of $A(K)$,
the least upper bound of $A(K)$ exists and equals $\sigma^*$.
Since $K$ is atomistic, this implies as above that the least upper bound of $K$
exists and equals $\sigma^*$.
Then $\sigma\le\tau$ for each $\tau\in K^*$ and also $\sigma$ is
incomparable with any element of $K$.
Thus $\sigma$ an atom of $K\oplus K^*$ and so the least upper bound of a set
of atoms.

If $S\ne\emptyset$ and $K$ is a cone, in symbols $K=\fll\hat 1\flr$, then
similarly to the above, $\sigma$ is the least upper bound of $S\cup\{\hat 1^*\}$
in $K\oplus K^*$.

Thus $K\oplus K^*$ is atomistic.
Moreover, we have proved that $A(K\oplus K^*)=A(K)$ if $K$ is not a cone,
and if $K$ is a cone, then $A(K\oplus K^*)=A(K)\cup\{\hat 1^*\}$.

Given a $\sigma\in K$, let $h(\sigma)=\sigma^*\in K^*$.
If $K$ is not a cone, then by the above
$A(\fll\sigma^*\flr)=A(K)\but A(\fll\sigma\flr)$.
If $K$ is a cone, then
$A(\fll\sigma^*\flr)=(A(K)\but A(\fll\sigma\flr))\cup\{\hat 1^*\}$.
In either case $A(\fll\sigma^*\flr)=A(K\oplus K^*)\but A(\fll\sigma\flr)$,
so $K\oplus K^*$ is dichotomial.
\end{proof}

If $P$ and $Q$ are posets, the atomistic category of $P*Q$ is clearly
the maximum of the atomistic categories of $P$ and $Q$.
The atomistic category of the $n$-skeleton of the $(2n+2)$-simplex is two,
since the $2n+3$ vertices of the simplex cannot be covered by two
$n$-simplices but can be covered by three.
Hence all the joins in Theorem \ref{5.9} have atomistic category two as well.

\begin{corollary}\label{5.11} If $K$ is a join of the $n_i$-skeleta of
the $(2n_i+2)$-simplices, then $K\oplus K^*$ is a dichotomial sphere.
\end{corollary}

More generally:

\begin{theorem}\label{construction} Every atomistic $2n$-obstructor that is
an $n$-dimensional cell complex is the $n$-skeleton of some dichotomial
$(2n+1)$-sphere.
\end{theorem}

\begin{proof} Let $K$ be the cell complex in question.
Suppose that $K$ is a union of two cells, $K=C\cup D$.
If $C\cap D=\emptyset$, then $K$ embeds in the $2n$-sphere
$\partial C*D\cup D*\partial C$, contradicting Lemma \ref{5.2}.
Else let $C'$ be the union of all cells of $C$ that are disjoint from $D$.
Then $C'\subset\partial C$.
On the other hand, since $K$ is atomistic, it embeds in $C'*D$.
Hence $K$ embeds in the $2n$-ball $\partial C*D$, again contradicting
Lemma \ref{5.2}.
\end{proof}

\subsection{Proof of Main Theorem (beginning)}

\begin{lemma}\label{4.3} Let $B$ be a dichotomial poset.
The following are equivalent:

(i) $|B|$ is homeomorphic to a sphere;

(ii) $B^\flat$ is a combinatorial manifold;

(iii) $B$ a cell complex.
\end{lemma}

\begin{proof} Let $h\:B\to B^*$ be the order-preserving isomorphism in
the definition of a dichotomial poset.
We note that (i)\imp(ii) is obvious.
\end{proof}

\begin{proof}[(iii)\imp(ii)] Let $\sigma\in B$.
Then $h(\cel\sigma\cer)=\fll h(\sigma)\flr^*$, so
$\partial^*\cel\sigma\cer\simeq(\partial\fll h(\sigma)\flr)^*$.
Let $[\sigma]\in B^\flat$ be the chain consisting of $\sigma$ only.
Then
$\lk([\sigma],B^\flat)\simeq
(\partial\fll\sigma\flr+\partial^*\cel\sigma\cer)^\flat$.
This is in turn isomorphic to
$(\partial\fll\sigma\flr)^\flat*(\partial^*\cel\sigma\cer)^\flat\simeq
(\partial\fll\sigma\flr)^\flat*(\partial\fll h(\sigma)\flr)^\flat$,
which is a combinatorial sphere.
Thus $B^\flat$ is a combinatorial manifold.
\end{proof}

\begin{proof}[(ii)\imp(iii)]
If $\sigma\in B$, and $\fll\tau\flr$ is a maximal simplex in
$\fll h(\sigma)\flr^\flat$, then the combinatorial sphere
$\lk(\tau,B^\flat)=(\partial\cel h(\sigma)\cer)^\flat$ is the barycentric
subdivision of $(\partial\cel h(\sigma)\cer)^*=h(\partial\fll\sigma\flr)$.
Hence $|\partial\fll\sigma\flr|$ is a sphere.
\end{proof}

\begin{proof}[(iii)$+$(ii)\imp(i)]
Suppose that $B^\flat$ is a combinatorial $m$-manifold, and let
$\fll\sigma\flr$ be an $m$-cell of $B$.
Then $h(\fll\sigma\flr)=\cel h(\sigma)\cer^*$.

Let $\tau$ be a $d$-cell of $B$ not contained in $\fll\sigma\flr$.
Then $h(\sigma)\in\fll\tau\flr$.
The intersection $\cel h(\sigma)\cer\cap\fll\tau\flr$ is the interval
$[\sigma,\tau]=\{\sigma\}+P+\{\tau\}$, where
$S=\lk(\sigma,\partial\fll\tau\flr)$.
Hence $[\sigma,\tau]^\flat\simeq\{[\sigma]\}*S^\flat*\{[\tau]\}$, where
$S^\flat\simeq\lk([\sigma],(\partial\fll\tau\flr)^\flat)$ is a combinatorial
$(d-2)$-sphere since $(\partial\fll\tau\flr)^\flat$ is a combinatorial
$(d-1)$-sphere and so in particular a combinatorial $(d-1)$-manifold.
Thus $\cel h(\sigma)\cer^\flat\cap\fll\tau\flr^\flat$ is a combinatorial $d$-ball
which meets $(\partial\fll\tau\flr)^\flat$ in a combinatorial $(d-1)$-ball.

Let $S$ be the subposet of $B$ consisting of all $\rho\in B$ that lie
neither in $\fll\sigma\flr$ nor in $\cel h(\sigma)\cer$.
Note that $h(S)=S^*$.
Since $B^\flat$ is a combinatorial $m$-manifold and $\fll\sigma\flr^\flat$ is
a combinatorial $m$-ball, $\fll S\flr^\flat$ is a combinatorial $m$-manifold
with two boundary components, $\fll S\flr^\flat\cap\fll\sigma\flr^\flat$ and
$\fll S\flr^\flat\cap\cel h(\sigma)\cer^\flat$.
By the above, $\fll S\flr^\flat$ meets $\fll\tau\flr^\flat$ in a
combinatorial $d$-ball $B_\tau$ meeting $(\partial\fll\tau\flr)^\flat$ in
a combinatorial $(d-1)$-ball $D_\tau$.
(We are using the relative combinatorial annulus theorem, which follows
from the uniqueness of regular neighborhoods, see e.g.\ \cite{RS}).

Collapsing each $|B_\tau|$ onto $|D_\tau|$ in an order of decreasing dimension, we
obtain a collapse of $|\fll S\flr|$ onto $|\fll S\flr\cap\fll\sigma\flr|$.
Hence $C\bydef\fll S\flr^\flat\cup\fll\sigma\flr^\flat$ is a regular neighborhood of
$\fll\sigma\flr^\flat$, and therefore $C$ is a combinatorial ball.
Thus $B^\flat$ is the union of two combinatorial balls $\cel h(\sigma)\cer^\flat$
and $C$ along an isomorphism of their boundaries.
By the Alexander trick, it must be a combinatorial sphere.
\end{proof}

\begin{lemma}\label{3.3b} If $K$ is a poset and $L$ is its $h$-minor, there exists
a $\Z/2$-map $H\:|L\circledast L|\to |K\circledast K|$.
Moreover, if $K$ is atomistic and $L$ is its proper $h$-minor, then $H$ is
non-surjective.
\end{lemma}

\begin{proof} It suffices to consider the case where $K$ maps onto $L$
by a non-injective order-preserving map $f$ such that $|f|$ is cell-like.

Let us define a map $g\:|L|\to|K|$ such that $g(|C|)\subset |f^{-1}(C)|$ for each
cone $C$ of $L$.
Assume that $g$ is defined on cones of dimension $<d$, and let $C$ be a
$d$-dimensional cone of $L$.
Since $f$ is cell-like, $|f^{-1}(C)|$ is contractible.
We then define $g$ on $|C|$ by mapping it via some null-homotopy of
$g(|\partial C|)$ in $|f^{-1}(C)|$.

Since $f$ is order-preserving, $f^{-1}(C)$ is a subcomplex of $K$ for each cone $C$
of $L$.
Since $g(|C|)\subset |f^{-1}(C)|$ for each cone $C$ of $L$, and $f$-preimages of
disjoint cones are disjoint, $g$ sends every disjoint pair of cones to a
pair of disjoint unions of cones.
Hence $(g*g)(|L\circledast L|)\subset |K\circledast K|$.

To prove the non-surjectivity, note that $g$ in fact sends every disjoint pair of
cones to a pair of unions of cones whose $|f|$-images are disjoint.
Since $f$ is order-preserving and non-injective, there exist distinct
$\sigma,\tau\in K$ whose $f$-images are the same.
Since $K$ is atomistic, $\sigma$ and $\tau$ may be assumed to be its atoms.
Then $(g*g)(|L\circledast L|)$ does not contain
$|\{\sigma\}\x\{\tau\}\cup\{\tau\}\x\{\sigma\}|$, which lies in
$|K\circledast K|$.
\end{proof}

\begin{lemma}[{\cite[Lemma 4.5]{M2}}]\label{quotient}
Let $K$ be an $n$-dimensional poset.
Then there exists a $\Z/2$-map $p\:|K\circledast K|\to S^0*|K\otimes K|$ such that
$p$ and $p/t$ induce isomorphisms on $i$-cohomology with arbitrary (possibly
twisted) coefficients for $i\ge n+2$ and an epimorphism for $i=n+1$.
\end{lemma}

\begin{proof} The map $p$ is the quotient map shrinking $|K*\emptyset|$ and
$|\emptyset*K|$ to points.
The relative mapping cylinder of $p$ collapses onto the pair of $(n+1)$-polyhedra
$(|CK\sqcup CK|,|K\sqcup K|)$, and the assertion follows.
\end{proof}

\begin{lemma}\label{is-linkless}
Let $K$ be an $n$-dimensional cell complex, $n\ge 2$.
An embedding $g\:|K|\emb B^{2n+1}$ is linkless iff for every pair of disjoint
subcomplexes $L$ and $M$ of $K$, the map $\tilde g|_{|L\x M|}\:|L\x M|\to S^{2n}$
is null-homotopic.
\end{lemma}

\begin{proof}
Let $P$ and $Q$ be disjoint subpolyhedra of $|K|$.
Let $P'$ be obtained by puncturing $K$ in some interior point of each $n$-cell $C$
of $K$ such that $|C|\not\subset P$; define $Q'$ similarly.
Then $P'$ deformation retracts onto $|L|\cup|K^{(n-1)}|$ and $Q'$ deformation
retracts onto $|M|\cup |K^{(n-1)}|$, where $L$ and $M$ are the maximal
subcomplexes of $K$ such that $|L|\subset P$ and $|M|\subset Q$.
Now the hypothesis implies that $\tilde g|_{P'\x Q'}$ is null-homotopic, and
therefore $\tilde g|_{P\x Q}$ is null-homotopic.
Then by the Haefliger--Weber criterion \cite{We}, $g|_{P\sqcup Q}$ is equivalent
to the embedding $h\:P\sqcup Q\emb B^{2n+1}$, obtained by combining $e_1g|_P$
and $e_2g|_Q$, where $e_1,e_2\:B^{2n+1}\to B^{2n+1}$ are embeddings with disjoint
images.
\end{proof}

\begin{lemma}\label{all-linkless} Let $K$ be an $n$-dimensional cell complex.
If $H^{2n+1}(|K\circledast K|)$ is cyclic, then every embedding of $|K|$
in $S^{2n+1}$ is linkless.
\end{lemma}

\begin{proof}
Let $L$ and $M$ be disjoint subcomplexes of $K$, and let $G=H^{2n}(|L\x M|)$.
Since $L$ and $M$ are disjoint, $H^{2n}(|L\x M\cup M\x L|)$ is isomorphic to
$G\oplus G$ and also is an epimorphic image of $H^{2n}(|K\otimes K|)$.
The latter group is cyclic by Lemma \ref{quotient} (as long as $n>0$), so $G$
must be zero.
If $n\ge 2$, the assertion follows from Lemma \ref{is-linkless} and the
Hopf classification theorem.
For $n=1$, the proof of Lemma \ref{is-linkless} works to show that if $P$ and $Q$
are disjoint subpolyhedra of $|K|$, then $H^2(P\x Q)=0$.
But then either $P$ or $Q$ must be a forest, and the assertion follows.
\end{proof}

\begin{lemma}\label{deleted} Let $L$ be an $n$-dimensional cell complex.
Then $H^{2n}(\overline{|L|},|L\otimes L|/t)=0$, where
$\overline{|L|}=(|L|\x |L|\but\Delta_{|L|})/t$.
\end{lemma}

\begin{proof}
If $C$ is a maximal cell of $L\x L$ that is not in $L\otimes L$, then $|C|$
meets the diagonal $\Delta_{|L|}$.
\end{proof}

Let us call an $n$-polyhedron $M$ an {\it $n$-circuit}, if $H^n(M\but\{x\})=0$
for every $x\in M$.

\begin{lemma}\label{proper minor} Let $K$ be an $n$-dimensional atomistic cell complex
such that $|K\circledast K|/t$ is a $(2n+1)$-circuit, and let $L$ be a proper
$h$-minor of $K$ that is a cell complex.
Then

(a) $|L|$ embeds in $S^{2n}$ if $n\ne 2$;

(b) every two embeddings (knotless if $n=1$) of $|K|$ in $S^{2n+1}$ become
equivalent when ``restricted'' to $|L|$.
\end{lemma}

It is interesting to compare this with the well-known result that
an $n$-polyhedron $P$ embeds in $S^{2n}$ if $P$ itself is an $n$-circuit
\cite{Sa-1}, \cite[8.2]{M2}.

\begin{proof} By Lemma \ref{3.3b} we have a non-surjective $\Z/2$-map
$f\:|L\circledast L|\to|K\circledast K|$.
Let $t$ denote the factor exchanging involution, let
$\phi\:|L\circledast L|/t\to\R P^\infty$ be its classifying map, and let
$\xi\in H^{2n+1}(\R P^\infty;\,\Z_T)\simeq\Z/2$ be the generator, where
$\Z_T$ denotes the twisted integer coefficients.
Since $\phi$ factors up to homotopy through the non-surjective map $f/t$
into the $(2n+1)$-circuit $|K\circledast K|/t$, we have $\phi^*(\xi)=0$.
By Lemma \ref{quotient}, there exists a $\Z/2$-map
$p\:|L\circledast L|\to S^0*|L\otimes L|$ such that $p/t$ induces an isomorphism
on $(2n+1)$-cohomology, as long as $n>0$.
Since $\phi$ factors up to homotopy as $p$ followed by a classifying map
$\psi\:(S^0*|L\otimes L|)/t\to\R P^\infty$ of $t$, we have $\psi^*(\xi)=0$.
It follows (see \cite[(5.1)]{CF}) that $\chi^*(\zeta)=0$, where
$\chi\:|L\otimes L|/t\to\R P^\infty$ is a classifying map of $t$ and
$\zeta\in H^{2n}(\R P^\infty;\,\Z)\simeq\Z/2$ is the generator.
By Lemma \ref{deleted}, $\chi_+^*(\zeta)=0$, where
$\chi_+\:\overline{|L|}\to\R P^\infty$ is a classifying map of $t$.
Now $\chi_+^*(\zeta)$ is the van Kampen obstruction $\theta(|L|)$; so $|L|$
embeds in $S^{2n}$ as long as $n\ge 3$ (see \cite{M2}).

Let $g,h\:|K|\to\R^{2n+1}$ be the given embeddings and $g',h'\:|L|\to\R^{2n+1}$
their ``restrictions''.
We have
$d(\tilde g',\tilde h')=d(\tilde gF,\tilde hF)=F^*d(\tilde g,\tilde h)$,
where the $\Z/2$-map $F\:|L\otimes L|\to|K\otimes K|$ is defined similarly
to $f$.
Now $(F/t)^*\:H^{2n}(|K\otimes K|/t;\,\Z_T)\to H^{2n}(|L\otimes L|/t;\,\Z_T)$
is the zero map, since it is equivalent to
$(f/t)^*\:H^{2n+1}(|K\circledast K|/t;\,\Z)\to H^{2n+1}(|L\circledast L|/t;\,\Z)$
via Thom-isomorphisms.
So $\tilde g'$ and $\tilde h'$ are equivalent as long as $n\ge 2$.

It remains to consider the case $n=1$.
Since $\phi^(\xi)=0$, there exists a $\Z/2$-map $|L\circledast L|\to S^2$
(this is similar to \cite[proof of 3.2]{M2}).
Hence by the Borsuk--Ulam theorem there exists no $\Z/2$-map
$S^3\to|L\circledast L|$.
Then by Lemma \ref{3.3b}, $L$ has no $2$-obstructor as a minor.
In particular, by the preceding observations, $L$ has no minor isomorphic to
$K_5$ or $K_{3,3}$.

Then by Wagner's version of the Kuratowski theorem, $|L|$ embeds in $S^2$.
Also, given a knotless embedding of $|K|$ in $S^3$, it is also linkless by
Lemma \ref{all-linkless}.
Its ``restriction'' to $|L|$ is also linkless and knotless, using Lemma
\ref{trivial} and the definition of a knotless embedding as an embedding $g$
such that $g(|C|)$ bounds an embedded disk in $S^3$ for every circuit $C$ of
the graph.
Hence the Robertson--Seymour--Thomas theorem implies that every two such
``restrictions'' are equivalent.
\end{proof}

\section{Linkless embeddings}

\subsection{Proof of Main Theorem (conclusion)}

Let $K=(\mathcal K,\le)$ be an $n$-dimensional poset, for instance,
an $n$-dimensional simplicial or cell complex.
If $S$ is a subcomplex of $K$, let $\bar S$ be the subcomplex of $K$
consisting of all cones of $K$ disjoint from $S$.
Note that $\bar{\bar S}=S$.

Consider the set $\lambda_K$ of all subcomplexes $S$ of $K$ such that
$H^n(|S|)\otimes H^n(|\bar S|)$ is nonzero (or, equivalently, not all maps
$|S\x\bar S|\to S^{2n}$ are null-homotopic).
Then $t_K\:S\mapsto\bar S$ is a free involution on $\lambda_K$.

\begin{lemma}\label{4.4'} If $L$ is the $n$-skeleton of a $(2n+2)$-dimensional
dichotomial cell complex $B$, then each $S\in\lambda_L$ is the boundary of
some $(n+1)$-cell of $B$.
\end{lemma}

\begin{proof} Let $K$ be the union of $L$ and a half of the $(n+1)$-cells of $B$,
with precisely one cell from each complementary pair.
Then $K$ is self-dual in $B$, so by Theorem \ref{5.6}, $K\oplus K^*$ is
isomorphic to $B$.
Hence by Lemma \ref{4.3}, $|K\oplus K^*|$ is a sphere, and therefore so is
$|K\circledast K|$.

By Lemma \ref{quotient},
$H^{2n+1}(|S*\bar S|)\simeq H^n(|S|)\otimes H^n(|\bar S|)\ne 0$, so by
the Alexander duality the complement to $|S*\bar S|$ in the sphere
$|K\circledast K|$ contains at least two connected components.
Since $S$ and $\bar S$ are disjoint, and $|\bar S*S|$ is connected (regardless of
whether any of $|S|$ and $|\bar S|$ is connected!), $|\bar S*S|$ lies one of
these open components.
The closure of that component is cellulated by a subcomplex $W$ of
$K\circledast K$, and the closure of the union of the remaining components by
a subcomplex $V$.

Since $|V|$ has nonempty interior, $V$ contains at least one $(2n+2)$-cell.
It must be of the form $C*D$, where $C$ is an $(n+1)$-cell of $K*\emptyset$ and
$D$ is an $n$-cell of $\emptyset*K$, or vice versa.
Let us consider the first case.

If $v$ is a vertex of $S$ not contained in $C$, the join $C*v$ lies in
$K\circledast K$.
Then it is a cell of $K\circledast K$, and therefore lies either in $V$ or in $W$.
Since $|\emptyset*v|$ lies in $|\bar S*S|$ and so in the interior of $|W|$,
we have $C*v\subset W$.
On the other hand, $C*\emptyset$ lies in $C*D$ and so in $V$.
Hence $C*\emptyset$ lies in $V\cap W$, which is a subcomplex of $S*\bar S$.
Thus $C\subset S$.
But this cannot be since $S$ is $n$-dimensional and $C$ is an $(n+1)$-cell.

Thus $C$, and therefore also $\partial C$, contains all vertices of $S$.
But $B$ is atomistic, so its subcomplex $K$ is atomistic, and we have
$S\subset\partial C$.
Since $H^n(|S|)\ne 0$ and $|\partial C|$ is an $n$-sphere, $\partial C=S$.

In the case where $V$ contains $C*D$, where $C$ is an $n$-cell of $K*\emptyset$
and $D$ is an $(n+1)$-cell of $\emptyset*K$, we can similarly show that
$\partial D=\bar S$.

This proves that either $S$ or $\bar S$ bounds a cell in $K$ (not just
in $B$).

Suppose that $S$ does not bound a cell in $K$.
Then by the above, $\bar S$ bounds a cell $D$ in $K$.
Let us amend $K$ by exchanging $D$ with its complementary $(n+1)$-cell,
and let $K'$ denote the resulting subcomplex of $B$.
Then $\bar S$ does not bound a cell in $K'$.
Hence by the above, $S$ bounds a cell in $K'$.
\end{proof}

Let $\hat K$ be the poset $(\mathcal K\cup\lambda_K,\preceq)$, where $p\preceq q$
iff either $p,q\in K$ and $p\le q$ or $p\in K$, $q\in S$ and $p\in q$.
Given a section $\xi\:\lambda_K/t_K\to\lambda_K$ of the double covering
$\lambda_K\to\lambda_K/t_K$, we have the subcomplex $K_\xi$ of $\hat K$ obtained
by adjoining to $K$ all elements of $\xi(\lambda_K/t_K)$.
Note that under the hypothesis of Lemma \ref{4.4'}, $\hat K$ and $K_\xi$ are
cell complexes.

\begin{example} If $K=K_6$, then $K_\xi$ can be chosen to be
the semi-icosahedron, and if $K$ is the Petersen graph, then $K_\xi$ can
be chosen to be the semi-dodecahedron (see \S\ref{semi}).
\end{example}

\begin{lemma}\label{3.3a} If $K$ is a poset and $L$ is its (proper) $h$-minor, then
each $L_\zeta$ is a (proper) $h$-minor of some $K_\xi$.
\end{lemma}

\begin{proof} Suppose that $f\:K\to L$ is an order-preserving surjection such that
$|f|$ is cell-like.
Given subcomplexes $M$, $N$ of $L$ such that $t_L(M)=N$ and $M,N\in\lambda_L$,
we have that $M'\bydef f^{-1}(M)$ and $N'\bydef f^{-1}(N)$ belong to $\lambda_K$ since
$|f|$ restricts to a homotopy equivalence $|f|^{-1}(P)\to P$ for every
subpolyhedron $P$ of $|L|$.
Up to relabelling, we may assume that $\zeta(\{M,N\})=M$.
Then we set $\xi(\{M',t_K(M')\})=M'$ and $\xi(\{N',t_K(N')\})=t_K(N')$.
If $t_K(N')\ne M'$, then we label $M'$ as a ``primary'' element of
the image of $\xi$.

This defines $\xi$ on a subset of $\lambda_K/t_K$.
We extend it to the remaining elements arbitrarily, and do not introduce
any new labels.
The subcomplex $K_\xi'$ of $K_\xi$ obtained by adjoining to $K$ all primary
elements of $\xi(\lambda_K/t_K)$ then admits an order-preserving surjection $g$ onto
$L_\zeta$ such that $g$ is an extension of $f$, and $|g|$ is cell-like.

The remaining case where $L$ is a subcomplex of $K$ is similar (and easier).
\end{proof}

Let $K\hat\oplus K^*=K_\xi\oplus (K_\xi)^*$.
This poset does not depend on the choice of $\xi$, since it is isomorphic to
$\hat K\oplus K^*$ (and also to $K\oplus (\hat K)^*$).
Moreover, it is easy to see that the involution on $K\hat\oplus K^*$ also
does not depend on the choice of $\xi$.
Lemma \ref{4.4'} has the following

\begin{corollary}\label{4.4} If $L$ is the $n$-skeleton of a
$(2n+2)$-dichotomial cell complex $B$, then $L\hat\oplus L^*$ is
anti-equivariantly isomorphic to $B$.
\end{corollary}

\begin{proof} By Lemma \ref{4.4'}, $L_\xi$ is a cell complex for every $\xi$.
Hence $L_\xi$ is isomorphic to a $K$ as in the proof of Lemma \ref{4.4'},
and therefore by Theorem \ref{5.6}, $L_\xi\oplus L_\xi^*$ is anti-equivariantly
isomorphic to $B$.
\end{proof}

Let $K\hat\circledast K$ be the union of $K\circledast K$ and cones of
the form $C(S*\bar S)$, where $S\in\lambda_K$.
In more detail, $K\hat\circledast K=(P\cup\lambda_K,\preceq)$, where $(P,\le)$ is
the deleted join $K\circledast K=C^*K\otimes C^*K\subset K*K$, and $p\preceq q$
iff $p,q\in P$ and $p\le q$ or $p=(\sigma,\tau)\in P$, $q\in\lambda_K$, and
either

\begin{itemize}
\item $\sigma,\tau\ne\hat 1$ and $\sigma\in q$ and $\tau\in t_K(q)$, or

\item $\sigma=\hat 1$ and $\tau\in t_K(q)$, or

\item $\tau=\hat 1$ and $\sigma\in q$.
\end{itemize}

Since $C(S*\bar S)$ is subdivided by $(CS)*\bar S$ (and also by $S*(C\bar S)$),
we obtain that $K_\xi\circledast K_\xi$ is a subdivision of $K\hat\circledast K$,
for each $\xi$.
Then from Lemmas \ref{3.3b} and \ref{3.3a} we get the following

\begin{corollary}\label{3.3}
If $L$ is an $h$-minor of a poset $K$, there exists a $\Z/2$-map
$H\:|L\hat\circledast L|\to |K\hat\circledast K|$.
Moreover, if $K$ is atomistic and $L$ is its proper $h$-minor, then $H$ is
non-surjective.
\end{corollary}

Since $K_\xi\circledast K_\xi$ (which does depend on $\xi$) is
a subdivision of $K\hat\circledast K$, we also obtain that
$|K\hat\circledast K|\cong |K\hat\oplus K^*|$, equivariantly.
It is easy to find such a $\Z/2$-homeomorphism that does not depend on
the choice of $\xi$.

\begin{theorem}\label{3.0} Let $K$ be an $n$-dimensional cell complex.
$|K|$ is linklessly embeddable in $S^{2n+1}$ iff there exists a $\Z/2$-map
$|K\hat\circledast K|\to S^{2n+1}$.
\end{theorem}

This is similar to \cite[Theorem 4.2]{M2} but we give a more detailed proof here.
A part of the proof is also parallel to a part of the proof of
Lemma \ref{proper minor}.

\begin{proof}
Given a linkless embedding $g\:|K|\emb S^{2n+1}$, we may extend it to an embedding
$G\:|C^*K|\emb B^{2n+2}$ and pick null-homotopies
$H_S\:|C^*S\sqcup C^*\bar S|\to B^{2n+2}$ of the links $g(|S\cup\bar S|)$ such that
$H_S(|C^*S|)\cap H_S(|C^*\bar S|)=\emptyset$ and
$H_S^{-1}(S^{2n+1})=|S\sqcup\bar S|$, for each $S\in\lambda_K$.
Since each $H_S$ is homotopic through maps $H_t:|C^*S\sqcup C*\bar S|\to B^{2n+2}$
satisfying $H_t^{-1}(S^{2n+1})=|S\sqcup\bar S|$ (but not necessarily
$H_t(|C^*S|)\cap H_t(|C^*\bar S|)=\emptyset$) to the restriction of $G$,
the deleted product maps $\tilde G\:|C^*K\otimes C^*K|\to S^{2n+1}$ and
$\tilde H_S|_{\dots}\:|C^*S\x C^*\bar S\sqcup C^*\bar S\x C^*S|\to S^{2n+1}$ have
equivariantly homotopic restrictions to $|S*\bar S\sqcup\bar S*S|$.
Hence $G$ extends to an equivariant map $|K\hat\circledast K|\to S^{2n+1}$
(using the homeomorphism $|C^*(S*\bar S)|\cong |C(S*\bar S)|$).

Conversely, let $K\hat\otimes K$ be the union of $K\otimes K$ and cones of
the form $C(S\x\bar S)$, where $S\in\lambda_K$.
Then the quotient of $|K\hat\circledast K|$ obtained by shrinking $|K*\emptyset|$
and $|\emptyset*K|$ to points is $\Z/2$-homeomorphic to the suspension
$S^0*|K\hat\otimes K|$.
Then similarly to Lemma \ref{quotient} there exists a $\Z/2$-map
$S^0*|K\hat\otimes K|\to S^{2n+1}$.
By the equivariant Freudenthal suspension theorem, $|K\otimes K|$ admits
a $\Z/2$-map $\phi$ to $S^{2n}$ whose restriction to $|S\x\bar S|$
is null-homotopic for each $S\in\lambda$.
By Lemma \ref{deleted}, $\phi$ extends equivariantly over $|K\x K|\but\Delta_{|K|}$.
Hence by the Haefliger--Weber criterion (\cite{We}; alternatively, see
\cite[3.1]{M2} and Remark \ref{erratum2} below), $|K|$ admits an embedding $g$
into $B^{2n+1}$ such that $\tilde g\:|K\otimes K|\to S^{2n}$ is $\Z/2$-homotopic
to $\phi$.
In particular, the restriction of $\tilde g$ to $S\x\bar S$ is null-homotopic
for each $S\in\lambda_K$.

When $n=1$, this implies that any two disjointly embedded circles in
$g(|K|)$ have zero linking number.
Then by Theorem \ref{RST}(b), $|K|$ admits a linkless embedding in $S^3$.

When $n\ge 2$, the embedding $g$ itself is linkless by Lemma \ref{is-linkless}.
\end{proof}

\begin{corollary}[of the proof]
Let $K$ be the $n$-skeleton of a $(2n+2)$-dimensional dichotomial cell complex.
Then every embedding of $|K|$ in $S^{2n+1}$ contains a link of two
disjoint $n$-spheres with an odd linking number.
\end{corollary}

\begin{proof} By Lemma \ref{4.4'}, every $S\in\lambda_K$ cellulates an $n$-sphere.
If $g\:|K|\emb\R^{2n+1}$ is an embedding that for every $S\in\lambda_K$ links
$|S|$ and $|\bar S|$ with an even linking number, then by proof of the ``only if''
part of Theorem \ref{3.0}, $|K\hat\circledast K|$ admits an $\Z/2$-map to
$S^{2n+2}$ of even degree.
However no even degree self-map of a sphere can be equivariant with respect
to the antipodal involution.
(For the resulting self-map of the projective space lifts to the double
cover, so must send $w_1$ of the covering to itself.
However the top power of $w_1$ goes to zero if the degree is even.)
\end{proof}

\begin{corollary} Let $K$ be an $n$-dimensional atomistic cell complex such
that $|K\hat\circledast K|$ is a $(2n+2)$-circuit.
Then $|K|$ admits a linkless embedding in $S^{2n+1}$.
\end{corollary}

The proof is similar to a part of the proof of Lemma \ref{proper minor}.

\begin{proof}
By Corollary \ref{3.3} we have a non-surjective $\Z/2$-map
$f\:|L\hat\circledast L|\to|K\hat\circledast K|$.
Let $t$ denote the factor exchanging involution, let
$\phi\:|L\hat\circledast L|/t\to\R P^\infty$ be its classifying map, and let
$\xi\in H^{2n+2}(\R P^\infty;\,\Z)\simeq\Z/2$ be the generator.
Since $\phi$ factors up to homotopy through the non-surjective map $f/t$
into the $(2n+2)$-circuit $|K\hat\circledast K|/t$, we have $\phi^*(\xi)=0$.
Then $|L\hat\circledast L|$ admits a $\Z/2$-map into $S^{2n+1}$ (see
\cite[proof of 3.2]{M2}), and the assertion follows from Thorem \ref{3.0}.
\end{proof}

\subsection{Equivariant homotopy of the Petersen family}\label{Petersen topology}

\begin{example}[the $n$-skeleton of the $(2n+3)$-simplex]\label{3.4}
Let $\Delta$ be the $(2n+3)$-simplex and $L$ its $n$-skeleton.
For each pair of complementary $(n+1)$-simplices, choose one, and let $L_\xi$
be the union of $L$ with the chosen $(n+1)$-simplices.
Then $L_\xi$ is self-dual, hence by Theorem \ref{5.6},
$|L_\xi\circledast L_\xi|$ is $\Z/2$-homeomorphic to the $(2n+2)$-sphere.
On the other hand, $|L_\xi\circledast L_\xi|$ is also $\Z/2$-homeomorphic to
$|L\hat\circledast L|$ by the above.
Hence by the Borsuk--Ulam theorem the latter admits no $\Z/2$-map to $S^{2n+1}$.
Thus by Theorem \ref{3.0}, $L$ is not linklessly embeddable in $S^{2n+1}$.

In particular, $K_6$ is not linklessly embeddable in $S^3$.
This was first proved by Conway and Gordon, see also \cite{Sac1}, \cite{Sac2};
the $n$-dimensional generalization is proved in \cite[Corollary 1.1]{LS},
\cite{Ta} and \cite[Example 4.7]{M2} (in all cases, by methods different from
the above).
\end{example}

\begin{definition}[$\nabla$Y-exchanges]
A graph $H$ is said to be obtained by a $\nabla$Y-{\it exchange} from a graph $G$,
if $G$ contains a subgraph $\nabla$ isomorphic to $\partial\Delta^2$, and $H$ is
obtained from $G$ by removing the three edges of $\nabla$ and instead adjoining
the triod Y$\bydef\Delta^0*(\nabla)^0$.
We shall call a $\nabla$Y-exchange {\it dangerous} if $G$ contains no circuit
disjoint from $\nabla$; and {\it allowable} if either (i) it is dangerous, or
(ii) it is non-dangerous and $G$ contains precisely one circuit $C$ disjoint
from $\nabla$, and $H$ contains no circuit disjoint from $C$.
\end{definition}

Let us consider an allowable $\nabla$Y-exchange $G\rightsquigarrow H$.
The obvious map $|\nabla|\to|$Y$|$, sending the barycenters of the edges to
the cone point and keeping $|\nabla^{(0)}|$ fixed, yields a map $f\:|G|\to |H|$,
which sends every pair of disjoint cells of $G$ to a pair of disjoint subcomplexes
of $H$.

Given an $H_\zeta$, let us choose a $G_\xi$ as follows.
In the dangerous case we have a canonical bijection between $\lambda_G$ and
$\lambda_H$, and we let $\xi$ correspond to $\zeta$ under this bijection.
In the non-dangerous case, $\lambda_G$ contains two additional circuits, $\nabla$
and $C$, and we set $\xi(\{\nabla,C\})=\nabla$.
Conversely, every $\xi$ satisfying the latter property uniquely determines a
$\zeta$ for a non-dangerous exchange, and every $\xi$ whatsoever uniquely
determines a $\zeta$ for a dangerous exchange.

With such $\zeta$ and $\xi$, our $f$ extends to an $f_\zeta\:|G_\xi|\to |H_\zeta|$,
which in the non-dangerous case is collapsible.
This $f_\zeta$ still sends every pair of disjoint cells of $G_\xi$ to a pair of
disjoint subcomplexes of $H_\zeta$.
Hence $f_\zeta*f_\zeta$ restricts to
a $\Z/2$-map $F_\zeta\:|G_\xi\circledast G_\xi|\to |H_\zeta\circledast H_\zeta|$.
In the non-dangerous case, $F_\zeta$ is collapsible, and so is an equivariant
homotopy equivalence.

\begin{example}[Petersen family]\label{3.5}
There is the following diagram of $\nabla$Y-exchanges, which uses the notation
of Fig.\ 1:

\begin{equation}
\begin{CD}
@.\!\!\!\!K_{3,3,1}\!\!\!\!@.@.@.\\
@.\hskip25pt\searrow\hskip-25pt@.@.@.\\
K_6@>>>\Gamma_7@>>>\Gamma_8@>>>\Gamma_9@>>>P\\
@.\hskip15pt_{\mathbf !}\searrow\hskip-30pt@.@.@.\\
@.@.\!\!\!\!K_{4,4}\!\but\! e\!\!\!\!@.@.
\end{CD}\tag{$*$}
\end{equation}
\medskip

It can be verified by inspection that all these $\nabla$Y-exchanges are allowable,
and that no further $\nabla$Y- or Y$\nabla$-exchange (allowable or not) can be
applied to any graph in this diagram; thus these seven graphs are the Petersen
family graphs.
The only dangerous $\nabla$Y-exchange is marked by an ``\,$\mathbf !$\,''.

Write $L=K_6$, and select $L_\xi$ so that it contains a quadruple of $2$-cells
whose pairwise intersections contain no edges.
For instance, the hemi-icosahedron (see \S\ref{semi}) will do.
Then the horizontal sequence of $\nabla$Y-exchanges in ($*$) can be made
along these four $2$-cells, which compatibly defines a $G_\xi$ for each $G$
in this sequence.
These also uniquely determine compatible $G_\xi$ for $G=K_{3,3,1}$ and
$K_{4,4}\but$(edge).
Then by the above, for each $G$ in the Petersen family, we obtain
a $\Z/2$-map $|L\hat\circledast L|\to |G\hat\circledast G|$
(using the invertibility of the $\Z/2$-homotopy equivalence in the case of
$K_{3,3,1}$).
On the other hand, by the preceding example $|L\hat\circledast L|$ is
$\Z/2$-homeomorphic to $S^4$, hence by the Borsuk--Ulam Theorem it admits no
$\Z/2$-map to $S^3$.
Thus $|G\hat\circledast G|$, for each $G$ in the Petersen family, admits no
$\Z/2$-map to $S^3$.

The (easy) ``only if'' direction in Theorem \ref{3.0} now implies that none of
the graphs of the Petersen family is linklessly embeddable in $S^3$, a fact
originally observed by Sachs \cite{Sac1}, \cite{Sac2}.
Corollary \ref{3.3} then implies the easy direction in
the Robertson--Seymor--Thomas Theorem \ref{RST}(b).
\end{example}

\begin{remark}\label{hemi-exchanges} If we do choose $L_\xi$ to be
the hemi-icosahedron, then all the non-dangerous $\nabla$Y and
Y$\nabla$-exchanges in ($*$) yield cellulations of $\R P^2$ by the complexes
$G_\xi$, in particular, by the hemi-dodecahedron $P_\xi$.
\end{remark}

\begin{remark}\label{K44-nonembed}
We shall see in the next subsection that the $F_\xi$ corresponding to all
the non-dangerous $\nabla$Y-exchanges in ($*$) can be approximated by equivariant
homeomorphisms (for those in the horizontal line this also follows from the
Cohen--Homma Theorem \ref{Cohen-Homma}).
This cannot be the case for the dangerous one.
Indeed if $G=K_{4,4}\but$(edge), then $G\circledast G$ alone contains
a join of two triods, which is non-embeddable in $S^4$ since the link of its
central edge is $K_{3,3}$, which does not embed in $S^2$.
\end{remark}

\subsection{Combinatorics of the Petersen family} \label{transforms}

Following van Kampen \cite{vK2} and McCrory \cite{Mc}, we define a {\it combinatorial
manifold} as a cell complex $K$ such that $K^*$ is also a cell complex
(cf.\ \cite{M3}).
If additionally both $K$ and $K^*$ are atomistic, we say that $K$ is
an {\it $\alpha$-combinatorial} manifold.
Now an {\it homology ($\alpha$-)combinatorial manifold} as an (atomistic)
homology cell complex $K$ such that $K^*$ is also an (atomistic) homology
cell complex.
Here a poset $K$ is called an {\it homology cell complex} if for each $\sigma\in K$
there exists an $m$ such that $H_i(|\partial\fll\sigma\flr|)\simeq H_i(S^m)$,
with integer coefficients.

Let $K$ be a poset and let $C\in K$.
We say that $K$ is Y$\nabla${\it -transformable} at $C$ if $C$ is covered by
precisely three elements $D_1$, $D_2$, $D_3$ such that
\begin{roster}
\item $\partial\fll D_i^*\flr
\subset\partial\fll D_{i+1}^*\flr\cup\partial\fll D_{i-1}^*\flr$, and
\item $\fll D_{i-1}\flr\cap\fll D_{i+1}\flr=\fll C\flr$
\end{roster}
for each $i$ (addition $\bmod 3$).

\begin{lemma}\label{homcomb} Suppose that $K$ is an homology combinatorial
$\alpha$-manifold and $C\in K$ is covered by precisely three elements
$D_1$, $D_2$, $D_3$.
Then $K$ is Y$\nabla$-transformable at $C$.
\end{lemma}

The proof of property (ii) uses the hypothesis that $K^*$ is atomistic.

\begin{proof}[Proof. (i)] If $E>D_i$, pick an $E'<E$ such that $E'$ covers $D$.
The homology combinatorial sphere $\partial E$ is a pseudo-manifold, so
$\fll C\flr$ is contained in precisely two homology $(c+1)$-cells in $\partial E$.
One of them is $\fll D_i\flr$, so the other can only be $\fll D_{i+1}\flr$ or
$\fll D_{i-1}\flr$.

{\it (ii).}
Since the link $\partial\cel C\cer^L$ of $C$ in
$L\bydef\fll D_1\flr\cup\fll D_2\flr\cup\fll D_3\flr$ is the $3$-point set,
which is not a homology sphere, $K$ cannot be of dimension $c+1$, and therefore
$D_i$ is not maximal.
Hence $C^*$ and $D_i^*$ are not atoms of $\fll C^*\flr$ and so all the atoms of
$\fll C^*\flr$ lie in $\fll C^*\flr\but\{C^*,D_i^*\}$.
By (i) the latter equals $\fll D_{i-1}^*\flr\cup\fll D_{i+1}^*\flr$.
Hence $D_{i-1}^*$ and $D_{i+1}^*$ do not simultaneously belong to any cell of $K^*$
other than $\fll C^*\flr$.
\end{proof}

\begin{definition}[Y$\nabla$-transform]
Suppose that $K$ is a poset that is Y$\nabla$-transformable at some $C\in K$.
Then we define the Y$\nabla${\it -transform} of $K$ at $C$ to be the poset $K_C$
that has one element $\hat B$ for each $B\in K$ and satisfies the following for
each $A,A'\notin\{C,D_1,D_2,D_3\}$:

\begin{itemize}
\item $\hat C>\hat D_i$, $i=1,2,3$;

\item $\hat A<\hat D_i$ iff $A<D_{i+1}$ or $A<D_{i-1}$;

\item $\hat A<\hat C$ iff $A<D_1$ or $A<D_2$ or $A<D_3$;

\item $\hat A>\hat D_i$ iff $A>D_{i+1}$ and $A>D_{i-1}$;

\item $\hat A>\hat C$ iff $A>D_1$ and $A>D_2$ and $A>D_3$;

\item $\hat A>\hat A'$ iff $A>A'$.
\end{itemize}

This definition immediately implies that $(K_C)^*$ is Y$\nabla$-transformable
at $(\hat C)^*$.
A straightforward Boolean logic using that $K$ is Y$\nabla$-transformable at $C$
shows further that the Y$\nabla$-transform of $(K_C)^*$ at $(\hat C)^*$ is
isomorphic to $K^*$.

We say that a poset $L$ is $\nabla$Y{\it -transformable} at $E\in L$ if $L^*$ is
Y$\nabla$-transformable at $E^*$.
In that case the $\nabla$Y{\it -transform} of $L$ at $E$ is defined to be
the poset $L^E\bydef ((L^*)_{E^*})^*$.
\end{definition}

\begin{example} Let $M$ be a triangulation of the Mazur contractible $4$-manifold.
Let $M_1\cup M_2\cup M_3$ be the union of three copies of $M$ identified
along an isomorphism of their boundaries.
Then each $M_i\cup M_{i+1}$ (addition $\bmod 3$) is the double of $M$ and hence is
homeomorphic to $S^4$.
We glue it up by a $5$-cell $C_{i-1}$.
Then $C_1\cap C_2=M_3$, so $C_1\cup C_2$ is simply-connected and acyclic, hence
contractible.
Its boundary $M_1\cup M_2\cong S^4$, so $(C_1\cup C_2)\cup C_3$ is a homotopy
sphere, hence the genuine $5$-sphere.
We glue it up by a pair of $6$-cells $D_1$, $D_2$.
The resulting cell complex $K$ is homeomorphic to $S^6$, and hence its dual is
also a cell complex (see \cite{M3}).

Now $D_1^*$ is covered by precisely three elements $C_1^*$, $C_2^*$,
$C_3^*$ of $K^*$.
Let $L=(K^*)_{D_1^*}$ and let $D=\widehat{D_1^*}$.
Then $D^*$ is a cell of $L^*$ such that $|\partial\fll D^*\flr|\cong|\partial M|$,
is a nontrivial homology $3$-sphere.

Note that $D_1^*$ is of codimension $6$ in $K^*$.
\end{example}

\begin{lemma}\label{homology mfld}
Let $K$ be a homology $\alpha$-combinatorial manifold,
and suppose that $C\in K$ is covered by precisely three elements of $K$.
Then the {\rm Y}$\nabla$-transform $K_C$ is a homology
$\alpha$-combinatorial manifold.

If additionally $K$ is an $\alpha$-combinatorial manifold, and
$\fll C\flr$ is of codimension $\le 5$ in $K$, then $K_C$ is an
$\alpha$-combinatorial manifold.
\end{lemma}

\begin{proof}
{\it Case 1.} Let us first prove that $|R|$ and $|A_i|$ are homology spheres, where
$R=\partial\fll(\hat C)^*\flr$ and $A_i=\partial\fll(\hat D_i)^*\flr$.
Let $S=\partial\fll C^*\flr$ and let $B_i=\fll D_i^*\flr$.
If $\fll C\flr$ has codimension $k$ in $K$, then the homology $(k-1)$-sphere
$|S|$ is the union of three homology $(k-1)$-balls $|B_i|$.
Since $|\partial B_1|$ is a homology $(k-2)$-sphere, by the Alexander duality
the closures of its complementary domains, including
$|B_2\cup B_3|$, are homology balls.
Then the Mayer--Vietoris sequence implies that $|B_2\cap B_3|$ is a homology
$(k-2)$-ball, and in particular $|\partial(B_2\cap B_3)|$ is a homology
$(k-3)$-sphere.
Now $\partial(B_2\cap B_3)=B_1\cap B_2\cap B_3\simeq R$, and
$(B_2\cap B_3)\cup_R CR\simeq A_1$.

If $k\le 5$, then the homology $(k-3)$-sphere $|R|$ must be a genuine sphere;
and if additionally $K$ is an $\alpha$-combinatorial manifold, then
$|\partial B_2|$ is a genuine $(k-2)$-sphere, so by the Sch\"onflies and Alexander
theorems, $|R|$ bounds in it a genuine $(k-2)$-ball $|B_2\cap B_3|$.

{\it Case 2.} Next, each $\fll\hat D_i\flr$ is an (homology) cell since it is
isomorphic to the union of the (homology) cells $\fll D_{i-1}\flr$ and
$\fll D_{i+1}\flr$ whose intersection is the (homology) cell $\fll C\flr$.
Also, $\fll\hat C\flr$ is an (homology) cell since it is isomorphic to
the union of the (homology) cells $\fll\hat D_{i-1}\flr$ and $\fll\hat D_{i+1}\flr$
whose intersection is isomorphic to $(\partial\fll D_i\flr)\but\{C\}$.
The latter cellulates an (homology) ball, which is the closure of the
complement to the homology ball $|\fll C\flr|$ in the homology sphere
$|\partial\fll D_i\flr|$.

{\it Case 3.} Finally, let us consider a poset $K_C'$ that has one element
$\check B$ for
each $B\in K$ and additional elements $\check F_1$, $\check F_2$, $\check F_3$
and $\check E_{12}$, $\check E_{13}$, $\check E_{21}$, $\check E_{23}$,
$\check E_{31}$, $\check E_{32}$, and satisfies the following for each
$A,A'\notin\{C,D_1,D_2,D_3\}$ and all $i,j\in\{1,2,3\}$, $i\ne j$:

\begin{itemize}
\item $\check D_i>\check C>\check F_j$ and $\check D_i>\check E_{ij}>\check F_j$;

\item $\check C$ is incomparable with $\check E_{ij}$;

\item $\check A<\check E_{ij}$ iff $\check A<\check D_i$ iff $A<D_i$;

\item $\check A<\check F_j$ iff $\check A<\check C$ iff $A<C$;

\item $\check A>\check E_{ij}$ iff $\check A>\check F_j$ iff
$A>D_{j+1}$ and $A>D_{j-1}$;

\item $\check A>\check D_i$ iff $\check A>\check C$ iff
$A>D_1$ and $A>D_2$ and $A>D_3$;

\item $\check A>\check A'$ iff $A>A'$.
\end{itemize}

It is easy to see that there exist subdivisions $\phi\:K_C'\to K_C$ and
$\psi\:(K_C')^*\to K^*$.
We note that $K_C'$ and $(K_C')^*$ are non-atomistic posets whose cones are cells,
except for one cone $\fll C\flr$ in $K_C'$ isomorphic to $(\partial\fll C\flr)+T$
and one cone $\fll E^*\flr$ in $(K_C')^*$ isomorphic to
$(\partial\fll(\hat C)^*\flr)+T$, where $T$ denotes the cone over
$(\Delta^2)^{(0)}$.
Then each of $\phi^*\:(K_C')^*\to (K_C)^*$ and $\psi^*\:K_C'\to K$ is
obtained by taking the quotient by the copy of $T$ followed by three
elementary zippings.
Either from this, or because $\phi$ and $\psi$ are subdivisions, $|\phi^*|=|\phi|$
and $|\psi^*|=|\psi|$ are collapsible.
On the other hand, for each $A\notin\{C,D_1,D_2,D_3\}$ we have
$(\psi^*)^{-1}(\fll A\flr)=\fll\check A\flr=\phi^{-1}(\fll\hat A\flr)$ and
$\psi^{-1}(\fll A^*\flr)=\fll(\check A)^*\flr=(\phi^*)^{-1}(\fll(\hat A)^*\flr)$.
Since a collapsible map is a homology equivalence, we obtain that each cone
of $K_C$ other than $\fll\hat C\flr$ and $\fll\hat D_i\flr$, and each cone
of $(K_C)^*$ other than $\fll(\hat C)^*\flr$ and $\fll(\hat D_i)^*\flr$ is
a homology cell.
Then by the above, $K_C$ and $K_C^*$ are homology cell complexes.
It is clear that they are atomistic.

Moreover, for each $A\in K$ other than $C$, $D_1$, $D_2$, $D_3$, if
$|\partial\fll A^*\flr|$ (resp.\ $|\partial\cel A^*\cer|$) is a sphere, then so
is the subdivided $|\partial\fll(\check A)^*\flr|$ (resp.\
$|\partial\cel(\check A)^*\cer|$), and hence by the Cohen--Homma Theorem
\ref{Cohen-Homma} so is $|\partial\fll(\hat A)^*\flr|$ (resp.\
$|\partial\cel(\hat A)^*\cer|$).
\end{proof}

\begin{definition}[($\nabla$,Y)-transform]
Suppose that $K$ is a dichotomial homology cell complex and $(A,\bar A)$ is
a pair of its complementary elements such that $A$ is covered by precisely
three elements and $\bar A$ covers precisely three elements.
Consider first the Y$\nabla$-transform $K_{\bar A}$.
Then the $\nabla$Y-transform $(K_{\bar A})^{\hat A}$ is again a dichotomial
homology cell complex.
We shall call it the ($\nabla$,Y){\it -transform} of $K$ along $(A,\bar A)$,
and also the (Y,$\nabla$){\it -transform} of $K$ along $(\bar A,A)$.
\end{definition}

\begin{example}\label{4.5} The ($\nabla$,Y)-transform of $\partial\Delta^3$ along
$(\Delta^2,\Delta^0)$ is isomorphic to $\partial\Delta^3$.
\end{example}

\begin{example}[$K_{3,3}$ revisited]\label{4.6}
The ($\nabla$,Y)-transform of $\partial\Delta^4$ along $(\Delta^2,\Delta^1)$ is
the dichotomial complex with $1$-skeleton $K_{3,3}$.
\end{example}

\begin{example}[Petersen family except $K_{4,4}\but$(edge)]\label{4.8}
It is easy to see that the horizontal sequence of $\nabla$Y-exchanges in
($*$) lifts to a sequence of ($\nabla$,Y)-transforms of dichotomial cellulations
of $S^4$, along pairs of the types $(\Delta^2,D_i)$ with $i=3,4,5,6$, where $D_i$
is a $2$-cell with $i$ edges in the boundary.
The inverse (Y,$\nabla$)-transforms are along pairs of the type (vertex, $4$-cell).
The Y$\nabla$-exchange $\Gamma_8\rightsquigarrow K_{3,3,1}$ lifts to
a similar (Y,$\nabla$)-transform, whose inverse ($\nabla$,Y)-transform is of
the type $(\Delta^2,D_4)$.
Thus all $\Gamma_i$'s, $P$, and $K_{3,3,1}$ are the $1$-skeleta of their
respective dichotomial cellulations of $S^4$.
\end{example}

\begin{example}[$K_{4,4}\but$(edge) lives in a dichotomial $5$-sphere]\label{4.9}
Let us think of $K_6$ as the $1$-skeleton of the boundary of a top cell of
$\partial\Delta^6$.
The sequence of two $\nabla$Y-exchanges leading to $K_{4,4}\but$(edge)
lifts to a sequence of two ($\nabla$,Y)-transforms of dichotomial cellulations
of $S^5$, along pairs of the types $(\Delta^2,\Delta^3)$ and
$(\Delta^2,\Sigma^3)$, where $\Sigma^3$ is isomorphic to a top cell of
the dichotomial $3$-complex with $1$-skeleton $K_{3,3}$.
Thus $K_{4,4}\but$(edge) is the $1$-skeleton of the boundary of a top cell of
a dichotomial $5$-sphere.
\end{example}

\begin{definition}[Pseudo-Y$\nabla$-transform]
Suppose that $K$ is a poset and $\fll D_1\flr$, $\fll D_2\flr$, $\fll D_3\flr$ are
its $(c+1)$-dimensional cones such that no two of them have a common $c$-dimensional
subcone.
The {\it pseudo-{\rm Y}$\nabla$-transform} of $K$ at $(D_1,D_2,D_3)$ is the poset $L$
that has one element $\hat B$ for each $B\in K$, and an additional element $\hat C$,
and satisfies the following (same as above) for each $A,A'\notin\{D_1,D_2,D_3\}$:

\begin{itemize}
\item $\hat C>\hat D_i$, $i=1,2,3$;

\item $\hat A<\hat D_i$ iff $A<D_{i+1}$ or $A<D_{i-1}$;

\item $\hat A<\hat C$ iff $A<D_1$ or $A<D_2$ or $A<D_3$;

\item $\hat A>\hat D_i$ iff $A>D_{i+1}$ and $A>D_{i-1}$;

\item $\hat A>\hat C$ iff $A>D_1$ and $A>D_2$ and $A>D_3$;

\item $\hat A>\hat A'$ iff $A>A'$.
\end{itemize}

We also say that $L^*$ is obtained by the $\nabla$Y{\it -transform} of $K^*$ at
$(D_1^*,D_2^*,D_3^*)$.

A {\it pseudo-{\rm ($\nabla$,Y)}-transform} of a dichotomial poset $K$ along
$(A_1,A_2,A_3;\,\bar A_1,\bar A_2,\bar A_3)$ is now defined similarly to
a ($\nabla$,Y)-transform.
\end{definition}

\begin{example}\label{4.10} It is easy to see that the dangerous
$\nabla$Y-exchange $\Gamma_7\rightsquigarrow K_{4,4}\but$(edge) lifts to
the pseudo-($\nabla$,Y)-transform along $(A_1,A_2,A_3;\,\bar A_1,\bar A_2,\bar A_3)$,
where $A_i$ are the edges of the $\nabla$ and so $\bar A_i$ are their complementary
$3$-cells.
The resulting dichotomial poset $Q$ has $K_{4,4}\but$(edge) as its $1$-skeleton,
and $|Q|$ is homeomorphic to the double mapping cone of the map
$|\nabla*\nabla|\to |$Y$*$Y$|$.
In particular, shrinking the pair of joins of two triods to points, we obtain
a collapsible $\Z/2$-map $|Q|\to S^4$.
Alternatively, it suffices to shrink $|$Y$*$Y$|\cong I*|K_{3,3}|$ to $pt*|K_{3,3}|$.
\end{example}

\section{Proofs for \S\ref{intro}}

\subsection{Collapsible and cell-like maps}\label{collapsing}

The following result is due to T. Homma (see \cite{Bry} for references and for
corrections of Homma's other proofs in there) and M. M. Cohen \cite{Co}.

\begin{theorem}[Cohen--Homma]\label{Cohen-Homma} If $M$ is a closed manifold,
$X$ is a polyhedron, and $f\:M\to X$ is a collapsible map, then $X$ is homeomorphic
to $M$.
\end{theorem}

We include a proof modulo Cohen's simplicial transversality lemma (see
\cite{Co}, \cite{M3})

\begin{proof} Let us triangulate $f$ by a simplicial map $F\:L\to K$.
Write $n=\dim M$.
Given a simplex $\sigma$ of $K$, let $\sigma^*$ denote its dual cone,
which is the join of the barycenter $\hat\sigma$ with the derived link
$\partial\sigma^*$ (see e.g.\ \cite[2.27(6)]{RS}).
By Cohen's simplicial transversality lemma, if $\sigma$ is a simplex
of $K$ of dimension $n-i$, then $\sigma^*_f\bydef f^{-1}(\sigma^*)$ is an $i$-manifold
with boundary $\partial\sigma^*_f\bydef f^{-1}(\partial\sigma^*)$, and this manifold
collapses onto the point-inverse $f^{-1}(\hat\sigma)$.
The latter is collapsible by our hypothesis, so $\sigma^*_f$ is an $i$-ball.

Let $\sigma$ be a maximal simplex of $K$.
Then $\partial\sigma^*=\emptyset$ and so $\sigma^*_f$ is a closed manifold.
Since it also must be a ball of some dimension, this dimension is zero.
Thus all maximal simplices of $K$ have dimension $n$; and if $\sigma$ is
such a simplex, then $f$ restricts to a homeomorphism $\sigma^*_f\to\sigma^*$
between a $0$-ball and a point.

Let $M_k$, resp.\ $X_k$ be the union of the manifolds $\sigma^*_f$, resp.\ of
the cones $\sigma^*$, for all simplices $\sigma$ of $K$ of dimension $\ge n-k$.
Assume inductively that there is a homeomorphism $g\:M_k\to X_k$ that sends
each $\sigma^*_f$ into $\sigma^*$.
Given a simplex $\tau^{n-k-1}$ of $K$, by the assumption $g$ restricts to
a homeomorphism $h_\tau\:\partial\tau^*_f\to\partial\tau^*$.
Since $\tau^*_f$ is a ball, it is homeomorphic to the cone over $\partial\tau^*_f$,
so we can extend $h_\tau$ to a homeomorphism $\tau^*_f\to\tau^*$.
At the end of this inductive construction lies a homeomorphism $M=M_n\to X_n=X$.
\end{proof}

\begin{corollary} If $f\:P\to Q$ is a collapsible map between polyhedra,
and $P$ embeds in a manifold $M$, then $Q$ embeds in $M$.
\end{corollary}

\begin{proof}
Let us identify $P$ with a subpolyhedron of $M$.
Then the adjunction space $M\cup_{f}Q$ is a polyhedron, and the quotient map
$F\:M\to M\cup_{f}Q$ is collapsible.
Its target contains $Q$ and is homeomorphic to $M$ by the preceding theorem.
\end{proof}

\begin{theorem} If $\phi\:P\to Q$ is a cell-like map between $n$-polyhedra, and
$P$ embeds in an $m$-manifold $M$, where $m\ge n+3$, then $Q$ embeds in $M$.
\end{theorem}

This can be deduced from known results, albeit in an awkward way.
Using the theory of decomposition spaces \cite[23.2, 5.2]{Dav}, the quotient map
$M\to M\cup_\phi Q$ can be approximated by topological homeomorphisms, and
therefore $Q$ {\it topologically} embeds in $M$.
Once again using the codimension three hypothesis, we can approximate
the topological embedding by a PL one (see \cite[5.8.1]{DV}).

Below we give a proof avoiding topological embeddings; but we shall also see
that they arose not accidentally, for we barely avoid using the $4$-dimensional
topological Poincar\'e conjecture (=Freedman's theorem).

\begin{proof} We use the notation in the proof of the Cohen--Homma theorem.
We may assume that $\phi$ is triangulated by a simplicial map $\Phi\:B\to A$,
where $A$ and $B$ are subcomplexes of $K$ and $L$ and $\Phi$ is the
restriction of $F$.
Then the $\sigma^*_f$ are no longer balls but contractible manifolds with spines
of codimension $\ge 3$.
Because of the latter, their boundaries are simply connected and hence homotopy
spheres.
By the Poincar\'e conjecture those of dimensions $\ge 5$ are genuine spheres
and then what they bound are genuine balls of dimensions $\ge 6$.
The $3$-dimensional and $4$-dimensional contractible manifolds are also
genuine balls since their codimension $\ge 3$ spines must be collapsible.
Of the $5$-dimensional contractible manifolds with $2$-dimensional spines
we only note that they would be genuine balls if either the Andrews--Curtis
conjecture or the $4$-dimensional PL Poincar\'e conjecture were known to hold.

Write $Q_k=X_k\cap Q$, and for a simplex $\sigma$ in $A$, let
$\sigma^*_A=\sigma^*\cap A$ be the dual cone of $\sigma$ in $A$, which is
the join of the barycenter $\hat\sigma$ with the derived link
$\partial\sigma^*_A$.
We claim that for each $k$ there is an embedding $g\:Q_k\to M_k$ that sends
each $\sigma^*_A$ into $\sigma^*_f$.
There is nothing to prove for $k\le 2$.
For $k=3$, we may have an $(n-3)$-simplex $\sigma$ in $A$, and then we need
to embed its barycenter $\sigma^*_A$ into the $3$-ball $\sigma^*_f$; this is
not hard.
For $k=4$ and an $(n-4)$-simplex $\sigma$ in $A$, we have the finite set
$\partial\sigma^*_A$ embedded in the $3$-sphere $\partial\sigma^*_f$, and we need
to extend this embedding to an embedding of the $1$-polyhedron $\sigma^*_A$ into
the $4$-ball $\sigma^*_f$; this is also not hard.
For $k=5$ and an $(n-5)$-simplex $\sigma$ in $A$, we have the $1$-polyhedron
$\partial\sigma^*_A$ embedded in the homotopy $4$-sphere $\partial\sigma^*_f$,
and we need to extend this embedding to an embedding of the $2$-polyhedron
$\sigma^*_A$ into the homotopy $5$-ball $\sigma^*_f$.
This can be done: the boundary embedding extends to a map $\sigma^*_A\to\sigma^*_f$
since $\sigma^*_f$ is contractible, and then this map of a $2$-polyhedron in
a $5$-manifold can be approximated by an embedding by general position.
For $k\ge 6$ we simply use conewise extension just like we did for $k=1$ and
in the proof of the Cohen--Homma lemma.
Eventually we obtain an embedding $Q=Q_n\emb M_n=M$.
\end{proof}

\begin{corollary} Let $f\:X\to Y$ be a map between $n$-polyhedra, where $X$
embeds in an $m$-manifold $M$.
Then $Y$ embeds in $M$ if either

(a) the point-inverse of the barycenter of every $k$-simplex is collapsible for
$k\ge m-n-1$ and collapses onto an $(m-n-2-k)$-polyhedron for $k\le m-n-2$; or

(b) $m-n\ge 3$, and $f$ is fiberwise homotopy equivalent to a $g\:Z\to Y$
whose nondegenerate point-inverses lie in a subpolyhedron of dimension $\le m-n-2$.
\end{corollary}

\begin{proof}[Proof. (a)] Let us triangulate $\phi$ by a simplicial map
$\phi\:K\to L$.
Since a collapse may be viewed as a map with collapsible
point-inverses (see \cite[\S8]{Co2}), $f$ is the composition of a collapsible map
$g\:X\to Z$ and a map $h\:Z\to Y$ whose non-degenerate point-inverses lie
in the subpolyhedron $Q\bydef h^{-1}(|L^{(m-n-2)}|)$ of dimension $\le m-n-2$.
By Theorem \ref{minors embed}(a), $Z$ embeds in $M$.
Since the mapping cylinder $MC(h|_Q)$ is of dimension $\le m-n-1$, by general
position, this embedding extends to an embedding of $Z\cup MC(h|_Q)$ in $M$.
The point-inverses of the projection $Z\cup MC(h|_Q)\to Y$ are cones, so
applying theorem \ref{minors embed}(a) once again, we obtain that $Y$ embeds
in $M$.
\end{proof}

\begin{proof}[(b)] Let $\phi\:Z\to X$ be the given fiberwise homotopy equivalence
over $Y$ and let $Q$ be the given subpolyhedron.
Since $MC(g|_Q)$ is of dimension $\le m-n-1$, the original embedding of
$X$ in $M$ extends to an embedding of $X\cup_\phi MC(g|_Q)$ in $M$.
On the other hand, $X\cup_\phi MC(g|_Q)$ is a fiberwise deformation retract
of $X\cup_\phi MC(g)$, which in turn fiberwise deformation retracts onto $Y$.
Hence every point-inverse of the projection $X\cup_\phi MC(g|_Q)\to Y$
is contractible.
Thus $Y$ embeds in $M$ by Theorem \ref{minors embed}(b).
\end{proof}

\subsection{Edge-minors}\label{edge-minors2}

\begin{proof}[Proof of Lemma \ref{Steinitz}] Let $K'$ be a subdivision of
the given $2$-complex $K$.
We call a simplex $\sigma$ of $K'$ ``old'' if $|\sigma|=|\tau|$ for some
simplex $\tau$ of $K$, and ``new'' otherwise.
A subcomplex of $K'$ is said to be ``old'' if it consists entirely of old
simplices, and ``new'' otherwise.

Let $\partial\Delta^3$ be a missing tetrahedron in $K'$.
Then $|\partial\Delta^3|$ has to be triangulated by a subcomplex of $K$,
and therefore $\partial\Delta^3$ is old.

Let $\partial\Delta^2$ be a missing triangle in $K'$.
If $|\partial\Delta^2|$ lies in $|K^{(1)}|$, then it has to be triangulated by
a subcomplex of $K^{(1)}$, and therefore $\partial\Delta^2$ is old.
Thus every new missing triangle $\partial\Delta^2$ in $K'$ has a vertex in
the interior of the combinatorial ball $\sigma'$ for some $2$-simplex $\sigma$
of $K$, and therefore $\partial\Delta^2$ is itself contained in $\sigma'$.

Given a $2$-simplex $\sigma$ of $K$, every missing triangle $\partial\Delta^2$
of $K'$ contained in $\sigma'$ bounds a combinatorial disk in $\sigma'$.
If two such disks intersect in at least one $2$-simplex, then one is contained in
the other.
Let $D_0$ be an innermost such disk in $\sigma'$.
Since $\partial D_0$ is a missing triangle in $K'$, $D_0$ is not a $2$-simplex,
and so contains at least one edge $\tau$ in its interior.
By the minimality of $D_0$, $\tau$ is not contained in any missing triangle in $K'$.
Since $\partial D_0$ is a complete graph, at least one vertex of $\tau$ lies in
the interior of $D_0$, and hence in the interior of $\sigma'$.
Then we may contract $e$.

It remains to consider the case where $K'$ contains no new missing triangle.
If $\tau$ is a new edge with at least one vertex in the interior of $\sigma'$
for some $2$-simplex $\sigma$ of $K$, we may contract $\tau$.
If $\tau$ is a new edge such that $|\tau|$ lies in $|K^{(1)}|$, we may contract
$\tau$.
If there are no new edges of these two types, $K'=K$.
\end{proof}

\begin{proof}[Proof of Addendum \ref{self-dual-addendum}]
By Theorem \ref{self-dual} and assertion (2) above we may assume that $L$ is
obtained from $K$ by a single edge contraction.
Let $\sigma=\rho_1*\rho_2$ be the contracted edge.

Let us first consider the case where $r=1$.
Then the opposite $(m-2)$-simplex $\tau$ to $\sigma$ in $\Delta^m$ is not
contained in $K$; so $K$ lies in $\sigma*\partial\tau$.
The simplicial map $K\to L$ extends to an edge contraction
$f\:\sigma*\tau\to\rho*\tau$ of the entire $\Delta^m$, where $\rho$ is
the $0$-simplex.
We have $L=f(K)\subset\rho*\partial\tau$.
This is a combinatorial $(m-2)$-ball, and $|\rho*\partial\tau|$ lies in
$S^{m-2}\bydef|\tau\cup\rho*\partial\tau|$, which proves the assertion of (a).
Moreover, it is not hard to see that the composition $g$ of the inclusion
$|L|\subset S^{m-2}$ and the embedding $S^{m-2}\emb S^{m-1}$,
$\rho\mapsto\rho_1$, is equivalent to the embedding $|L|\emb S^{m-1}$
induced by $j$.
Choosing $h$ to be the reflection $S^{m-1}\to S^{m-1}$ in $S^{m-2}$, we have
$hg=g$, which proves the assertion of (b).

The case where $\sigma$ lies in a single factor of the join $K_1*\dots*K_r$,
say in $K_1$, reduces to the case just considered, by observing that
$K_2*\dots*K_r$ lies in the combinatorial sphere
$\partial\Delta^{m_2}*\dots*\partial\Delta^{m_r}$ of dimension $m-1-m_1$.

The remaining cases similarly reduce to the case where $r=2$ and
$\sigma=\sigma_1*\sigma_2$, where $\sigma_i\in K_i$.
Let $\tau_i$ be the opposite $(m_i-1)$-simplex to $\sigma_i$ in $\Delta^{m_i}$;
then $K_i$ lies in $\sigma_i*\partial\tau_i$.
Similarly to the above, we have $L\subset\rho*\partial\tau_1*\partial\tau_2$.
This is a combinatorial $(m-2)$-ball, which lies in the combinatorial
$(m-2)$-sphere
$\tau_1*\partial\tau_2\cup\rho*\partial\tau_1*\partial\tau_2$, etc.
\end{proof}

\section{Embeddability is commensurable with linkless
embeddability}\label{commensurability}

This section is concerned with relations between embeddings and linkless
embeddings.
It somewhat diverges from the combinatorial spirit of this paper, and instead
contributes to a central theme of \cite{M2}; thus it is best viewed as
an addendum to \cite{M2}.
This said, the main result of this section also gives a geometric view of
some examples and constructions mentioned in the introduction, and might
contribute to an initial groundwork for a proof of a higher-dimensional
Kuratowski(--Wagner) theorem.

An embedding $g$ of a graph $|G|$ in $S^3$ is called {\it panelled} if for every
circuit $Z$ in the graph, $g$ extends to an embedding of $|G\cup CZ|$.

\begin{lemma} \label{panels} (a) {\rm (Robertson--Seymour--Thomas \cite{RST})}
A graph $G$ admits a panelled embedding in $S^3$ iff $|G|$ admits an embedding $g$
in $S^3$ such that for every two disjoint circuits $C$, $C'$ in the graph,
$g$ links $|C|$ and $|C'|$ with an even linking number.

(b) \cite[Lemma 4.1]{M2} An embedding of a graph in $S^3$ is panelled iff it
is linkless and knotless.
\end{lemma}

\begin{remark}\label{erratum} The proof of (b) in \cite{M2} contains a minor
inaccuracy: in showing that the embedding is linkless it explicitly treats
the splitting of two disjoint subgraphs only in the case where one of them
is connected; however, the argument works in the general case without any
modifications.
\end{remark}

We also need a higher-dimensional analogue of Lemma \ref{panels}.
To this end, we recall from \S\ref{embeddings} that we call an $n$-polyhedron $M$
an {\it $n$-circuit}, if $H^n(M\but\{x\})=0$ for every $x\in M$.
This implies that $H^n(M)$ is cyclic (since it is an epimorphic image
of $H^n(M,M\but\{x\})$, where $x$ lies in the interior of an $n$-simplex
of some triangulation of $M$) and that $H^n(P)=0$ for every proper subpolyhedron
$P$ of $M$.
An {\it oriented} $n$-circuit $M$ is endowed with a generator $\xi_M$ of $H^n(M)$.

\begin{lemma}\label{cohomology} Let $P$ be a polyhedron.
For every nonzero class $x\in H^n(P)$ there exists a singular $n$-circuit
$f\:M\to P$ such that $f^*(x)\ne 0$.
Moreover, $f^*(x)$ is of the same order as $x$.
\end{lemma}

We note that a more convenient and far-reaching geometric view of cohomology exists
(see \cite[\S2]{M2}, \cite{Fe}, \cite{BRS}), but that is not what we need here.

\begin{proof}
By the universal coefficient formula, $x$ maps to some $h\:H_n(P)\to\Z$.
If $x$ is of infinite order, then $h$ is nontrivial.
Let us pick a $y\in H_n(P)$ with $h(y)\ne 0$.
Then $y$ is representable by a singular oriented $\Z$-pseudo-manifold
$f\:M\to P$ such that $y=f_*([M])$ (see \cite[Theorem 1.3.7]{Fe}).
Then $hf_*([M])\ne 0$, so $f^*(h)\:H_n(M)\to\Z$ is nontrivial.
Hence by the naturality in the UCF, $f^*(x)\ne 0$.

If $x$ is of a finite order, $m$ say, then $h$ is trivial.
Hence $x$ comes from some extension $\Z\mono G\epi H_{n-1}(P)$ of order $m$ in
$\Ext(H_{n-1}(P),\Z)$.
Then there exists a $z\in H_{n-1}(P)$ of order $m$ that is covered by
an element $\hat z\in G$ of infinite order.
In particular, $z$ is the Bockstein image of some $y\in H_n(P;\,\Z/m)$.
This $y$ is representable by a singular oriented $\Z/m$-pseudo-manifold
$f\:M\to P$ such that $y=f_*([M])$ (see \cite[Chapter III]{BRS}, \cite{Do}).
Then by the naturality of the Bockstein homomorphism, $z=f_*(\beta[M])$.
Hence the induced extension $\Z\mono f^*G\epi H_{n-1}(M)$ is such that
$\beta[M]$ is covered by an element of $f^*G$ of infinite order, which maps
onto $\hat x$.
Since $\beta[M]$ is of order $m$, the induced extension is itself of order $m$
in $\Ext(H_{n-1}(M),\Z)$.
Hence by the naturality in the UCF, $f^*(x)$ is of order $m$.
\end{proof}

The {\it linking number} of an oriented singular $m$-circuit $f\:M\to\R^{m+n+1}$
and an oriented singular $n$-circuit $g\:N\to\R^{m+n+1}$ with disjoint images is
the degree of the composition
$M\x N\xr{f\x g}\R^{m+n+1}\x\R^{m+n+1}\but\Delta_{\R^{m+n+1}}\simeq S^{m+n}$,
which is given by the formula $(m,n)\mapsto\frac{m-n}{||m-n||}$.
This degree $\lk(f,g)$ lives in the cyclic group
$H^{m+n}(M\x N)\simeq H^m(M)\otimes H^n(N)$.
The linking number of an unoriented $m$-circuit and an unoriented
$n$-circuit in $S^{m+n+1}$ is well-defined up to a sign.

By an {\it $n$-circuit with boundary} we mean any $n$-polyhedron $M$ along with
an $(n-1)$-dimensional subpolyhedron $\partial M$ such that $M/\partial M$ is
a genuine $n$-circuit.
If additionally $\partial M$ is an $(n-1)$-circuit and the coboundary map
$H^{n-1}(\partial M)\to H^n(M)$ is an isomorphism, then we say that $\partial M$
{\it bounds} $M$.

\begin{lemma} \label{circuits}
Let $P$ be an $n$-polyhedron, $n\ge 2$, and $g\:P\emb S^{2n+1}$ an embedding.

(a) $g$ is linkless iff every pair of singular $n$-circuits in $P$ with disjoint
images have zero linking number under $g$.

(b) $g$ is linkless iff for every $(n+1)$-circuit $Z$ with boundary, every
$f\:\partial Z\to P$ and every function $\phi\:\partial Z\to I$, the embedding
$g\x\id_I\:P\x I\emb S^{2n+1}\x I$ extends to an embedding of
$P\x I\cup_{f\x\phi}Z$ in $S^{2n+1}\x I$.
\end{lemma}

\begin{proof}[Proof. (a)] Let $Q$ and $R$ be disjoint subpolyhedra of $g(P)$.
Since $Q$ is of codimension $\ge 3$ in the sphere, $M\bydef S^{2n+1}\but Q$ is
simply-connected.
Hence by the Alexander duality and the Hurewicz theorem $M$ is $(n-1)$-connected,
and $\pi_n(M)\simeq H_n(M)\simeq H^n(Q)$.
Therefore
$$H^n(R;\,\pi_n(M))\simeq H^n(R;\,H^n(Q))\simeq H^n(R)\otimes H^n(Q)\simeq
H^n(R\x Q).$$
The first obstruction to null-homotopy of the inclusion $R\subset M$ can be
identified under this string of isomorphisms with the image of the generator
of $H^{2n}(S^{2n})$ in $H^{2n}(R\x Q)$.
If this image is nonzero, it is of the form $\sum r_i\otimes q_i$, where
each $r_i\in H^n(R)$ and each $q_i\in H^n(Q)$, and each $r_i\otimes q_i$ is
nonzero.
Then by Lemma \ref{cohomology} there are singular $n$-circuits $f\:V\to Q$ and
$g\:W\to R$ such that $f^*(r_1)$ has the same order as $r_1$, and $g^*(q_1)$
has the same order as $q_1$.
Then $f^*(r_1)\otimes g^*(q_1)$ is nonzero, and also it equals $\lk(f,g)$,
contradicting the hypothesis.
Hence by obstruction theory $R$ is null-homotopic in $M$.
Then by engulfing (see e.g.\ \cite[Lemma 2]{Ze1}, \cite[Chapter VII]{Hu}),
$M$ contains a ball that contains $R$.
\end{proof}

\begin{proof}[(b)] The if direction follows from (a), since the projection of
the image of $Z$ onto $S^{2n+1}$ is disjoint from $g(P)$, except at
$gf(\partial Z)$.

Conversely, let $N$ be the second derived neighborhood of $f(\partial Z)$ in
some triangulation $K$ of $P$.
Given a linkless embedding $P\subset S^{2n+1}$, let $B$ be a codimension
zero ball containing $f(\partial Z)$ and such that $B\cap P$ is contained in $N$.

{\it Step 1.} By a simple engulfing argument, we may assume that the intersection
of each simplex $\sigma$ of $K$ with the interior of $B$ is connected.
In more detail, if $C$ and $D$ are two components of this intersection, pick some
$c\in C$ and $d\in D$.
Since the intersection of $\sigma$ with the interior of $N_i$ is connected,
we can join $c$ and $d$ by an arc $J$ in that intersection.
We may assume that $J$ meets $\partial B$ in finitely many points;
arguing by induction, we may further assume that there are just two of
them, $c'$ and $d'$.
Let $J'$ be the segment of $J$ spanned by these points.
Let us also join $c'$ and $d'$ by an arc $J''$ in $\partial B$ that meets
$P$ only in its endpoints.
The $1$-sphere $J'\cup J''$ bounds a $2$-disk $D$ in the closure of the
complement to $B$ in $S^{2n+1}$ that meets $B\cup P$ only in $\partial D$.
Then a small regular neighborhood $\beta$ of $D$ in the closure of the complement
of $B$ is such that the ball $B\cup\beta$ still meets $P$ along a subpolyhedron
of $N$.
Moreover, replacing $B$ with $B\cup\beta$ decreases the number of components
in the intersection of a simplex $\tau$ of $K$ with the interior of $B\cup\beta$
when $\tau=\sigma$, and keeps it the same when $\tau\ne\sigma$.

{\it Step 2.}
Let us write $Z'=Z\cup_{f\x\phi}(f\x\phi)(\partial Z)$.
Pick a generic map $F\:Z'\to B\x I$ extending the embedding $g|_{\partial Z'}$.
Since $F$ is a codimension three map, and $Z'$ is a circuit, by
the Penrose--Whitehead--Zeeman trick it can be replaced by an embedding $G$
that agrees with $F$ on $\partial Z'$ (compare \cite{Sa-1} and \cite[\S8]{M2}).
In more detail, given a self-intersection $F(p)=F(q)$ of $F$, since $Z$ is a circuit,
$p$ and $q$ can be joined by an arc $J$ in $Z'$ that contains only generic points.
So a small regular neighborhood $R$ of $J$ in $Z'$ is a $2$-disk.
Now $F(J)$ bounds a $2$-disk $D$ in $S^{2n+1}\x I$ that meets $F(Z')$ only in
$\partial D$ and is disjoint from $g(P\x I)$.
A small regular neighborhood $S$ of $D$ in $S^{2n+1}\x I$ is a ball disjoint
from $g(P\x I)$ and we may assume that $F^{-1}(S)=R$.
Then we redefine $F$ on $S$ by the conewise extension $R\to S$ of the
boundary restriction $F|_{\partial R}\:\partial R\to\partial S$.

{\it Step 3.}
By a further application of the same trick, using that the intersection of
each $\sigma$ with the interior of $B$ is connected, we may further amend $G$
so that the image of the resulting embedding $G'$ meets $g(P\x I)$ only in
$g(\partial Z')$.
Indeed, given an intersection $F(p)=g(q)$ between $F$ and $g$, we may join
$q$ to some point $r\in\partial Z'$ by an arc $J$ in $P\x I$ going only through
generic points of $P\x I$ (except for $r$ itself) and such that $g(J)$ lies in
the interior of $B$.
We may then join $p$ and $r$ by an arc $J'$ in $Z'$ going only through
generic points of $Z'$ (except for $r$ itself).
Then a regular neighborhood of $J\cup J'$ is a cone, and the preceding
construction works.
\end{proof}

We write $CP$ to denote the cone $pt*P$ over the polyhedron $P$.

\begin{theorem}\label{commensuration} Let $P$ be an $n$-polyhedron, and let $Q$ be
an $(n-1)$-dimensional subpolyhedron of $P$ such that the closure of every
component of $P\but Q$ is an $n$-circuit with boundary.
In part (a), assume further that every pair of disjoint singular $(n-1)$-circuits
in $Q$ bounds disjoint singular $n$-circuits in $P$.

(a) $Q$ linklessly embeds in $S^{2n-1}$ iff $P\cup CQ$ embeds in $S^{2n}$.

(b) $P$ embeds in $S^{2n}$ iff $P\cup CQ$ linklessly embeds in $S^{2n+1}$.
\end{theorem}

The case $n=1$ of (b) and much of the case $n=2$ of (a) were proved by van der Holst
\cite{Ho}, which the author discovered after writing up the proof below.
The case $n=1$ in the ``only if'' assertion in (b) was also proved earlier in
\cite{RST'}.

\begin{proof}[Proof. (a), if] Given an embedding $P\cup CQ\subset S^{2n}$, let $B$
be a regular neighborhood of $CQ$ relative to $P$.
Since $CQ$ link-collapses onto $Q$, this $B$ is a manifold \cite{HZ} (cf.\
\cite{Hus}), and therefore a ball (see \cite{HZ} or \cite{Co2}).
Then $Q\subset\partial B$ is an embedding that is linkless (and knotless, if $n=2$)
by Lemma \ref{circuits}(a) and Lemma \ref{panels}(a).
\end{proof}

\begin{proof}[(a), only if]
Suppose we are given a linkless embedding $Q\subset S^{2n-1}$; if $n=2$, we
may further assume that it is panelled by Lemma \ref{panels}(a).
Let us extend it to the conical embedding of $CQ$ in $B^{2n}$ and to the vertical
embedding of $Q\x I$ into a collar $S^{2n-1}\x I$ of $B^{2n}$ in $S^{2n}$.
Let $M_1,\dots,M_r$ be the closures of the components of $P\but Q$.

If $n=2$, each $\partial M_i$ bounds an embedded disk $D$ in $S^3$ that meets $Q$
only in $\partial D$.
Then $Q\x\{\frac ir\}$ can be easily approximated by an embedded copy of $M_i$
lying in $S^3\x [\frac{i-1}r,\frac ir]$ and meeting $Q\x I$ only in
$\partial M_i$.

If $n\ge 3$, then Lemma \ref{circuits}(b) yields an embedding $g_i$ of each $M_i$
in $S^{2n-1}\x [\frac{i-1}r,\frac ir]$ extending the inclusion of
$\partial M_i$ onto $\partial M_i\x\{\frac ir\}$ and disjoint from $Q\x I$
elsewhere.

In either case, different $M_i$'s will be disjoint because of their heights
in $S^{2n+1}\x I$.
Let $\hat P=CQ\cup Q\x I\cup M_1\cup\dots\cup M_r$.
The projection $\pi\:Q\x I\to Q$ yields a collapsible map $S^4\to S^4\cup_\pi Q$.
By the Cohen--Homma Theorem \ref{Cohen-Homma}, $S^4\cup_\pi Q$ is homeomorphic
to $S^4$, and therefore $P=\hat P\cup_\pi Q$ embeds in $S^4$.
\end{proof}

\begin{proof}[Proof. (b), only if] Let us write $S^{2n+1}=\{\nu,\sigma\}*S^{2n}$.
Given an embedding $P\subset S^{2n}$, it extends to the conical embedding
$P\cup\nu*Q\subset \nu*P\subset\nu*S^{2n}\subset S^{2n+1}$.

If $n=1$, then the latter is panelled, since every circuit $Z$ of
$P\cup\nu*Q$ either lies in $P$ (and so bounds the disk $\sigma*Z$) or is of
the form $\nu*(\partial J)\cup J$, where $J$ is an arc in $P$ (and so bounds
the disk $\nu*J$).

Now suppose that $n\ge 2$ and let $N$ and $S$ be disjoint subpolyhedra of
$P\cup\nu*Q$.
Without loss of generality $\nu\in N$.
Then $H^n(S,\,S\cap P)\simeq H^n(S\cup P,\,P)=0$ due to
$H^n(P\cup\nu*Q\but\nu*\emptyset,\,P)=0$.
Let $\Sigma=\sigma*(S\cap P)$; then
$H^n(S\cup\Sigma)\simeq H^n(S\cup\Sigma,\Sigma)=0$.
By the Alexander duality, the open manifold $M\bydef S^{2n+1}\but (S\cup\Sigma)$ is
homologically $n$-connected.
But also it is simply-connected as the complement to the codimension $\ge 3$
subset $S\but\Sigma$ in the Euclidean space $S^{2n+1}\but\Sigma$.
Hence $M$ is $n$-connected, and so the $n$-polyhedron $N$, which has
codimension $\ge 3$ in $M$, can be engulfed into a ball in $M$ (see e.g.\
\cite{Ze1} or \cite[Chapter VII]{Hu}).
\end{proof}

\begin{proof}[(b), if]
If $n=1$, then the proof of Lemma \ref{panels}(b) in \cite{M2} works to extend
the given panelled embedding of $P\cup CQ$ to an embedding $CP\emb S^3$.
By considering the link of the cone vertex, we obtain an embedding of $P$ in $S^2$.

If $n\ge 2$, let $M_1,\dots,M_r$ be the closures of the components of $P\but Q$.
Pick a map $f_i\:M_i\to C(\partial M_i)$ that restricts to the identification
of $\partial M_i$ with the base of the cone.
The mapping cylinder $MC(f_i)$ contains
$\mu_i\bydef M_i\cup MC(f_i|_{\partial M_i})\cup C(\partial M_i)$, which is a copy of
$M_i\cup C(\partial M_i)$.
By Lemma \ref{circuits}(b), the natural embedding of $\mu_i$ in $MC(\id_B)$,
which is a copy of $B\x I$, extends to an embedding $G_i\:MC(f_i)\emb B\x I$
whose image meets $(P\cup CQ)\x I$ only in $\mu_i$.

Combining the embeddings $G_i$ together, we obtain a map $F$ of $MC(f)$ into
$S^{2n+1}\x I$, where $f\:P\cup CQ\to CQ$ is obtained by combining the $f_i$.
By construction, $F$ restricts to the natural embedding of
$\mu\bydef P\cup MC(f|_Q)\cup CQ$, which is a copy of $P\cup CQ$, in
$MC(\id_{S^{2n+1}})$, which is a copy of $S^{2n+1}\x I$.
The only double points of $F$ are isolated, and occur between $MC(f_i)$
and $MC(f_j)$ for $i\ne j$.
Since they all share the cone vertex in $\mu$, yet another application of
the Penrose--Whitehead--Zeeman trick (in addition to those in Lemma
\ref{circuits}(b)) enables one to replace $F$ by an embedding $G$ that agrees
with $F$ on $\mu$ and still meets $(P\cup CQ)\x I$ only in $\mu$.

Viewed as an embedding of one mapping cylinder in another, $G$ may be assumed to
be level-preserving near the target base --- in other words, at some interval
$[1-\eps,1]$ of the parameter values (cf.\ \cite[proof of 4.23]{RS}).
Hence $G$ restricts to a concordance (with parameter values in $[0,1-\eps]$)
keeping $Q$ fixed between the inclusion $P\subset S^{2n+1}$ and
an embedding of $P$ in the boundary of the second derived neighborhood $B$ of
$CQ$ modulo $Q$ in an appropriate triangulation of $S^{2n+1}$.
Since $CQ$ link-collapses onto $Q$, this $B$ is a manifold \cite{HZ} (cf.\
\cite{Hus}), and therefore a ball (see \cite{HZ} or \cite{Co2}).
Since the image of this concordance meets $(P\cup CQ)\x [0,1-\eps]$ only in
$P\x\{0\}\cup CQ\x [0,1-\eps]$, it may be viewed as a concordance of
the entire $P\cup CQ$ keeping $CQ$ fixed.
Then by the Concordance Implies Isotopy theorem we get an isotopy of $S^{2n+1}$
keeping $CQ$ fixed and taking $P$ into the $2n$-sphere $\partial B$.
\end{proof}

\subsection*{Appendix: Embedding 2-polyhedra in 4-sphere}

It is well-known that if $K$ is a $2$-dimensional cell complex such that $|K^{(1)}|$
embeds in $S^2$, then $|K|$ embeds in $S^4$ (cf.\ \cite[p.\ 44]{2DH}); see
\cite[Theorem 4]{Cu} for a related result.
The following is essentially proved in \cite{Ho} (see also \cite{Gi}).

\begin{corollary}
If $K$ is a $2$-dimensional cell complex such that $|K^{(1)}\but\cel v\cer|$
linklessly embeds in $S^3$ for some vertex $v$ of $K$, then $|K|$ embeds in $S^4$.
\end{corollary}

\begin{proof} Let $L=K\but\cel v\cer$.
Then $|K|$ embeds in $|L\cup L^{(1)}*\{v\}|$.
By the hypothesis $|L^{(1)}|$ linklessly embeds in $S^3$, hence by
Theorem \ref{commensuration}(a), $|L\cup L^{(1)}*\{v\}|$ embeds in $S^4$.
\end{proof}

It is a well-known open problem whether every contractible $2$-polyhedron
embeds (i.e., PL embeds) in $S^4$; an affirmative solution is well-known to be
implied by the Andrews--Curtis conjecture, cf.\ Curtis \cite[\S2]{Cu}.
(Indeed, by general position every $2$-polyhedron $P$ immerses in $I^4$, and
therefore embeds in a $4$-manifold $M$.
Let $N$ be the regular neighborhood of $P$ in $M$.
If $P$ $3$-deforms to a point, then the double of $N$ is the $4$-sphere, see
\cite[Assertion (59) in Ch.\ I]{2DH}.)

The most nontrivial ingredient of the proof of Theorem \ref{commensuration}
is Lemma \ref{panels}(a) of Robertson--Seymour--Thomas.
The dependence of part (b) on this lemma is only apparent: it can be eliminated
altogether if we replace ``linklessly'' by ``linklessly (and knotlessly,
when $n=1$)''.
Part (a) depends on the Robertson--Seymour--Thomas lemma in an essential way.
However, it does not use full strength of the lemma.

The remaining power of their lemma is captured by the following striking result,
which can also be deduced from the results of \cite{Ho}.

\begin{theorem}\label{Whitney} Let $P$ be a $2$-polyhedron and $Q$ a $1$-dimensional
subpolyhedron of $P$ such that the closure of every component of $P\but Q$ is a disk,
and every two disjoint circuits in $Q$ bound disjoint singular surfaces in $P$.

Then $P\cup CQ$ embeds in $S^4$ iff the $\bmod 2$ van Kampen obstruction of
$P\cup CQ$ vanishes.
\end{theorem}

This is saying basically that in a certain situation, the Whitney trick works
in dimension $4$ in the PL category.

\begin{proof}
Let us pick an embedding of $Q$ in $S^3=\partial B^4$, extend it to the conical
embedding $CQ\to B^4$ and also to a map of $P$ into the other hemisphere of $S^4$.
This defines a map $f\:P\cup CQ\to S^4$, and we may assume that it only has
transverse double points.

Let $\bar P$ be the quotient of $P\x P\but\Delta_P$ by the factor exchanging
involution $T$.
Let $G$ be a triangulation of $Q$, and let $K$ be the cell complex extending
the triangulation $CG$ of $CQ$ by adding the closures of the components of
$P\but Q$.
Let $\bar K\subset\bar P$ be the quotient by $T$ of the union of all products
$\sigma\x\tau$, where $\sigma$ and $\tau$ are disjoint cells of $K$.
Note that $H^4(\bar P;\,\Z/2)\simeq H^4(\bar K;\,\Z/2)$.
For disjoint cells $\sigma,\tau$ of $K$, let $\sigma\boxtimes\tau$ denote
the characteristic chain of the cell $(\sigma\x\tau\cup\tau\x\sigma)/T$ of
$\bar K$.
The van Kampen obstruction $\theta(P)\in H^4(\bar K;\,\Z/2)$ is represented by
a cellular cocycle $c$ such that $c(\sigma\boxtimes\tau)$ is the parity of
the number of intersections between $f(\sigma)$ and $f(\tau)$ (see \cite{M2}).

By the hypothesis, $c$ is the coboundary of a cellular $1$-chain $b$.
For each edge $\sigma$ of $K$ and each $2$-cell $\tau$ in $P$ disjoint from
$\sigma$ and such that $b(\sigma\boxtimes\tau)\ne 0$, let us pick a copy of
$S^2$ in a small neighborhood of $f(\sigma)$ in $S^4$, winding around
$f(\sigma)$ with an odd linking number, and connect this sphere to $f(\tau)$
by a thin tube disjoint from the image of $f$.
Next, for each edge $\sigma$ of $G$ and each $2$-simplex $C\tau$ of $CG$ disjoint
from $\sigma$ let us do an equivalent, but fancier procedure.
Let us pick a copy of $S^1$ in a small neighborhood of $f(\sigma)$ in $S^3$,
winding around $f(\sigma)$ with and odd linking number (it might be easier to
imagine the following steps if the linking number is $\pm 1$) and connect this
loop to $f(\tau)$ within $S^3$ by a thin tube $S^0\x I$ disjoint from the image
of $f$.
We extend this modification of $f(\tau)$ conewise to $f(C\tau)$ and by a
generic homotopy to a small neighborhood of $\tau$ mod $\partial\tau$ in $P$.
Note that $f(CQ)$ stays within $B^4$ and remains embedded.

The amended map $f'$ has the property that for every pair of disjoint cells
$\sigma$, $\tau$ of $K$, the intersection number between $f'(\sigma)$ and
$f'(\tau)$ is even.
In addition, $f'$ still embeds $CQ$ in $B^4$, conically.
Then all intersections between a $2$-cell $\sigma$ in $P$ and a $2$-simplex
$C\tau$ can be pushed through the base of the cone; since their $\bmod 2$
algebraic number is zero, this will not change the $\bmod 2$ algebraic
intersection numbers of $\sigma$ with other $2$-cells in $P$.
Thus all the intersections of the resulting map $f''$ are between $2$-cells of $P$,
which all lie in the upper hemisphere of $S^4$, and the intersection number
between any pair of disjoint $2$-cells is even.

Since $Q$ is still in $S^3$, we obtain that every two disjoint circuits in $Q$
have even linking number under $f''$.
Now the assertion follows from Lemma \ref{panels}(a) and Theorem
\ref{commensuration}(a).
\end{proof}

\begin{remark}\label{erratum2} The above argument is parallel to a proof of
the completeness of
the van Kampen obstruction in higher dimensions, see e.g.\ \cite[proof of 3.1]{M2}.
We note some confusing typos in the final paragraph of that proof in \cite{M2}:
``$f(p)=f(q)$'' should read ``$g(p)=g(q)$'', and more importantly ``Then a small
regular neighborhood of $f(J)$ ...'' should read ``Since $n>2$, the embedded
$1$-sphere $g(J)$ bounds an embedded $2$-disk $D$ that meets $g(Y\cup\sigma_i)$
only in its boundary. Then a small regular neighborhood of $D$ ...''
\end{remark}

\section*{Acknowledgements}
The author is grateful to A. N. Dranishnikov, I. Izmestiev, E. Nevo, E. V. Shchepin,
A. Skopenkov, M. Skopenkov, M. Tancer, S. Tarasov and M. Vyalyj for useful
discussions.


\begin{thebibliography}{00}

\bib{2DH}{book}{
title={Two-Dimensional Homotopy and Combinatorial Group Theory},
editor={C. Hog-Angeloni},
editor={W. Metzler},
editor={A. Sieradski},
series={London Math. Soc. Lect. Note Ser.},
volume={197},
date={1993},
}

\bib{AHS}{article}{
   author={Ad{\'a}mek, J.},
   author={Herrlich, H.},
   author={Strecker, G. E.},
   title={Abstract and Concrete Categories: The Joy of Cats},
   note={Reprint of the 1990 original (Wiley, New York)},
   journal={Repr. Theory Appl. Categ.},
   number={17},
   date={2006},
   pages={1--507; \href{http://katmat.math.uni-bremen.de/acc/acc.pdf}
   {\tt Online Edition}},
}

\bib{AdM}{article}{
   author={MacLane, S.},
   author={Adkisson, V. W.},
   title={Extensions of homeomorphisms on the sphere},
   conference={
      title={Lectures in Topology},
   },
   book={
      publisher={University of Michigan Press},
      place={Ann Arbor, Mich.},
   },
   date={1941},
   pages={223--236},
}

\bib{Ak}{article}{
   author={Akin, E.},
   title={Transverse cellular mappings of polyhedra},
   journal={Trans. Amer. Math. Soc.},
   volume={169},
   date={1972},
   pages={401--438},
   note={\href{http://www.ams.org/journals/tran/1972-169-00/S0002-9947-1972-0326745-7}
   {\tt Journal}},
}

\bib{ArM}{article}{
   author={Arnoux, P.},
   author={Marin, A.},
   title={The K\"uhnel triangulation of the complex projective plane from
   the view point of complex crystallography. II},
   journal={Mem. Fac. Sci. Kyushu Univ. Ser. A},
   volume={45},
   date={1991},
   number={2},
   pages={167--244},
   note={\href{http://www.jstage.jst.go.jp/article/kyushumfs/45/2/45_167}{\tt JStage}},
}

\bib{BBC}{article}{
   author={Babson, E. K.},
   author={Billera, L. J.},
   author={Chan, C. S.},
   title={Neighborly cubical spheres and a cubical lower bound conjecture},
   journal={Israel J. Math.},
   volume={102},
   date={1997},
   pages={297--315},
}

\bib{BD1}{article}{
   author={Bagchi, B.},
   author={Datta, B.},
   title={On K\"uhnel's $9$-vertex complex projective plane},
   journal={Geom. Dedicata},
   volume={50},
   date={1994},
   number={1},
   pages={1--13},
}

\bib{BD2}{article}{
   author={Bagchi, B.},
   author={Datta, B.},
   title={A short proof of the uniqueness of K\"uhnel's $9$-vertex complex
   projective plane},
   journal={Adv. Geom.},
   volume={1},
   date={2001},
   number={2},
   pages={157--163},
}

\bib{BD3}{article}{
   author={Bagchi, B.},
   author={Datta, B.},
   title={From the icosahedron to natural triangulations of $\C P^2$ and
   $S^2\x S^2$},
   journal={Discr. Comput. Geom.},
   status={published online on August 24, 2010},
   note={\arxiv{1004.3157}},
}

\bib{Bj}{article}{
   author    = {A. Bj\"orner},
   title     = {Posets, regular CW complexes and Bruhat order},
   journal   = {European J. Combin.},
   volume    = {5},
   year      = {1984},
   pages     = {7--16},
}

\bib{BPSZ}{article}{
   author={Bj{\"o}rner, A.},
   author={Paffenholz, A.},
   author={Sj{\"o}strand, J.},
   author={Ziegler, G. M.},
   title={Bier spheres and posets},
   journal={Discrete Comput. Geom.},
   volume={34},
   date={2005},
   number={1},
   pages={71--86},
   note={\arxiv[CO]{0311356}},
}

\bib{BK1}{article}{
   author={Brehm, U.},
   author={K{\"u}hnel, W.},
   title={Combinatorial manifolds with few vertices},
   journal={Topology},
   volume={26},
   date={1987},
   number={4},
   pages={465--473},
}

\bib{BK2}{article}{
   author={Brehm, U.},
   author={K{\"u}hnel, W.},
   title={$15$-vertex triangulations of an $8$-manifold},
   journal={Math. Ann.},
   volume={294},
   date={1992},
   number={1},
   pages={167--193},
   note={\href{http://resolver.sub.uni-goettingen.de/purl?GDZPPN002337266}{\tt GDZ}}
}

\bib{Bry}{article}{
author={J. L. Bryant},
title={Triangulation and general position of PL diagrams},
journal={Topol. Appl.},
volume={34},
date={1990},
pages={211--233},
}

\bib{BRS}{book}{
   author={Buoncristiano, S.},
   author={Rourke, C. P.},
   author={Sanderson, B. J.},
   title={A geometric approach to homology theory},
   note={London Math. Soc. Lecture Note Series},
   volume={18},
   publisher={Cambridge Univ. Press},
   place={Cambridge},
   date={1976},
}

\bib{CL}{article}{
   author={Chernavsky, A. V.},
   author={Leksine, V. P.},
   title={Unrecognizability of manifolds},
   journal={Ann. Pure Appl. Logic},
   volume={141},
   date={2006},
   number={3},
   pages={325--335},
}

\bib{Co}{article}{
   author={Cohen, M. M.},
   title={Simplicial structures and transverse cellularity},
   journal={Ann. Math.},
   volume={85},
   date={1967},
   pages={218--245},
}

\bib{Co2}{article}{
   author={Cohen, M. M.},
   title={A general theory of relative regular neighborhoods},
   journal={Trans. Amer. Math. Soc.},
   volume={136},
   date={1969},
   pages={189--229},
   note={\href{http://www.ams.org/journals/tran/1969-136-00/S0002-9947-1969-0248802-6}
   {\tt Journal}},
}

\bib{CF}{article}{
   author={Conner, P. E.},
   author={Floyd, E. E.},
   title={Fixed point free involutions and equivariant maps},
   journal={Bull. Amer. Math. Soc.},
   volume={66},
   date={1960},
   pages={416--441},
   note={\href{http://projecteuclid.org/euclid.bams/1183523747}{\tt ProjectEuclid}},
}

\bib{CD}{article}{
   author={{\v{C}}uki{\'c}, S. Lj.},
   author={Delucchi, E.},
   title={Simplicial shellable spheres via combinatorial blowups},
   journal={Proc. Amer. Math. Soc.},
   volume={135},
   date={2007},
   number={8},
   pages={2403--2414},
   note={\arxiv[CO]{0602101}},
}

\bib{Cu}{article}{
   author={Curtis, M. L.},
   title={On $2$-complexes in $4$-space},
   conference={
      title={Topology of $3$-Manifolds and Related Topics (Proc. Univ. of Georgia
      Institute, 1961)},
   },
   book={
      publisher={Prentice-Hall},
      place={Englewood Cliffs, N.J.},
   },
   date={1962},
   pages={204--207},
}

\bib{Da}{article}{
   author={Datta, B.},
   title={Minimal triangulation, complementarity and projective planes},
   conference={
      title={Geometry from the Pacific Rim},
      address={Singapore},
      date={1994},
   },
   book={
      publisher={de Gruyter},
      place={Berlin},
   },
   date={1997},
   pages={77--84},
}

\bib{Dav}{book}{
   author={Daverman, R. J.},
   title={Decompositions of Manifolds},
   series={Pure Appl. Math.},
   volume={124},
   publisher={Academic Press},
   place={Orlando, FL},
   date={1986},
   reprint={
   publisher={AMS Chelsea Publishing},
   place={Providence, RI},
   date={2007},
   },
}

\bib{DV}{book}{
   author={Daverman, R. J.},
   author={Venema, G. A.},
   title={Embeddings in Manifolds},
   series={Grad. Studies in Math.},
   volume={106},
   publisher={Amer. Math. Soc.},
   place={Providence, RI},
   date={2009},
   note={\href{http://www.calvin.edu/~venema/embeddingsbook/embeddings-noprint.pdf}
   {\tt Homepage}},
}

\bib{dL}{article}{
   author={de Longueville, M.},
   title={Bier spheres and barycentric subdivision},
   journal={J. Combin. Theory Ser. A},
   volume={105},
   date={2004},
   number={2},
   pages={355--357},
   note={\href{http://www.math.tu-berlin.de/~longue/bier_note.ps.gz}{\tt Homepage}},
}

\bib{Di}{book}{
   author={Diestel, R.},
   title={Graph theory},
   series={Grad. Texts in Math.},
   volume={173},
   edition={4},
   publisher={Springer},
   place={Heidelberg},
   date={2010},
   note={\href{http://diestel-graph-theory.com/basic.html}{\tt Electronic edition
   (abridged)}},
}

\bib{Do}{article}{
   author={Dominguez, E.},
   title={Interpr\'etation g\'eom\'etrique de l'homologie singuli\`ere \`a
   coefficients},
   language={French},
   journal={Rev. Roumaine Math. Pures Appl.},
   volume={23},
   date={1978},
   number={6},
   pages={857--868},
}

\bib{DM}{article}{
   author={Dragotti, S.},
   author={Magro, G.},
   title={${\scr L}$-manifolds and cone-dual maps},
   journal={Proc. Amer. Math. Soc.},
   volume={103},
   date={1988},
   number={4},
   pages={1281--1289},
   note={\href{http://www.ams.org/journals/proc/1988-103-04/S0002-9939-1988-0955023-5}
   {\tt Journal}},
}

\bib{Fe}{book}{
   author={Fenn, R. A.},
   title={Techniques of Geometric Topology},
   series={London Math. Soc. Lecture Note Series},
   volume={57},
   publisher={Cambridge Univ. Press},
   date={1983},
}

\bib{F1}{article}{
author={A. Flores},
title={\"Uber die Existenz $n$-dimensionaler Komplexe, die nicht in den
$R_{2n}$ topologisch einbettbar sind},
journal={Ergebnisse math. Kolloquium Wien},
date={1933},
volume={5},
pages={17--24},
}

\bib{F2}{article}{
author={A. Flores},
title={\"Uber $n$-dimensionale Komplexe die im $R_{2n+1}$ absolut
selbstverschlungen sind},
journal={Ergebnisse math. Kolloquium Wien},
date={1935},
volume={6},
pages={4--7},
}

\bib{FS}{article}{
author={O. Frink},
author={P. A. Smith},
title={Irreducible non-planar graphs},
journal={Bull. Amer. Math. Soc.},
volume={36},
number={3},
year={1930},
pages={214},
note={\href{http://www.ams.org/journals/bull/1930-36-03/S0002-9904-1930-04929-0}
{Abstract 36-3-179}},
}

\bib{Gi}{article}{
   author={Gillman, D.},
   title={Generalising Kuratowski's theorem from ${\bf R}^2$ to ${\bf R}^4$},
   booktitle={Proceedings of the Singapore conference on combinatorial
   mathematics and computing (Singapore, 1986)},
   journal={Ars Combin.},
   volume={23},
   date={1987},
   number={A},
   pages={135--140},
}

\bib{Gr}{article}{
   author={Gr{\"u}nbaum, B.},
   title={Imbeddings of simplicial complexes},
   journal={Comment. Math. Helv.},
   volume={44},
   date={1969},
   pages={502--513},
   note={\href{http://resolver.sub.uni-goettingen.de/purl?GDZPPN002059746}{\tt GDZ}},
}

\bib{HJ}{article}{
   author={Halin, R.},
   author={Jung, H. A.},
   title={Charakterisierung der Komplexe der Ebene und der $2$-Sph\"are},
   journal={Arch. Math.},
   volume={15},
   date={1964},
   pages={466--469},
}

\bib{Hat}{article}{
   author={Hatcher, A. E.},
   title={Higher simple homotopy theory},
   journal={Ann. of Math.},
   volume={102},
   date={1975},
   pages={101--137},
}

\bib{Ho}{article}{
   author={van der Holst, H.},
   title={Graphs and obstructions in four dimensions},
   journal={J. Combin. Theory Ser. B},
   volume={96},
   date={2006},
   number={3},
   pages={388--404},
}

\bib{Hu}{book}{
   author={Hudson, J. F. P.},
   title={Piecewise Linear Topology},
   series={University of Chicago Lecture Notes (prepared with the assistance
   of J. L. Shaneson and J. Lees)},
   publisher={W. A. Benjamin, Inc.},
   place={New York--Amsterdam},
   date={1969},
   note={\href{http://www.maths.ed.ac.uk/~aar/surgery/hudson.pdf}{\tt Andrew Ranicki's
   archive}},
}

\bib{HZ}{article}{
author={Hudson, J. F. P.},
author={E. C. Zeeman},
title={On regular neighborhoods},
journal={Proc. London Math. Soc.},
volume={14},
date={1964},
pages={719--745},
}

\bib{Hus}{article}{
author={L. Husch},
title={On relative regular neighborhoods},
journal={Proc. London Math. Soc.},
volume={19},
year={1969},
pages={577--585},
}

\bib{Iz}{article}{
   author={Izmestiev, I.},
   title={The Colin de Verdi\`ere number and graphs of polytopes},
   journal={Israel J. Math.},
   volume={178},
   date={2010},
   pages={427--444},
   note={\arxiv{0704.0349}},
}

\bib{vK}{article}{
author={E. R. van Kampen},
title={Komplexe in euklidischen R\"aumen},
journal={Abh. Math. Sem. Univ. Hamburg},
volume={9},
date={1932},
pages={72--78, 152--153},
}

\bib{vK2}{thesis}{
author={E. R. van Kampen},
title={Die kombinatorische Topologie und die Dualit\"atss\"atze},
language={German, Dutch}, date={Dissertation: Den Haag, 1929},
note={\href{http://www.zentralblatt-math.org/zmath/en/advanced/?q=an:55.0964.01}{Zentralblatt}
(see also I. M. James' ``History of Topology'', p. 54)},
}

\bib{Ku}{article}{
author={K. Kuratowski},
title={On the problem of skew curves in topology},
journal={Fund. Math.},
volume={15},
date={1930},
pages={271--283},
language={French}
translation={
note={Graph Theory (\L ag\'ow, 1981), Lecture Notes in Math., vol. 1018,
Springer, Berlin, 1983, pp. 1--13},
},
}


\bib{KB}{article}{
   author={K{\"u}hnel, W.},
   author={Banchoff, T. F.},
   title={The $9$-vertex complex projective plane},
   journal={Math. Intelligencer},
   volume={5},
   date={1983},
   number={3},
   pages={11--22},
}

\bib{La2}{article}{
   author={Lawson, T.},
   title={Splitting spheres as codimension $r$ doubles},
   journal={Houston J. Math.},
   volume={8},
   date={1982},
   number={2},
   pages={205--220},
}

\bib{La3}{article}{
   author={Lawson, T.},
   title={Detecting the standard embedding of $\R P^2$ in $S^4$},
   journal={Math. Ann.},
   volume={267},
   date={1984},
   number={4},
   pages={439--448},
   note={\href{http://resolver.sub.uni-goettingen.de/purl?GDZPPN002325136}{\tt GDZ}},
}

\bib{LS}{article}{
   author={Lov{\'a}sz, L.},
   author={Schrijver, A.},
   title={A Borsuk theorem for antipodal links and a spectral
   characterization of linklessly embeddable graphs},
   journal={Proc. Amer. Math. Soc.},
   volume={126},
   date={1998},
   number={5},
   pages={1275--1285},
   note={\href{http://www.ams.org/journals/proc/1998-126-05/S0002-9939-98-04244-0}
   {\tt Journal}},
}

\bib{MS}{article}{
   author={Madahar, K. V.},
   author={Sarkaria, K. S.},
   title={A minimal triangulation of the Hopf map and its application},
   journal={Geom. Dedicata},
   volume={82},
   date={2000},
   number={1-3},
   pages={105--114},
}

\bib{Mak}{article}{
   author={Makarychev, Yu.},
   title={A short proof of Kuratowski's graph planarity criterion},
   journal={J. Graph Theory},
   volume={25},
   date={1997},
   number={2},
   pages={129--131},
   note={\href{http://ttic.uchicago.edu/~yury/papers/kuratowski.pdf}{\tt Homepage}},
}

\bib{Mat}{book}{
   author={Matou{\v{s}}ek, J.},
   title={Using the Borsuk-Ulam Theorem},
   series={Universitext},
   subtitle={Lectures on topological methods in combinatorics and geometry
   (written in cooperation with Anders Bj\"orner and G\"unter M. Ziegler)},
   publisher={Springer-Verlag},
   place={Berlin},
   date={2003},
}

\bib{MTW}{article}{
author={Matou{\v{s}}ek, J.},
author={M. Tancer},
author={U. Wagner},
title={Hardness of embedding simplicial complexes in $\R^d$},
journal={J. Eur. Math. Soc.},
volume={13},
date={2011},
pages={259-295},
note={\arxiv{0807.0336}},
}

\bib{Mc}{article}{
   author    = {C. McCrory},
   title     = {Cone complexes and PL transversality},
   journal   = {Trans. Amer. Math. Soc.},
   volume    = {207},
   year      = {1975},
   pages     = {269--291},
}

\bib{McS}{book}{
   author={McMullen, P.},
   author={Schulte, E.},
   title={Abstract regular polytopes},
   series={Encycl. Math. Appl.},
   volume={92},
   publisher={Cambridge Univ. Press},
   date={2002},
}

\bib{M2}{article}{
   author={Melikhov, S. A.},
   title={The van Kampen obstruction and its relatives},
   journal={Tr. Mat. Inst. Steklova},
   volume={266},
   date={2009},
   pages={149--183; \arxiv[GT]{0612082}},
   reprint={
   note={Proc. Steklov Inst. Math. {\bf 266} (2009), 142--176}
   },
}

\bib{M3}{article}{
author={Melikhov, S. A.},
title={Combinatorics of combinatorial topology},
status={in preparation},
}

\bib{Mn}{article}{
   author    = {N. Mn\"ev},
   title     = {Combinatorial fiber bundles and fragmentation of
                fiberwise pl-homeomorphism},
   journal   = {Zap. Nauch. Sem. POMI},
   volume    = {344},
   year      = {2007},
   pages     = {56--173; \href{http://mi.mathnet.ru/znsl104}{\tt MathNet}},
   translation = {
    journal  = {J. Math. Sci.},
    volume   = {147},
    pages    = {7155--7217},
    note={\arxiv{0708.4039}},
    year     = {2007},
   },
}

\bib{MY}{article}{
   author={Morin, B.},
   author={Yoshida, M.},
   title={The K\"uhnel triangulation of the complex projective plane from
   the view point of complex crystallography. I},
   journal={Mem. Fac. Sci. Kyushu Univ. Ser. A},
   volume={45},
   date={1991},
   number={1},
   pages={55--142},
   note={\href{http://www.jstage.jst.go.jp/article/kyushumfs/45/1/45_55}{\tt JStage}},
}

\bib{Ne}{article}{
   author={Nevo, E.},
   title={Higher minors and Van Kampen's obstruction},
   journal={Math. Scand.},
   volume={101},
   date={2007},
   number={2},
   pages={161--176},
   note={\arxiv[CO]{0602531v2}},
}

\bib{Re}{article}{
   author={Reading, N.},
   title={The cd-index of Bruhat intervals},
   journal={Electron. J. Combin.},
   volume={11},
   date={2004},
   number={1},
   pages={Research Paper 74, 25pp.},
   note={\href{http://www.combinatorics.org/Volume_11/PDF/v11i1r74.pdf}{\tt Journal}},
}

\bib{RST}{article}{
   author={Robertson, N.},
   author={Seymour, P.},
   author={Thomas, R.},
   title={Sachs' linkless embedding conjecture},
   journal={J. Combin. Theory Ser. B},
   volume={64},
   date={1995},
   number={2},
   pages={185--227},
   note={\href{http://people.math.gatech.edu/~thomas/PAP/sachslinkess.pdf}
   {\tt Thomas' homepage}},
}

\bib{RST'}{article}{
   author={Robertson, N.},
   author={Seymour, P.},
   author={Thomas, R.},
   title={A survey of linkless embeddings},
   conference={
      title={Graph Structure Theory},
      address={Seattle, WA},
      date={1991},
   },
   book={
      series={Contemp. Math.},
      volume={147},
      publisher={Amer. Math. Soc.},
      place={Providence, RI},
   },
   date={1993},
   pages={125--136},
   note={\href{http://people.math.gatech.edu/~thomas/PAP/linklsurvey.pdf}
   {\tt Thomas' homepage}},
}

\bib{Ro}{article}{
   author={Rosen, R. H.},
   title={Decomposing $3$-space into circles and points},
   journal={Proc. Amer. Math. Soc.},
   volume={11},
   date={1960},
   pages={918--928},
   note={\href{http://www.ams.org/journals/proc/1960-011-06/S0002-9939-1960-0120611-5}
   {\tt Journal}},
}

\bib{RS}{book}{
   author={Rourke, C. P.},
   author={Sanderson, B. J.},
   title={Introduction to Piecewise-Linear Topology},
   series={Ergebn. der Math.},
   volume={69},
   publisher={Springer-Verlag},
   place={New York},
   date={1972},
}

\bib{Sac1}{article}{
   author={Sachs, H.},
   title={On a spatial analogue of Kuratowski's theorem on planar
   graphs---an open problem},
   conference={
      title={Graph Theory},
      address={\L ag\'ow},
      date={1981},
   },
   book={
      series={Lecture Notes in Math.},
      volume={1018},
      publisher={Springer},
      place={Berlin},
   },
   date={1983},
   pages={230--241},
}

\bib{Sac2}{article}{
   author={Sachs, H.},
   title={On spatial representations of finite graphs},
   conference={
      title={Finite and Infinite Sets, Vol.\ I, II},
      address={Eger},
      date={1981},
   },
   book={
      series={Colloq. Math. Soc. J\'anos Bolyai},
      volume={37},
      publisher={North-Holland},
      place={Amsterdam},
   },
   date={1984},
   pages={649--662},
}

\bib{Sa-1}{article}{
   author={Sarkaria, K. S.},
   title={Embedding and unknotting of some polyhedra},
   journal={Proc. Amer. Math. Soc.},
   volume={100},
   date={1987},
   number={1},
   pages={201--203},
   note={\href{http://www.ams.org/journals/proc/1987-100-01/S0002-9939-1987-0883429-0/}
   {\tt Journal}},
}

\bib{Sa0}{article}{
   author={Sarkaria, K. S.},
   title={Unknotting and colouring of polyhedra},
   journal={Bull. Polish Acad. Sci. Math.},
   volume={36},
   date={1988},
   number={11-12},
   pages={661--664},
   note={\href{http://kssarkaria.org/docs/Unknotting and colouring of polyhedra.pdf}
   {\tt Homepage}},
}

\bib{Sa1}{article}{
   author={Sarkaria, K. S.},
   title={Kuratowski complexes},
   journal={Topology},
   volume={30},
   date={1991},
   number={1},
   pages={67--76},
   note={\href{http://kssarkaria.org/docs/Kuratowski.pdf}{\tt Homepage}},
}

\bib{Sa3}{article}{
   author={Sarkaria, K. S.},
   title={A one-dimensional Whitney trick and Kuratowski's graph planarity
   criterion},
   journal={Israel J. Math.},
   volume={73},
   date={1991},
   number={1},
   pages={79--89},
   note={\href{http://kssarkaria.org/docs/One-dimensional.pdf}{\tt Homepage}},
}

\bib{Sch}{article}{
   author={Schild, G.},
   title={Some minimal nonembeddable complexes},
   journal={Topology Appl.},
   volume={53},
   date={1993},
   number={2},
   pages={177--185},
}

\bib{Sk-A1}{article}{
author={A. B. Skopenkov},
title={Around the Kuratowski graph planarity criterion},
partial={
 journal={Mat. Prosveshchenije},
 volume={9},
 date={2005},
 pages={116--118},
},
partial={
 part={Erratum},
 volume={10}
 date={2006},
 pages={276--277},
},
partial={
 part={Addendum (with A. S. Telishev)},
 volume={11},
 date={2007},
 pages={159--160},
},
note={(MCCME, Moscow; in Russian).
Updated version with an abridged English translation: \arxiv{0802.3820}},
}

\bib{Sk-A2}{article}{
author={A. B. Skopenkov},
title={Embeddings into the plane of graphs with vertices of degree $4$},
note={\arxiv{1008.4940}},
}

\bib{Sk-A3}{article}{
   author={A. B. Skopenkov},
   title={A classification of smooth embeddings of 4-manifolds in 7-space,
   I},
   journal={Topology Appl.},
   volume={157},
   date={2010},
   number={13},
   pages={2094--2110},
   note={\arxiv[GT]{0512594}},
}

\bib{Sk-M}{article}{
   author={M. Skopenkov},
   title={On approximability by embeddings of cycles in the plane},
   journal={Topology Appl.},
   volume={134},
   date={2003},
   number={1},
   pages={1--22},
   note={\arxiv{0808.1187}},
}
\bib{Sk-M2}{article}{
   author={M. Skopenkov},
   title={Embedding products of graphs into Euclidean spaces},
   journal={Fund. Math.},
   volume={179},
   date={2003},
   number={3},
   pages={191--198},
   note={\arxiv{0808.1199}},
}

\bib{Sm}{article}{
   author={Smale, S.},
   title={A Vietoris mapping theorem for homotopy},
   journal={Proc. Amer. Math. Soc.},
   volume={8},
   date={1957},
   pages={604--610},
   note={\href{http://www.ams.org/journals/proc/1957-008-03/S0002-9939-1957-0087106-9/
   }{\tt Journal}},
}

\bib{Su}{article}{
author={Sulanke, T.},
title={Examples of pseudo-minimal triangulations},
note={\href{http://hep.physics.indiana.edu/~tsulanke/graphs/pseudomin/examples.pdf}
{\tt Homepage}},
}

\bib{Ta}{article}{
   author={Taniyama, K.},
   title={Higher dimensional links in a simplicial complex embedded in a
   sphere},
   journal={Pacific J. Math.},
   volume={194},
   date={2000},
   number={2},
   pages={465--467},
   note={\href{http://pjm.berkeley.edu/pjm/2000/194-2/p14.xhtml}{\tt Journal}},
}

\bib{tD}{book}{
   author={tom Dieck, T.},
   title={Algebraic topology},
   series={EMS Textbooks in Math.},
   publisher={European Math. Soc.},
   place={Z\"urich},
   date={2008},
}

\bib{Um1}{article}{
   author={Ummel, B. R.},
   title={Some examples relating the deleted product to imbeddability},
   journal={Proc. Amer. Math. Soc.},
   volume={31},
   date={1972},
   pages={307--311},
   note={\href{http://www.ams.org/journals/proc/1972-031-01/S0002-9939-1972-0290349-0}
   {\tt Journal}},
}

\bib{Um2}{article}{
   author={Ummel, B. R.},
   title={Imbedding classes and $n$-minimal complexes},
   journal={Proc. Amer. Math. Soc.},
   volume={38},
   date={1973},
   pages={201--206},
   note={\href{http://www.ams.org/journals/proc/1973-038-01/S0002-9939-1973-0317336-9}
   {\tt Journal}},
}

\bib{VS}{article}{
   author={Volovikov, A. Yu.},
   author={Shchepin, E. V.},
   title={Antipodes and embeddings},
   journal={Mat. Sb.},
   volume={196},
   date={2005},
   number={1},
   pages={3--32, \href{http://mi.mathnet.ru/eng/msb1259}{MathNet}},
   translation={
      journal={Sb. Math.},
      volume={196},
      date={2005},
      number={1-2},
      pages={1--28},
      issn={1064-5616},
   },
}

\bib{Wa}{article}{
   author={Wagner, K.},
   title={\"Uber eine Eigenschaft der ebenen Komplexe},
   journal={Math. Ann.},
   volume={114},
   date={1937},
   number={1},
   pages={570--590},
   note={\href{http://resolver.sub.uni-goettingen.de/purl?GDZPPN002279355}{\tt GDZ}},
}

\bib{We}{article}{
   author={Weber, C.},
   title={Plongements de polyh\`edres dans le domaine m\'etastable},
   language={French},
   journal={Comment. Math. Helv.},
   volume={42},
   date={1967},
   pages={1--27},
   note={\href{http://resolver.sub.uni-goettingen.de/purl?GDZPPN002058839}{\tt GDZ}},
}

\bib{Wh}{article}{
author={H. Whitney},
title={Planar graphs},
journal={Fund. Math.},
volume={21},
year={1933},
pages={73--84},
note={\href{http://matwbn.icm.edu.pl/ksiazki/fm/fm21/fm21112.pdf}{Kolekcja Matematyczna}},
}

\bib{Za}{article}{
   author={Zaks, J.},
   title={On minimal complexes},
   journal={Pacific J. Math.},
   volume={28},
   date={1969},
   pages={721--727},
   note={\href{http://projecteuclid.org/euclid.pjm/1102983325}{\tt ProjectEuclid}},
}

\bib{Ze1}{article}{
   author={Zeeman, E. C.},
   title={The Poincar\'e conjecture for $n\geq 5$},
   conference={
      title={Topology of 3-Manifolds and Related Topics (Proc. Univ. of
      Georgia Institute, 1961)},
   },
   book={
      publisher={Prentice-Hall},
      place={Englewood Cliffs, N.J.},
   },
   date={1962},
   pages={198--204},
}

\bib{Zi}{book}{
   author={Ziegler, G. M.},
   title={Lectures on polytopes},
   series={Graduate Texts in Math.},
   volume={152},
   publisher={Springer-Verlag},
   place={New York},
   date={1995},
}

\bib{Zo}{article}{
author={A. Zomorodian},
title={Survey of results on minimal triangulations},
date={1998},
note={\href{http://www.cs.dartmouth.edu/~afra/goodies/minimal.pdf}{\tt Homepage}},
}
\end{thebibliography}
\end{document}